\documentclass[11pt,reqno]{amsart}

\usepackage{commandearticle_new}
\usepackage{pgf,tikz}
\usetikzlibrary{arrows}
\usepackage[francais]{layout}

\makeatletter
\def\subsection{\@startsection{subsection}{2}
	\z@{.5\linespacing\@plus.7\linespacing}{.25\linespacing}%
	{\normalfont\bfseries}}
\def\subsubsection{\@startsection{subsubsection}{3}%
	\z@{.5\linespacing\@plus.7\linespacing}{.25\linespacing}%
	{\normalfont\itshape}}
\makeatother
\graphicspath{{images/}}

\newcommand{\grandtraittop}{\rule{\textwidth}{0.25em}\par\vskip0.5\baselineskip}
\newcommand{\grandtraitbottom}{\par\vskip0.5\baselineskip\rule{\textwidth}{0.25em}}

\everymath{\displaystyle}
\begin{document}
\title[Stability of a relative velocity lattice Boltzmann scheme]{Stability of a bidimensional\\ relative velocity lattice Boltzmann scheme}
\author{François Dubois}
\address[François Dubois]{Laboratoire des Structures et des Syst\`emes Coupl\'es,
Conservatoire National des Arts et M\'etiers, Paris -- d\'epartement de math\'ematiques, Univ Paris Sud, Laboratoire de math\'ematiques, UMR 8628, Orsay, F-91405} 
\email{Francois.Dubois@math.u-psud.fr}

\author{Tony Février}
\address[Tony Février]{Univ Paris Sud, Laboratoire de math\'ematiques, UMR 8628, Orsay, F-91405, Orsay, F-91405}
\email{Tony.Février@math.u-psud.fr}

\author{Benjamin Graille}
\address[Benjamin Graille]{Univ Paris Sud, Laboratoire de math\'ematiques, UMR 8628, Orsay, F-91405, CNRS, Orsay, F-91405}
\email{Benjamin.Graille@math.u-psud.fr}

\begin{abstract}
In this contribution, we study the theoretical and numerical stability of a bidimensional relative velocity lattice Boltzmann scheme. These relative velocity schemes introduce a velocity field parameter called ``relative velocity'' function of space and time. They generalize the d'Humières multiple relaxation times scheme and the cascaded automaton. This contribution studies the stability of a four velocities scheme applied  to a single linear advection equation according to the value of this relative velocity. We especially compare when it is equal to $0$ (multiple relaxation times scheme) or to the advection velocity (``cascaded like'' scheme). The comparison is made in terms of $\Li$ and $L^2$ stability. The $\Li$ stability area is fully described in terms of relaxation parameters and advection velocity for the two choices of relative velocity. These results establish that no hierarchy of these two choices exists for the $\Li$ notion. Instead, choosing the parameter equal to the advection velocity improves the numerical $L^2$ stability of the scheme. This choice cancels some dispersive terms and improve the numerical stability on a representative test case. We theoretically strengthen these results with a weighted $L^2$ notion of stability.
\end{abstract}

\date{\today}
\maketitle

\section*{Introduction}

Lattice Boltzmann schemes are applicable to many different fields such as hydrodynamics, acoustics, magnetohydrodynamics and multiphase fluids \cite{dHu:1992:0,LalLuo:2000:0,Ricot:2009:1,Dellar:2002:0,He:1999:0}, because the associated  algorithm is simple, fast and flexible. This algorithm, associated with a cartesian lattice and a finite set of velocities, consists in computing some particle distribution functions at discrete values of time: they are updated through two successive phases of transport (exact) and relaxation (also called collision).

 In the d'Humières multiple relaxation times (MRT) framework \cite{dHu:1992:0}, the relaxation is made thanks to a set of moments, linear combinations of the particle distributions. Each of these moments relaxes towards an equilibrium with a proper time scale.
The MRT schemes have been widely studied in terms of consistency \cite{Dub:2008:0,Dub:2009:0,Junk:2005:0} but slightly in terms of stability \cite{LalLuo:2000:0,Ginz:2010:0}. They can be particularly used to solve the Navier-Stokes equations \cite{Qian:1992:0,Chen:1998:0,dHu:1992:0,LalLuo:2000:0,Geier:2006:0}. They however undergo some instabilities in the small viscosity limit \cite{LalLuo:2000:0}.

A ``cascaded automaton'' \cite{Geier:2006:1}, introduced in two and three dimensions for the Navier-Stokes equations, reduces these instabilities \cite{Geier:2006:0}. Some studies aim to understand the algorithm of this automaton and its improvements in terms of stability. First, it has been written in the MRT framework with a generalized equilibrium depending on the non conserved quantities \cite{Asi:2008:0}. It has then been included in a new class of relative velocity schemes \cite{Fev:2014:0,Fev:2014:1} that are the center of interest of the authors. Both MRT and cascaded are particular cases of the relative velocity schemes.

These relative velocity schemes use a set of moments depending on a velocity field parameter called relative velocity for the relaxation. This parameter is a function of space and time. Their consistency has been studied for one and two conservation laws \cite{Fev:2014:0,Fev:2014:1}. Their stability has also been investigated numerically in the case of the $\ddqn$ scheme for the Navier-Stokes equations \cite{Fev:2014:3}. This study, using the $L^2$ von Neumann stability \cite{LalLuo:2000:0,Sterling:1996:0}, shows the importance of the choice of the polynomials defining the moments, that justifies the stability gain obtained by the cascaded automaton.

This contribution studies the theoretical and the numerical stability of a bidimensional scheme with four velocities for the simulation of a single linear advection equation. We expect to improve the stability for a relative velocity equal to the advection velocity (``cascaded like'' scheme) instead of $0$ (MRT scheme). This study is based on two stability notions. The first is a $\Li$ notion that has been used for one dimensional and vectorial schemes \cite{Dell:2013:0,Gra:2014:0}. The second is a weighted $L^2$ notion introduced in \cite{Yong:2006:0}. It has been applied to MRT schemes linearized around the zero velocity or in the BGK (Bhatnagar, Gross, Krook) case \cite{BGK:1954:0} for advection equations \cite{Yong:2006:0,Junk:2009:0,Rhein:2010:0}.

In the first part, we briefly recall the framework of the relative velocity schemes. Then we present the four velocities scheme of interest. In the second part, the effect of the relative velocity on the stability is illustrated thanks to the method of the equivalent equations \cite{War:1974:0,Dub:2008:0}. We particularly focus on the structure of the dispersion terms for two choices of relative velocity ($0$ or the advection velocity). These results are then illustrated by a numerical $L^2$ von Neumann stability study. The two last parts give theoretical stability results. In the third part, the $\Li$ stability area is fully described in terms of relaxation parameters and advection velocity for the two choices of relative velocity. In the fourth one, a weighted $L^2$ notion \cite{Junk:2009:0} is used to validate the numerical results obtained in the second part according to the relative velocity.  

\section{Framework}\label{se:pst}

This section presents first the relative velocity schemes in a general framework. It is then particularized to a bidimensional scheme with four velocities for the simulation of a linear advection equation.

\subsection{Presentation of the relative velocity schemes}

For $d\in\N^{\star}$, we consider a cartesian lattice $\lattice$ of $\R^d$ associated with a space step $\dx$. The space time is given by the acoustic scaling $\dt=\dx/\lambda$ for $\lambda\in\R$ a velocity scale. A set of $q\in\N^{\star}$ velocities denoted by $\vectv=\{\vj[0],\ldots,\vj[q-1]\}$ depending on the velocity scale is chosen such that for each node $\vectx\in\lattice$, $\vectx+\vj\dt$ is still a node of $\lattice$. The lattice Boltzmann method computes the discrete values of several particle distributions denoted by $\fj$, $0\leq j\leq q-1$, the distribution $\fj$ being associated with the velocity $\vj$. We denote by $\vectf=(\fj[0],\ldots,\fj[q-1])$ the vector of the particle distributions. An iteration of the algorithm consists in the succession of a phase of relaxation, nonlinear and local in space, and a phase of transport distributing the particles on the neighbouring nodes.

The relaxation operator of the relative velocity schemes originates from the framework of the MRT schemes \cite{dHu:1992:0}. Some moments, linear combinations of the particle distributions, relax with a proper relaxation rate. We define a matrix of moments depending on a velocity field $\utilde$ function of space and time.
This matrix, supposed to be invertible, is characterized by a set of polynomials $\Pk$, $0\leq k\leq q{-}1$, 
\begin{equation*}\label{eq:MatMu}
 \Miju[kj] = \Pk(\vj-\utilde), \qquad 0\leq k,j\leq q{-}1.
\end{equation*}
Let us note that the MRT scheme corresponds to $\utilde=\vectz$.
The vector of moments is then obtained from the particle distributions through a linear transformation
\begin{equation}\label{eq:ftomu}
\vectmu = \MatMu \; \vectf.
\end{equation}
The relaxation then carries on  the components $\mku$, $0\leq k\leq q-1$, of $\vectmu=(\mku[0],\ldots,\mku[q-1])$
\begin{equation}\label{eq:relaxationu}
\mkue = \mku{+}\sk (\mkueq{-}\mku), \qquad 0\leq k\leq q{-}1,
\end{equation}
where $\mkueq$ are the moments at equilibrium and $\sk$ the relaxation parameters. Some relaxation parameters are chosen null to enforce some conservation laws.
The equilibrium is supposed to derive from an equilibrium distribution $\vectfeq\in\R^q$ independent of $\utilde$
\begin{equation}\label{eq:mueq}
 \vectmequ = \MatMu \vectfeq.
\end{equation}
The equilibrium $\vectfeq$ is an {\it a priori} non linear function of these conserved moments.
The postcollision distributions are obtained by the linear transformation
\begin{equation}\label{eq:mtof}
\vectfe = \MatMinvu\vectmue.
\end{equation}
Then, the particles are advected from node to node
\begin{equation*}\label{eq:transport}
 \fj(\vectx,t+\dt) = \fje(\vectx-\vj\dt,t), \qquad 0\leq j\leq q{-}1.
\end{equation*}
In the following, the relative velocity $\utilde$ is a real constant. When $\utilde$ is specified, the relative velocity scheme is called scheme relative to $\utilde$. 

\subsection{The twisted relative velocity $\ddqqn$ scheme for a linear advection equation} 

In this section, we present a relative velocity $\ddqqn$ scheme for a linear advection equation. We expect the scheme relative to the advection velocity ($\utilde=\vectV$), also called ``cascaded like'' scheme in the following, to be more stable than the MRT scheme $\utilde=\vectz$ when this velocity $\vectV$ increases.

We want to approximate the following bidimensional advection equation
\begin{equation}\label{eq:edptrspt}
\Dt\rho(\vectx,t)+\vectV\cdot\nabla\rho(\vectx,t)=0,\quad \vectx\in\R^2, t\in\vars{\R}{+},
\end{equation}
where $\vectV = (\Vx,\Vy)\in\R^2$ is the advection velocity.
With this aim, we consider the four velocity scheme called twisted $\ddqqn$ defined by 
the set of velocities
$\vectv=\{(\lambda,\lambda),(-\lambda,\lambda),(-\lambda,-\lambda),(\lambda,-\lambda)\},$ for $\lambda=\dx/\dt$ the velocity scale (figure~\ref{fig:ddqqn}). The moments are associated with the polynomials $1,X,Y,XY$. 

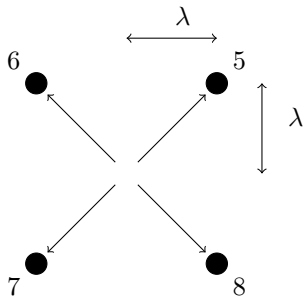
\begin{figure}[H]
\begin{center}
\begin{tikzpicture}[scale=1.5,inner sep = 1.mm] 
		\tikzstyle{every node}=[font=\small]
		\draw[<->](0,1.2) -- (0.8,1.2);\draw[<->](1.2,0) -- (1.2,0.8);
		\draw[<-] (-0.7,-0.7) -- (-0.1,-0.1); \draw[->] (0.1,0.1) -- (0.7,0.7);
		\draw[<-] (-0.7,0.7) -- (-0.1,0.1); \draw[->] (0.1,-0.1) -- (0.7,-0.7);
	         \node at (1,1) {5};\node at (-1,1) {6};\node at (-1,-1) {7};\node at (1,-1) {8}; \node at (0.5,1.4) {$ \lambda$};\node at (1.5,0.5) {$\lambda$};
	          \node[circle,draw,fill=black] at (0.8,-0.8) {};\node[circle,draw,fill=black] at (-0.8,-0.8) {};\node[circle,draw,fill=black] at (-0.8,0.8) {};\node[circle,draw,fill=black] at (0.8,0.8) {};
\end{tikzpicture}\end{center}
\caption{Velocities of the twisted $\ddqqn$ scheme.}
\label{fig:ddqqn}
\end{figure}

Another choice with four velocities could be the $\ddqqn$ defined by the set 
$\vectv=\{(\lambda,0),(0,\lambda),(-\lambda,0),(0,-\lambda)\},$ and by the moments associated with $1,X,Y,X^2-Y^2$. 
However, we have chosen to work with the twisted scheme because this scheme is exactly a $\duqd^2$ in terms of velocities and moments. This feature simplifies the proofs of consistency and stability. Moreover, the stability results of the $\ddqqn$ scheme can be deduced of the results for the twisted $\ddqqn$ scheme: the stability areas correspond thanks to the composition of a rotation and a homothety (Appendix \ref{se:app3}). 

The framework containing only one equation, we choose one conservation law on the density $$\rho=\sum_{j=0}^{3}\fj.$$ The other moments are relaxed with a relaxation parameter $\sk[q]\in\R$ for the first order moments and $\sk[xy]\in\R$ for the second order one. To approach the equation (\ref{eq:edptrspt}), the equilibrium must read
\begin{equation}\label{eq:eqd2q4}
\vectmequ = (\rho,(\Vx-\utx)\rho,(\Vy-\uty)\rho,\mkueq[3]),
\end{equation} 
where $\mkueq[3]$ is the second order moment equilibrium determining the diffusion terms. 
Two different equilibria are chosen: the non intrinsic equilibrium
\begin{equation}\label{eq:eqd2q4tre}
\mkueq[3]=\rho(\Vx-\utx)(\Vy-\uty),
\end{equation} 
and the intrinsic equilibrium
\begin{equation}\label{eq:eqd2q4treis}
\mkueq[3]= \rho(\utx\uty-\utx\Vy-\uty\Vx).
\end{equation} 
The first equilibrium results from the will to see the twisted $\ddqqn$ as the product of the $\duqd$ by himself: it could be seen as the product of two one dimensional equilibria. The second equilibrium is called intrinsic because it leads to a diffusion term that is independent of the basis of writing (proposition \ref{th:ordre3difixis}) but is not isotropic.
It is important to note that these equilibria are chosen such that the relation (\ref{eq:mueq}) is satisfied. Then the second order asymptotics do not depend on the relative velocity $\utilde$ \cite{Fev:2014:0,Fev:2014:1}.

\section{Exploitation of the equivalent equations for the stability}\label{se:eqeq}

The purpose of this section is to predict the stability behaviour of the twisted scheme defined in the section \ref{se:pst} according to the choice of $\utilde$. To do so, we present the third order asymptotics of the scheme: we discuss on the definite positivity of the diffusion tensor and on the dispersive third order terms. The discussion is then illustrated on a numerical study of $L^2$ stability. We show that the cascaded like scheme contrary to the MRT scheme cancels some dispersive terms and stabilizes the scheme in the same time.

\subsection{Discussion for the scheme with a non intrinsic diffusion}\label{sub:diffix}

We study here the third order equivalent equations of the twisted $\ddqqn$ with a non intrinsic diffusion. The following proposition results from the particularization of a general derivation of the equivalent equations for a $\ddqq$ scheme with one conservation law, $d, q\in\N^{\star}$ \cite{Fev:2014:1}. We are interested in the structure of the upper orders according to the velocity field $\utilde$ and their influence on the stability.

\begin{proposition}[Diffusion and dispersion operators]\label{th:ordre3difix}
Given $\vectV\in\R^2$, $(\sk[q]$, $\sk[xy])\in\R^2$, the twisted $\ddqqn$ scheme relative to a constant velocity $\utilde\in\R^2$ associated with the equilibrium (\ref{eq:eqd2q4},\ref{eq:eqd2q4tre}) and with the relaxation parameters $(0,\sk[q],\sk[q],\sk[xy])$, approaches the third order equation
\begin{multline}\label{eq:o3trsptdif}
\Dt\rho(\vectx,t)+\vectV\cdot\nabla\rho(\vectx,t)-\Delta t~\var{\mD}{2}:\nabla^2\rho(\vectx,t)+\Delta t^2 \var{\mathcal{D}}{3}:\nabla^3\rho(\vectx,t)={\rm \mathcal{O}}(\Delta t^3),\\
\quad \vectx\in\R^2, t\in\vars{\R}{+},
\end{multline} 
where the diffusion matrix $\var{\mD}{2}$ is given by 
\begin{equation}\label{eq:Drtw}
\var{\mD}{2}=\sig[q]\begin{pmatrix} (\lambda^2-(\Vx)^2)&0\\
0&(\lambda^2-(\Vy)^2)
\end{pmatrix},
\end{equation}
 and the dispersion matrix $\var{\mD}{3}$ by 
\begin{equation}\label{eq:Disptw}
\var{\mD}{3}=\begin{pmatrix} \frac{\Vx}{6}(\lambda^2-(\Vx)^2)(1-12\varp{\sigma}{q}{2})&2\sig[q](\lambda^2{-}(\Vy)^2)(\utx{-}\Vx)(\sig[q]{-}\sig[xy])\\
2\sig[q](\lambda^2{-}(\Vx)^2)(\uty{-}\Vy)(\sig[q]{-}\sig[xy])& \frac{\Vy}{6}(\lambda^2-(\Vy)^2)(1-12\varp{\sigma}{q}{2})
\end{pmatrix}.
\end{equation}
The Hénon's parameters reas $\sig[q]=1/\sk[q]-1/2$ and $\sig[xy]=1/\sk[xy]-1/2$, the second and third order derivative operators are given by 
\begin{equation}\label{eq:grad}
\nabla^2\rho=\begin{pmatrix} \Dxy[x]\rho&\Dxy[xy]\rho\\
\Dxy[xy]\rho&\Dxy[y]\rho
\end{pmatrix},\quad
\nabla^3\rho=\begin{pmatrix} \Dxxx[x]\rho&\Dxxx[xyy]\rho\\
\Dxxx[xxy]\rho&\Dxxx[y]\rho
\end{pmatrix},
\end{equation}
and $:$ is the scalar product for the matrices viewed as vectors of $\R^4$.
\end{proposition}

These equations are obtained thanks to a formal calculus software, that implements an algorithm to obtain the equivalent equations in the linear case \cite{Augier:2014:0}. The following lemma exhibits the area of well-posedness of the second order truncation of (\ref{eq:o3trsptdif}).

\begin{lemme}[Well-posedness of the second order equation]\label{th:diffrtw}
The diffusion matrix $\mD$ defined by (\ref{eq:Drtw}) belongs to the space $\Sdpp$ of symmetric positive definite matrices if and only if $\normi<\lambda$.
\end{lemme}

\begin{proof}
The spectrum of the matrix $\var{\mD}{2}$ is $\{\lambda^2-(\Vx)^2, \lambda^2-(\Vy)^2\}$ that closes the proof.
\end{proof}

Because of the relation (\ref{eq:mueq}), the velocity parameter $\utilde$ appears only at the third order of the equations. Let us discuss of the structure of these third order dispersive terms according to the choice of $\utilde$.

First we choose the velocity field $\utilde$ equal to the advection velocity $\vectV$ (``cascaded like'' scheme). The dispersion terms are then small compared to the diffusion terms. Indeed, this choice cancels the two off-diagonal components of $\var{\mD}{3}$ depending on the parameter $\sig[xy]$. Thus resting dispersion terms are $\Vx(\lambda^2-(\Vx)^2)(1-12\varp{\sigma}{q}{2})/6$ and its symmetric in $y$. When $\sig[q]$ tends to $0$, those terms are equivalent to $\vectV(\lambda^2-(\Vx)^2)$ multiplied by a constant. Thus when the diffusion $(\lambda^2-(\Vx)^2)$ decreases, the dispersion decreases at least with the same speed.

On the contrary, the dispersion terms of the MRT scheme ($\utilde=\vectz$) create instabilities when the diffusion (represented by $\sig[q]$) is weak. Indeed, the off-diagonal terms are conserved for this choice of velocity parameter: they are given by
$-2\sig[q]\Vy(\lambda^2-(\Vx)^2)(\sig[q]-\sig[xy])$ and its symmetric in $y$. This term is equivalent to $2\sig[xy]\sig[q]\Vy(\lambda^2-(\Vx)^2),$ when $\sig[q]$ gets close to $0$. The dispersion terms are thus depending on the size of the parameter $\sig[xy]$. If $\sig[xy]$ is important and $\sig[q]$ tends to $0$, the numerical stability should be deteriorated by dispersion phenomena for $\Vx$ or $\Vy$ close to $\lambda$: indeed, the dispersion, behaving as $2\sig[xy]\sig[q]\lambda(\lambda^2-(\Vx)^2),$ becomes greater than the diffusion given by $\sig[q](\lambda^2-(\Vx)^2).$  Instead, taking $\sig[xy]$ close to $0$ should limit the dispersion effects.

\begin{remarque}[Exact third order scheme]
We note that the choices $\utilde=\vectV$ and $\smash{\sig[q]}=1/\sqrt{12}$ cancel the dispersion matrix $\var{\mD}{3}$ given by (\ref{eq:Disptw}).  Thus the scheme is consistent with the second order advection diffusion truncation of (\ref{eq:o3trsptdif}) at the third order. Moreover, there is still a degree of freedom with the parameter $\sig[xy]$. This is impossible for the MRT scheme ($\utilde=\vectz$) excepted when there is only one relaxation parameter (BGK scheme: $\sig[q]=\sig[xy]$).
\end{remarque}

\subsection{Discussion for the scheme with an intrinsic diffusion}\label{sub:diffixis}

We discuss on the structure of the third order equivalent equations for the twisted $\ddqqn$ intrinsic scheme obtained thanks to the linear algorithm derived in \cite{Augier:2014:0}. The study is completely analogous to the one of the section \ref{sub:diffix}.

\begin{proposition}[Diffusion and dispersion operators]\label{th:ordre3difixis}
Given $\vectV\in\R^2$, $(\sk[q]$, $\sk[xy])\in\R^2$, the twisted $\ddqqn$ scheme relative to a constant velocity $\utilde$ associated with the equilibrium (\ref{eq:eqd2q4},\ref{eq:eqd2q4treis}) and with the relaxation parameters $(0,\sk[q],\sk[q],\sk[xy])$, approaches the following third order equation 
\begin{multline*}\label{eq:o3trsptdifis}
\Dt\rho(\vectx,t)+\vectV\cdot\nabla\rho(\vectx,t)-\Delta t~\var{\mD}{2}:\nabla^2\rho(\vectx,t)+\Delta t^2 \var{\mD}{3}:\nabla^3\rho(\vectx,t)={\rm \mathcal{O}}(\Delta t^3),\\
\quad \vectx\in\R^2, t\in\vars{\R}{+},
\end{multline*} 
where the diffusion matrix $\var{\mD}{2}$ is given by 
\begin{equation}\label{eq:DHtw}
\var{\mD}{2}=\begin{pmatrix} \sig[q](\lambda^2-(\Vx)^2)&-\sig[q]\Vx\Vy\\
-\sig[q]\Vx\Vy&\sig[q](\lambda^2-(\Vy)^2)
\end{pmatrix},
\end{equation}
 and the dispersion matrix $\var{\mD}{3}$ by 
 \begin{equation*}\label{eq:Disptwis}
\var{\mD}{3}=\begin{pmatrix} \vars{\phi}{1}(\Vx)&\vars{\phi}{2}(\utx,\uty,\Vx,\Vy)\\
\vars{\phi}{2}(\uty,\utx,\Vy,\Vx)&\vars{\phi}{1}(\Vy)
\end{pmatrix},
\end{equation*}
with 
\begin{gather*}
\vars{\phi}{1}(\Vx)=\frac{\Vx}{6}(\lambda^2-(\Vx)^2)(1-12\varp{\sigma}{q}{2}),\\
\vars{\phi}{2}(\utx,\uty,\Vx,\Vy)=\frac{1}{2}\big(-\Vx(\Vy)^2+4\varp{\sigma}{q}{2}(\lambda^2(\utx-\Vx)+3\Vx(\Vy)^2-\uty\Vx\Vy\\-\utx(\Vy)^2)
-4\varp{\sigma}{q}{}\varp{\sigma}{xy}{}(\lambda^2(\utx-\Vx)-(\Vx\uty\Vy+\utx(\Vy)^2).
\end{gather*}
The Hénon's parameters $\sig[q]$ and $\sig[xy]$ are given by $1/\sk[q]-1/2$ and $1/\sk[xy]-1/2$, the operators $\nabla^2$ and $\nabla^3$ are given by (\ref{eq:grad}) and $:$ is the scalar product for the matrices viewed as vectors of $\R^4$.
\end{proposition}

\begin{remarque}
As expected, the diffusion is intrinsic. It is independent of the basis of writing. Indeed it is equal to 
$$\var{\mD}{2}:\nabla^2=\lambda^2(\Dxy[x]+\Dxy[y])-(\vectV\cdot\nabla)^2.$$
\end{remarque}

\begin{lemme}[Well-posedness of the second order equation]\label{th:diffHtw}
The diffusion matrix $\var{\mD}{2}$ defined by (\ref{eq:DHtw}) belongs to the space $\Sdpp$ of symmetric positive definite matrices if and only if $\norme< \lambda$.
\end{lemme}

\begin{proof}
The eigenvalues of the matrix $\var{\mD}{2}$ are $\lambda^2$ and $\lambda^2-\norme[2]$, that closes the proof.
\end{proof}

Let us compare the well-posedness areas of the schemes with a non intrinsic diffusion (section~\ref{sub:diffix}) and an intrinsic diffusion. These areas are represented on the figure~\ref{fig:ddqqn_diff}.
We note that the scheme with a non intrinsic diffusion has a greater area in $\vectV$. Consequently, it should authorize larger stable velocities than the scheme with an intrinsic diffusion.

We now get back to the intrinsic case and discuss on the dispersion terms. As in the non intrinsic case, the scheme should be stable for larger velocities with the choice $\utilde=\vectV$ than with $\utilde=\vectz$. Indeed, when $\sig[q]$ tends to $0$, the term $\vars{\phi}{2}(\utx,\uty,\Vx,\Vy)$ is of third order in $\vectV$ if $\utilde=\vectV$. Instead, if $\utilde=\vectz$, it is of first order in $\vectV$. More dispersion is created for $\utilde=\vectz$ when $\Vx$ or $\Vy$ are close to $\lambda$ and $\sig[xy]$ is big. 

\definecolor{ffqqqq}{rgb}{1,0,0}
\definecolor{qqqqff}{rgb}{0,0,1}
\definecolor{qqcctt}{rgb}{0,0.8,0.2}
\definecolor{uququq}{rgb}{0.25,0.25,0.25}
 \begin{figure}
\begin{center}
\begin{tikzpicture}[scale=0.7,line cap=round,line join=round,>=triangle 45,x=4.0cm,y=4.0cm]
\draw[->,color=black] (-1.2,0) -- (1.2,0);
\draw[->,color=black] (0,-1.2) -- (0,1.2);
\draw[color=black] (0pt,-10pt) node[right] {\footnotesize $0$};
\draw[shift={(1,0)},color=black] (0pt,2pt) -- (0pt,-2pt) node[below] {\footnotesize $\lambda$};
\draw[shift={(-1,0)},color=black] (0pt,2pt) -- (0pt,-2pt) node[below] {\footnotesize $-\lambda$};
\draw[shift={(0,1)},color=black] (2pt,0pt) -- (-2pt,0pt) node[left] {\footnotesize $\lambda$};
\draw[shift={(0,-1)},color=black] (2pt,0pt) -- (-2pt,0pt) node[left] {\footnotesize $-\lambda$};
\clip(-1.2,-1.2) rectangle (1.2,1.2);
\draw [fill=black,fill opacity=0.1] (0,0) circle (4cm);
\fill[color=black,fill=black,fill opacity=0.15] (-1,1) -- (-1,-1) -- (1,-1) -- (1,1) -- cycle;
\draw [color=black] (-1,1)-- (-1,-1);
\draw [color=black] (-1,-1)-- (1,-1);
\draw [color=black] (1,-1)-- (1,1);
\draw [color=black] (1,1)-- (-1,1);
\begin{scriptsize}
\fill [color=black] (0,0) circle (1.5pt);
\fill [color=black] (0,1) circle (1.5pt);
\fill [color=black] (1,0) circle (1.5pt);
\fill [color=black] (-1,0) circle (1.5pt);
\fill [color=black] (0,-1) circle (1.5pt);
\fill [color=black] (1,0) circle (1.5pt);
\fill [color=black] (0,1) circle (1.5pt);
\fill [color=black] (-1,1) circle (1.5pt);
\fill [color=black] (-1,-1) circle (1.5pt);
\fill [color=black] (1,-1) circle (1.5pt);
\fill [color=black] (1,1) circle (1.5pt);
\draw[color=black] (1.15,-0.075) node {$\Vx$};
\draw[color=black] (-0.085,1.15) node {$\Vy$};
\end{scriptsize}
\end{tikzpicture}
\end{center}
\caption{Well posedness areas in $\vectV$ for the twisted $\ddqqn$ schemes: with intrinsic diffusion (circle of radius $\lambda$), with a non intrinsic diffusion (square $[-\lambda,\lambda]^2$).}
\label{fig:ddqqn_diff}
\end{figure}
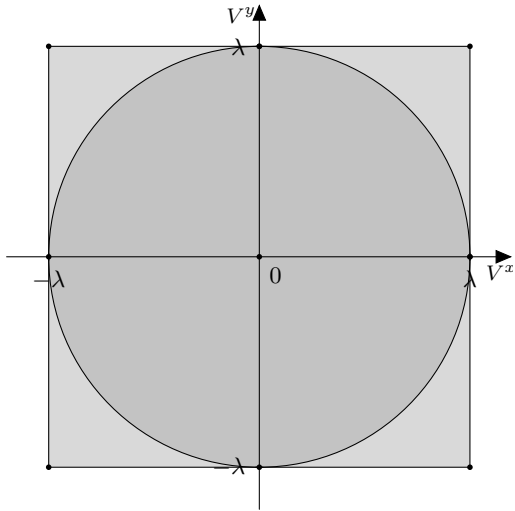

\subsection{Illustration with numerical stability experiments}\label{sub:stabeqeq}

In this section, we study the numerical $L^2$ stability for the twisted $\ddqqn$ scheme according to the choice of the velocity parameter $\utilde$. The link with the previous consistency study is made: the third order dispersion terms have an influence on the numerical stability. The choice $\utilde=\vectV$ provides larger stability areas in $\vectV$.

We present a first numerical experiment with the advection of a circular spot: the initial conditions are given by
$$\rho(\vectx,t)=1+\ind{\mC}(\vectx,t),\quad \vectx\in[0,1]^2, t\in\R,$$
where $\mC$ is the disc centered in $(1/2,1/2)$ of radius $0.1$ and $\ind{\mC}$ is the function equal to $1$ on this disc, null elsewhere. This spot moves with an advection velocity $\vectV\in\R^2$ in the domain $[0,1]^2$ constituted of $128^2$ points with periodic boundary conditions.

The scheme is considered as stable if it has not broken after 2000 iterations. We present the biggest advection velocities keeping the scheme stable for different choices of relaxation parameters. Two different directions of advection velocities are considered: $\theta=0$ and $\theta=\pi/4$.  We choose different values for the parameter $\sig[q]$, the parameter $\sig[xy]$ being set to $1/\sqrt{3}$. The results are presented in the tables \ref{table:dispdifix} and \ref{table:dispdifix2}.  
 
\begin{table}
\centering\small
\grandtraittop
\begin{tabular}{@{}p{5.3cm}p{1.4cm}p{1.4cm}p{1.4cm}p{1.4cm}p{1.4cm}@{}}
  $\sig[q]$& $1/10$& $1/20$&$1/50$&$1/100$&$1/200$  \\ \hline
     $\ddqqn$ non intrinsic $\utilde=\vectz$ & 1.00&0.80&0.49&0.34&0.23\\ 
     $\ddqqn$ non intrinsic $\utilde=\vectV$ & 1.00 &1.00 &1.00 &1.00 &1.00\\ 
      $\ddqqn$ intrinsic $\utilde=\vectz$ &1.00&0.79&0.48&0.33&0.23 \\ 
     $\ddqqn$ intrinsic $\utilde=\vectV$  & 1.00 &1.00 &1.00 &1.00 &1.00 \\ 
\end{tabular}
\grandtraitbottom
 \caption{Maximal $|\vectV|$ stable in $\lambda$ scale for $\theta=0$, $\sig[xy]=1/\sqrt{3}$ and $128^2$ points.}
 \label{table:dispdifix}
\end{table}

\begin{table}
\centering\small
\grandtraittop
\begin{tabular}{@{}p{5.3cm}p{1.4cm}p{1.4cm}p{1.4cm}p{1.4cm}p{1.4cm}@{}}
  $\sig[q]$& $1/10$& $1/20$&$1/50$&$1/100$&$1/200$  \\ \hline
      $\ddqqn$ non intrinsic $\utilde=\vectz$ & 1.41&0.80&0.42&0.28&0.20\\ 
     $\ddqqn$ non intrinsic $\utilde=\vectV$ & 1.41 &1.41 &1.41 &1.41 &1.41\\ 
      $\ddqqn$ intrinsic $\utilde=\vectz$ &0.75&0.56&0.36&0.26&0.18 \\ 
     $\ddqqn$ intrinsic $\utilde=\vectV$  &0.86&0.76&0.65&0.59&0.53 \\ 
\end{tabular}
\grandtraitbottom
 \caption{Maximal $|\vectV|$ stable in $\lambda$ scale  for $\theta=\pi/4$, $\sig[xy]=1/\sqrt{3}$ and $128^2$ points.}
 \label{table:dispdifix2}
\end{table}

The scheme relative to the advection velocity $\utilde=\vectV$ authorizes larger stable velocities than the MRT scheme ($\utilde=\vectz$). Most of the time, its stability areas do not depend on the relaxation parameters while they are bounded by $0$ and $2$.
Instead, the stability areas of the MRT scheme decreases when $\sig[q]$ tends to $0$.
This phenomenon occurs for the scheme relative to $\utilde=\vectV$ with an intrinsic diffusion in the direction $\theta=\pi/4$: however the instabilities appear for larger $\vectV$ than the MRT scheme.

The dispersion terms exhibited in the equivalent equations (section \ref{se:eqeq}) originate these phenomena. Their presence causes instabilities when $\sig[q]$ is too small compared to $\sig[xy]$ and $|\vectV|$. We have seen that these terms cancel for $\utilde=\vectV$ but are present when $\utilde=\vectz$. This explains the constance of the stability area for $\utilde=\vectV$ and its deterioration for $\utilde=\vectz$ when $\sig[q]$ decreases.

These dispersive phenomena are illustrated on the figures \ref{fig:tachedisp1} and \ref{fig:tachedisp2} for the scheme with an intrinsic diffusion. We choose $\sig[xy]=1/\sqrt{3}$, $\sig[q]=1/20$ and $\vectV=(0.9\lambda,0)$ for these draws. According to the data of the table \ref{table:dispdifix}, this configuration is a stable case for the scheme relative to $\utilde=\vectV$ and an unstable case for the MRT scheme. The spot is represented at the times $t=0$ and $t=0.4$ for $256^2$ points.

\begin{figure}
  \begin{minipage}{0.45\linewidth} 
\begin{center}
\includegraphics[trim = 50mm 10mm 50mm 10mm,width=1.\textwidth]{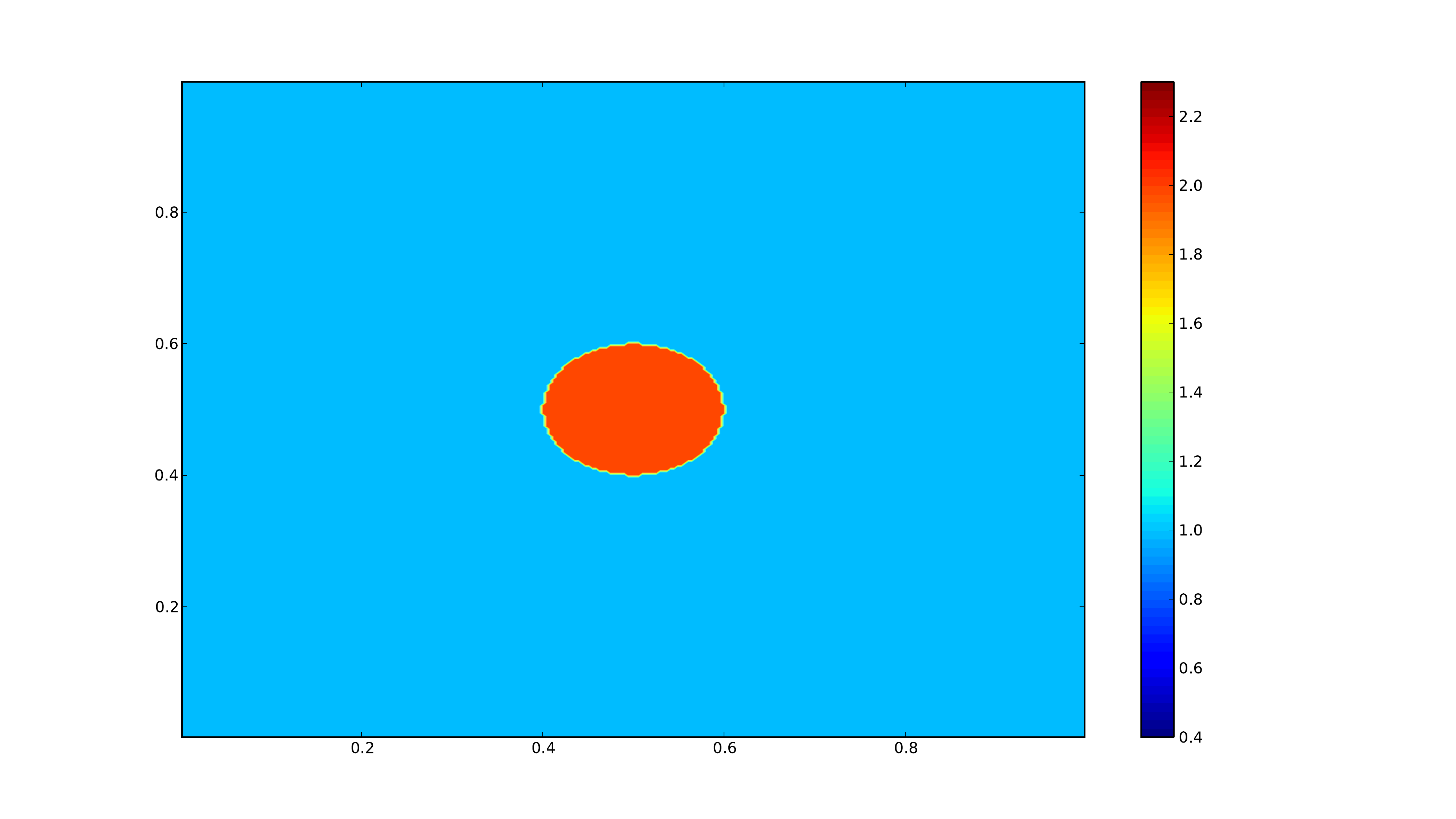}
\end{center}
\end{minipage}\hfill
\begin{minipage}{0.45\linewidth} 
\begin{center}
\includegraphics[trim = 50mm 10mm 50mm 10mm,width=1.\textwidth]{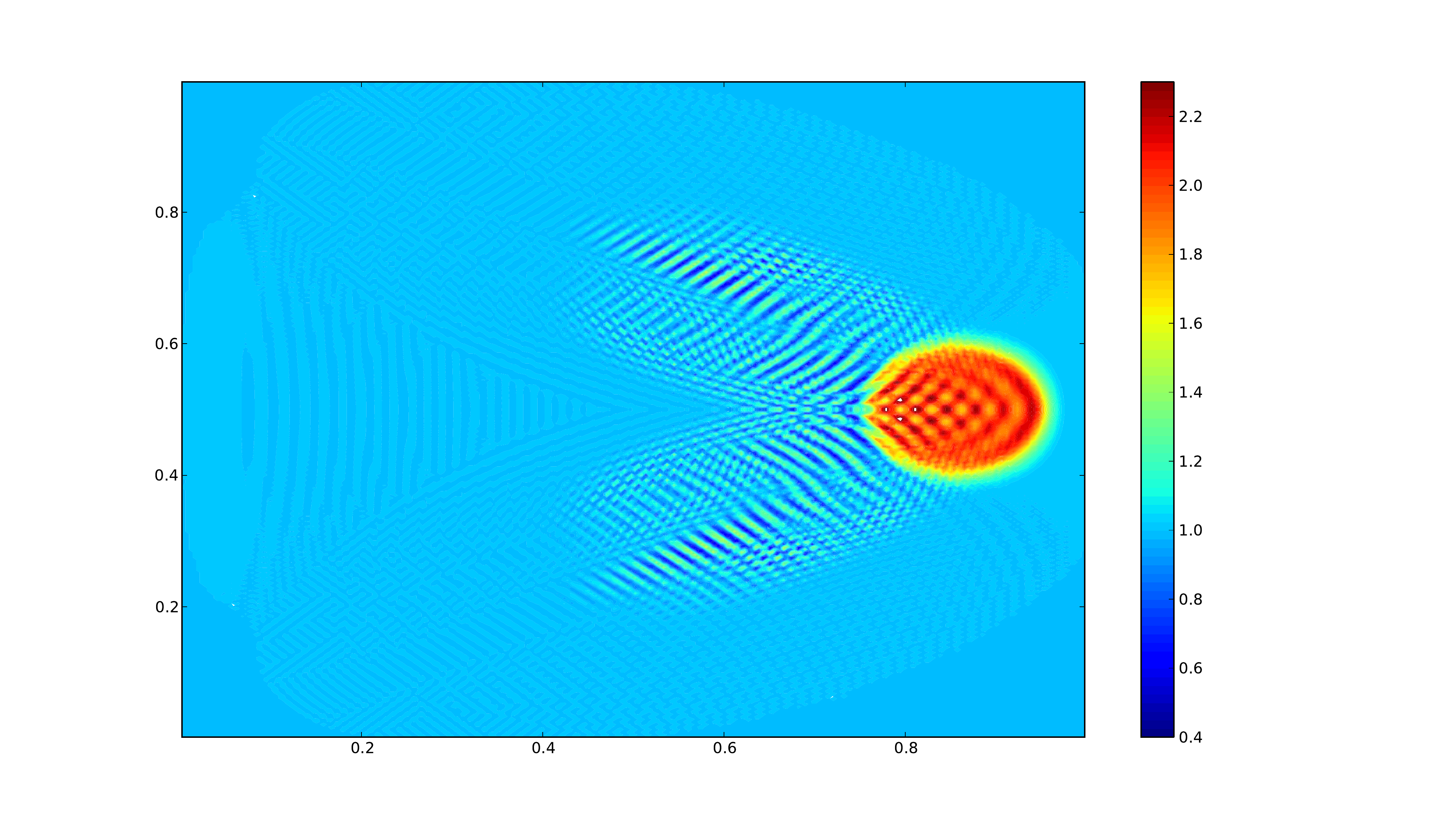}
\end{center}
\end{minipage}\\
\caption{Advected spot of velocity $\vectV=(0.9\lambda,0)$ for the twisted $\ddqqn$ scheme relative to $\utilde=\vectz$ (MRT) with an intrinsic diffusion at $t=0$ (left) and $t=0.4$ (right) for $\sig[xy]=1/\sqrt{3}$, $\sig[q]=1/20$.}
\label{fig:tachedisp1}
\end{figure}

\begin{figure}
  \begin{minipage}{0.45\linewidth} 
\begin{center}
\includegraphics[trim = 50mm 10mm 50mm 10mm,width=1.\textwidth]{images/Tache_disp_dHtw_t=0.pdf}
\end{center}
\end{minipage}\hfill
\begin{minipage}{0.45\linewidth} 
\begin{center}
\includegraphics[trim = 50mm 10mm 50mm 10mm,width=1.\textwidth]{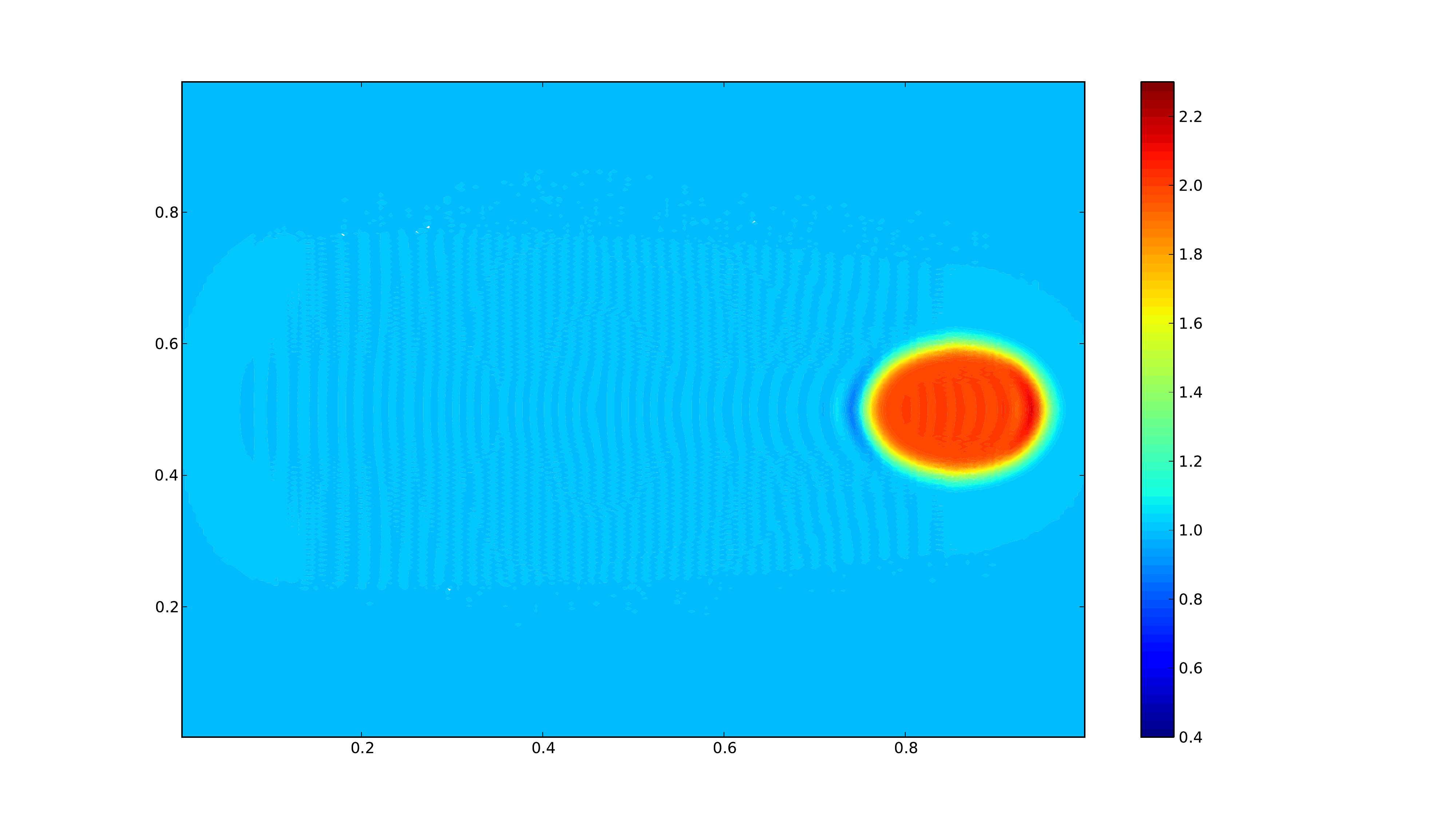}
\end{center}
\end{minipage}\\
\caption{Advected spot of velocity $\vectV=(0.9\lambda,0)$ for the twisted $\ddqqn$ scheme relative to $\utilde=\vectV$ with an intrinsic diffusion at $t=0$ (left) and $t=0.4$ (right) for $\sig[xy]=1/\sqrt{3}$, $\sig[q]=1/20$.}
\label{fig:tachedisp2}
\end{figure}

The dispersion for the MRT scheme is clearly visible on the figure~\ref{fig:tachedisp1}: some oscillations appear behind the advected spot. The scheme is going to break since the density $\rho$ is increasing. These phenomena are absent for the scheme relative to $\utilde=\vectV$ (figure~\ref{fig:tachedisp2}) that remains numerically stable.

Our second numerical experiment uses the $L^2$ linear stability of von Neumann. We discuss on the spectrum of the amplification matrix of the scheme. This matrix, characterizing an iteration of the scheme, has to be first determined. The equilibrium being linear (\ref{eq:eqd2q4},\ref{eq:eqd2q4tre},\ref{eq:eqd2q4treis}), there exists a matrix $\MatE=\MatE(\vectV,\vects)$ for $\vects=(0,\sk[q],\sk[q],\sk[xy])$ and $\vectV\in\R^2$ so that  $$\vectfeq=\MatE\vectf.$$  The relaxation phase of the relative velocity $\ddqqn$ scheme reads
\begin{equation*}\label{eq:rel}
\vectfe=(\MatI+\MatMu^{-1}\MatD\MatMu(\MatE-\MatI))\vectf,
\end{equation*}
where $\MatD={\rm diag}(\vects)$ is the diagonal matrix of the relaxation parameters. This expression holds for each node $\vectx$ of the lattice, the relaxation being local in space. The expression of the distribution after the transport phase is given by
\begin{equation}\label{eq:iter}
\fj(\vectx,t+\dt)=[(\MatI+\MatMu^{-1}\MatD\MatMu(\MatE-\MatI))\vectf]_j^{}(\vectx-\vj\dt,t),\quad\vectx\in\lattice, t\in\R.
\end{equation}
Taking the Fourier transform of (\ref{eq:iter}), the transport operator becomes local in space and is represented by the diagonal matrix $\MatA=\MatA(\vectk)$ for $\vectk\in\R^2$ whose diagonal components are given by $e^{i\dt\vectk.\vj}$, $0\leq j\leq3$.
 The amplification matrix then reads $\MatL(\utilde)=\MatL(\utilde,\vectV,\vectk,\vects)=\MatA(\MatI+\MatMu^{-1}\MatD\MatMu(\MatE-\MatI))$. It characterizes a time iteration of the scheme in the Fourier space
$$\widehat{\vectf}(\vectk,t+\dt)=\MatL(\utilde)\widehat{\vectf}(\vectk,t),\quad t\in\R,$$
where $\widehat{\vectf}$ is the Fourier transform of $\vectf$.
We want to exhibit the advection velocities $\vectV$ for which the scheme verifies the necessary condition of $L^2$ stability 
\begin{equation}\label{eq:CN}
\underset{\vectk\in\R^2}{\max}~r(\MatL(\utilde))\leq1,
\end{equation}
 where $r$ is the spectral radius.  We are thus interested in the set
\begin{equation*}\label{eq:maxVrho}
\{\vectV=(\Vx,\Vy),~ \underset{\vectk\in\R^2}{\max}~r(\MatL(\utilde))\leq1\}.
\end{equation*}

The figures  \ref{fig:vpddqq_1} to \ref{fig:vpddqq_5} present this set in the plan $(\Vx,\Vy)$  for the twisted $\ddqqn$ scheme with a non intrinsic diffusion and some relaxation parameters $\sk[q]$ and $\sk[xy]$. The left draw is about the MRT scheme ($\utilde=\vectz$) and the right one illustrates the case $\utilde=\vectV$. We expect the right areas to be larger than the left ones.

\begin{figure}
\begin{minipage}{0.45\linewidth} 
\begin{center}
\includegraphics[trim = 40mm 100mm 40mm 90mm,width=1.\textwidth]{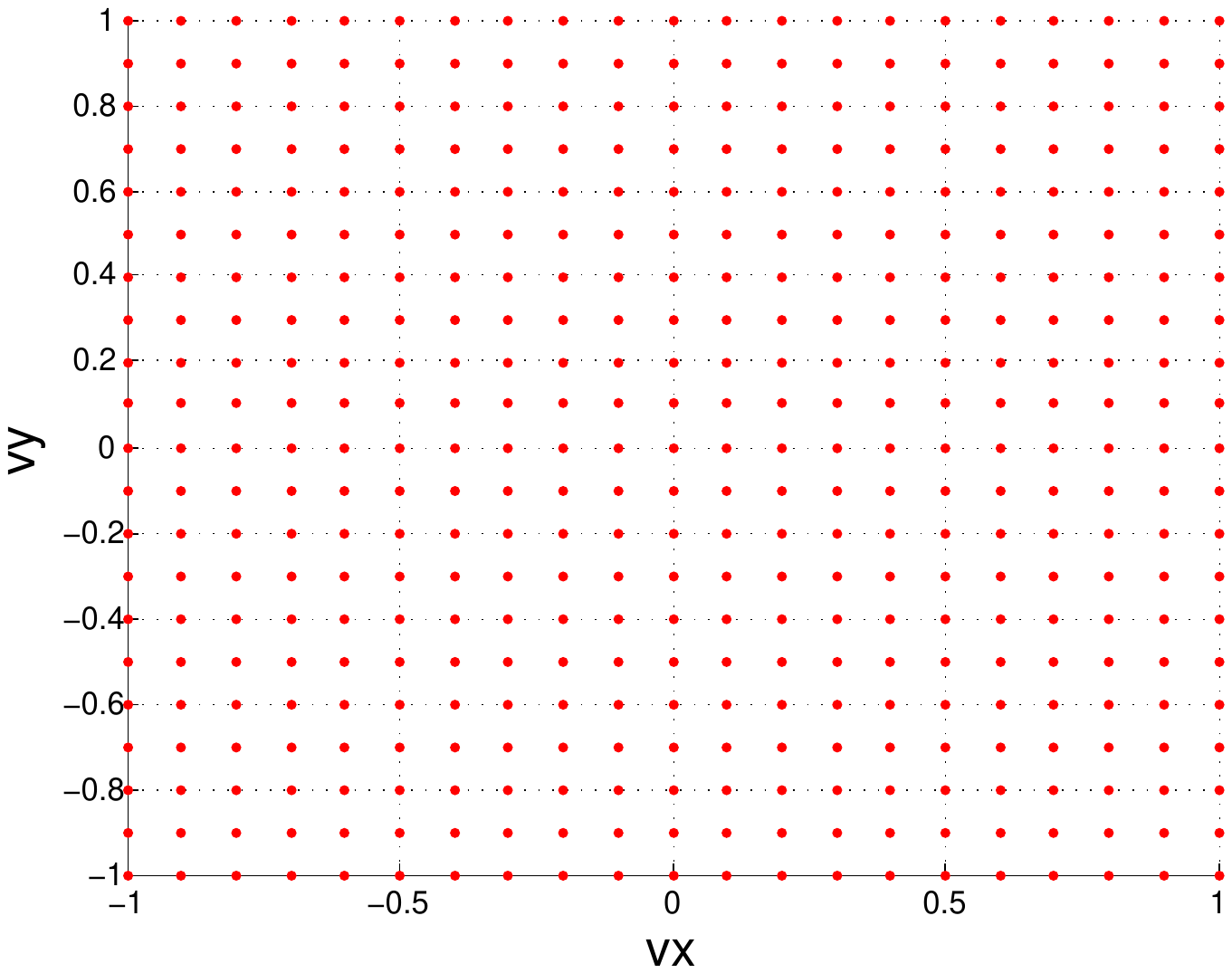}
\end{center}
\end{minipage}\hfill
\begin{minipage}{0.45\linewidth} 
\begin{center}
\includegraphics[trim = 40mm 100mm 40mm 90mm,width=1.\textwidth]{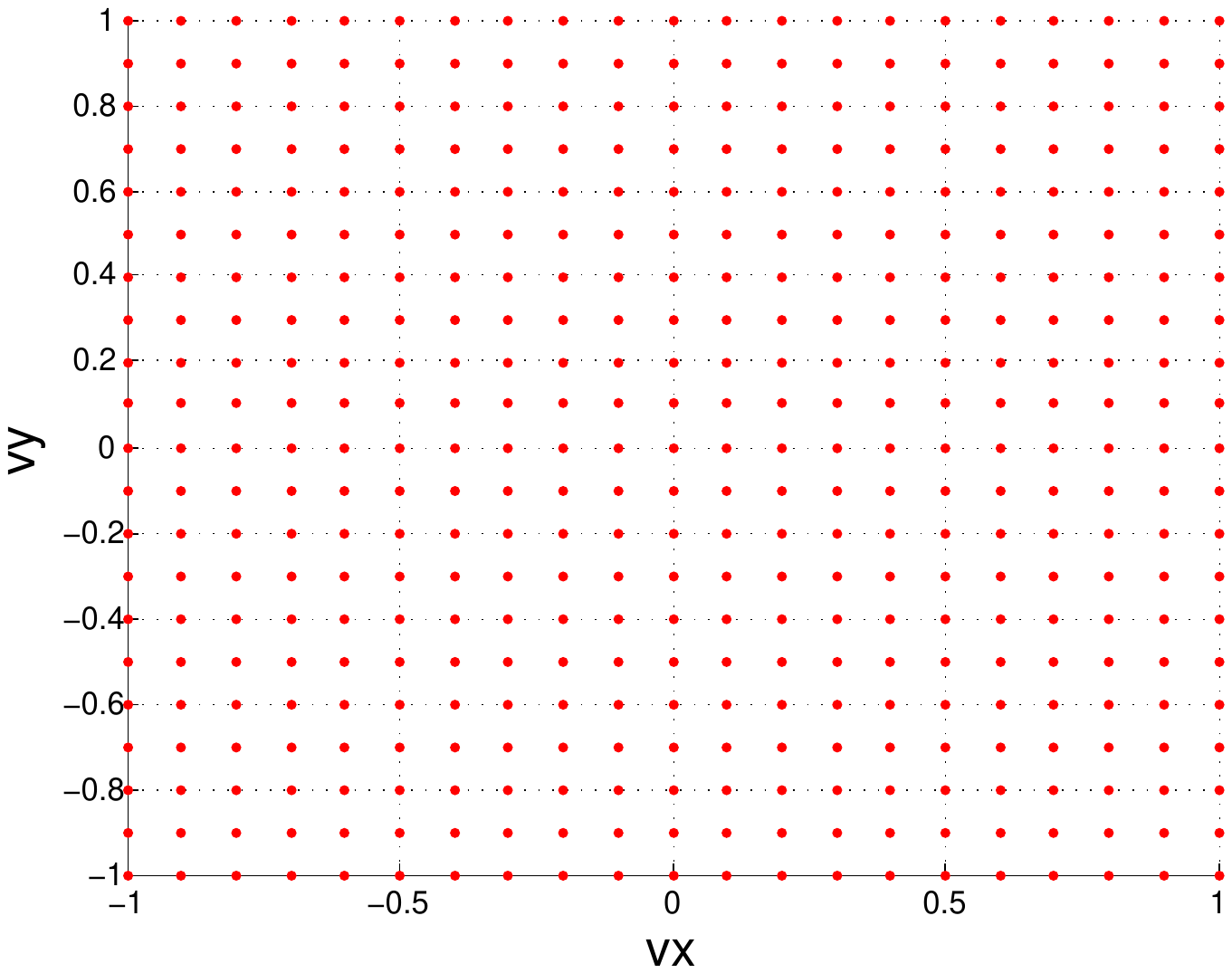}
\end{center}
\end{minipage}\\
\caption{Velocities $\vectV$ stable in $\lambda$ scale for the scheme relative to $\utilde=\vectz$ (MRT on the left), $\utilde=\vectV$ (on the right), for a non intrinsic diffusion with $\sk[q]=1$ and $\sk[xy]=1$.}
\label{fig:vpddqq_1}
\end{figure}

\begin{figure}
  \begin{minipage}{0.45\linewidth} 
\begin{center}
\includegraphics[trim = 40mm 100mm 40mm 90mm,width=1.\textwidth]{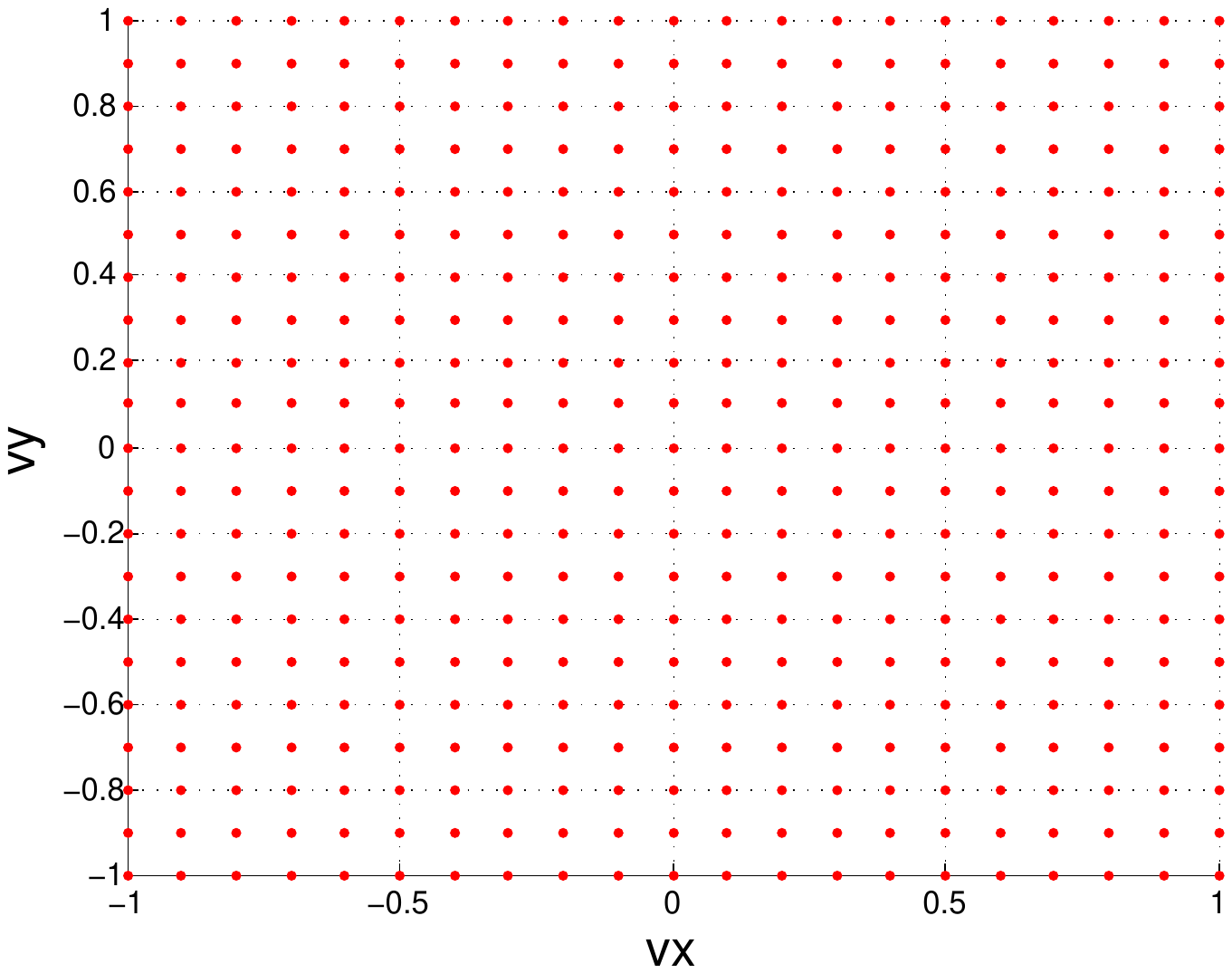}
\end{center}
\end{minipage}\hfill
\begin{minipage}{0.45\linewidth} 
\begin{center}
\includegraphics[trim = 40mm 100mm 40mm 90mm,width=1.\textwidth]{images/d2q4_t_nis_V_12.pdf}
\end{center}
\end{minipage}\\
\caption{Velocities $\vectV$ stable in $\lambda$ scale for the scheme relative to $\utilde=\vectz$ (MRT on the left), $\utilde=\vectV$ (on the right), for a non intrinsic diffusion with $\sk[q]=1$ and $\sk[xy]=1.5$.}
\label{fig:vpddqq_2}
\end{figure}

\begin{figure}
\begin{minipage}{0.45\linewidth} 
\begin{center}
\includegraphics[trim = 40mm 100mm 40mm 90mm,width=1.\textwidth]{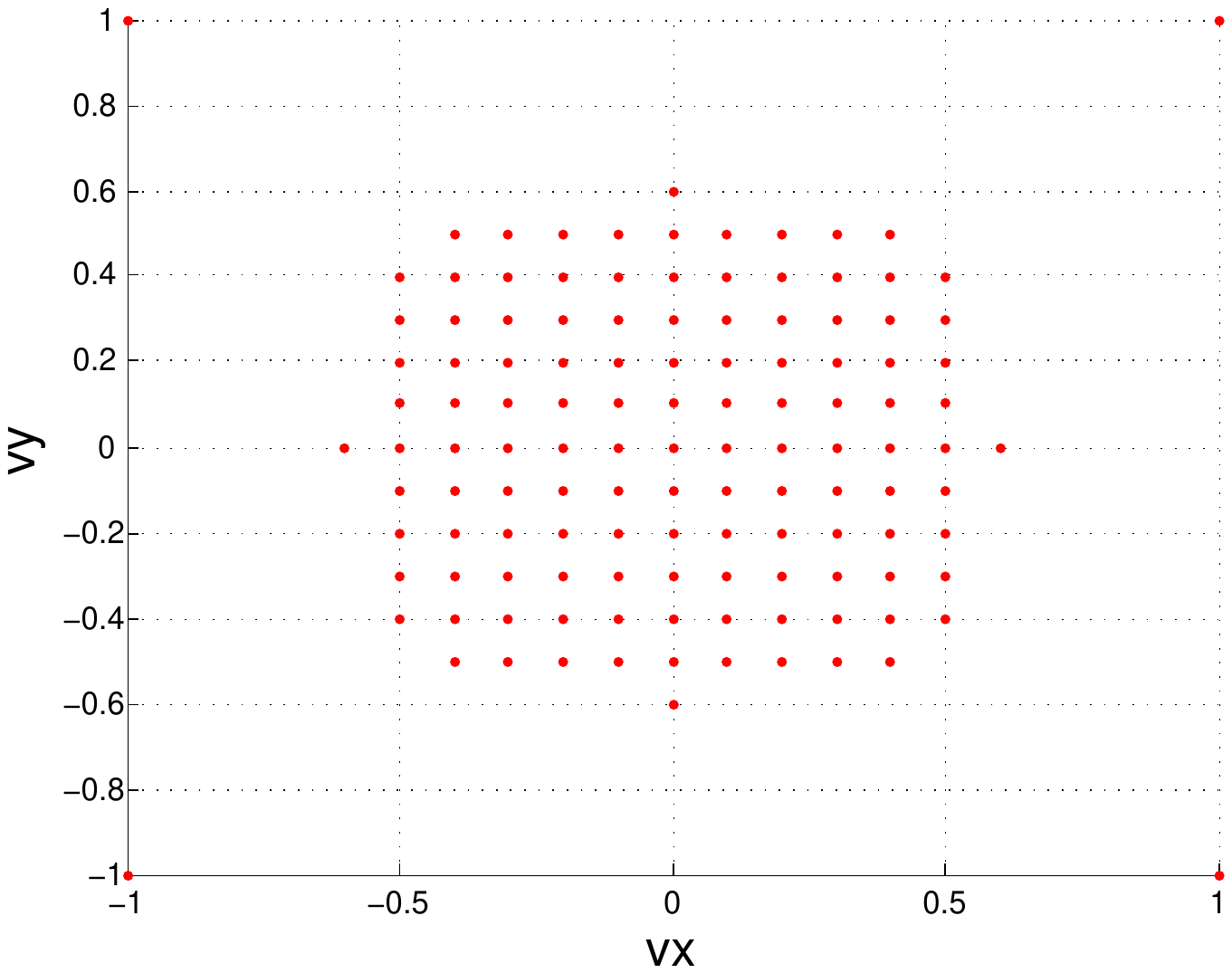}
\end{center}
\end{minipage}\hfill
\begin{minipage}{0.45\linewidth} 
\begin{center}
\includegraphics[trim = 40mm 100mm 40mm 90mm,width=1.\textwidth]{images/d2q4_t_nis_V_12.pdf}
\end{center}
\end{minipage}\\
\caption{Velocities $\vectV$ stable in $\lambda$ scale for the scheme relative to $\utilde=\vectz$ (MRT on the left), $\utilde=\vectV$ (on the right), for a non intrinsic diffusion with $\sk[q]=1$ and $\sk[xy]=1.9$.}
\label{fig:vpddqq_3}
\end{figure}

\begin{figure}
  \begin{minipage}{0.45\linewidth} 
\begin{center}
\includegraphics[trim = 40mm 100mm 40mm 90mm,width=1.\textwidth]{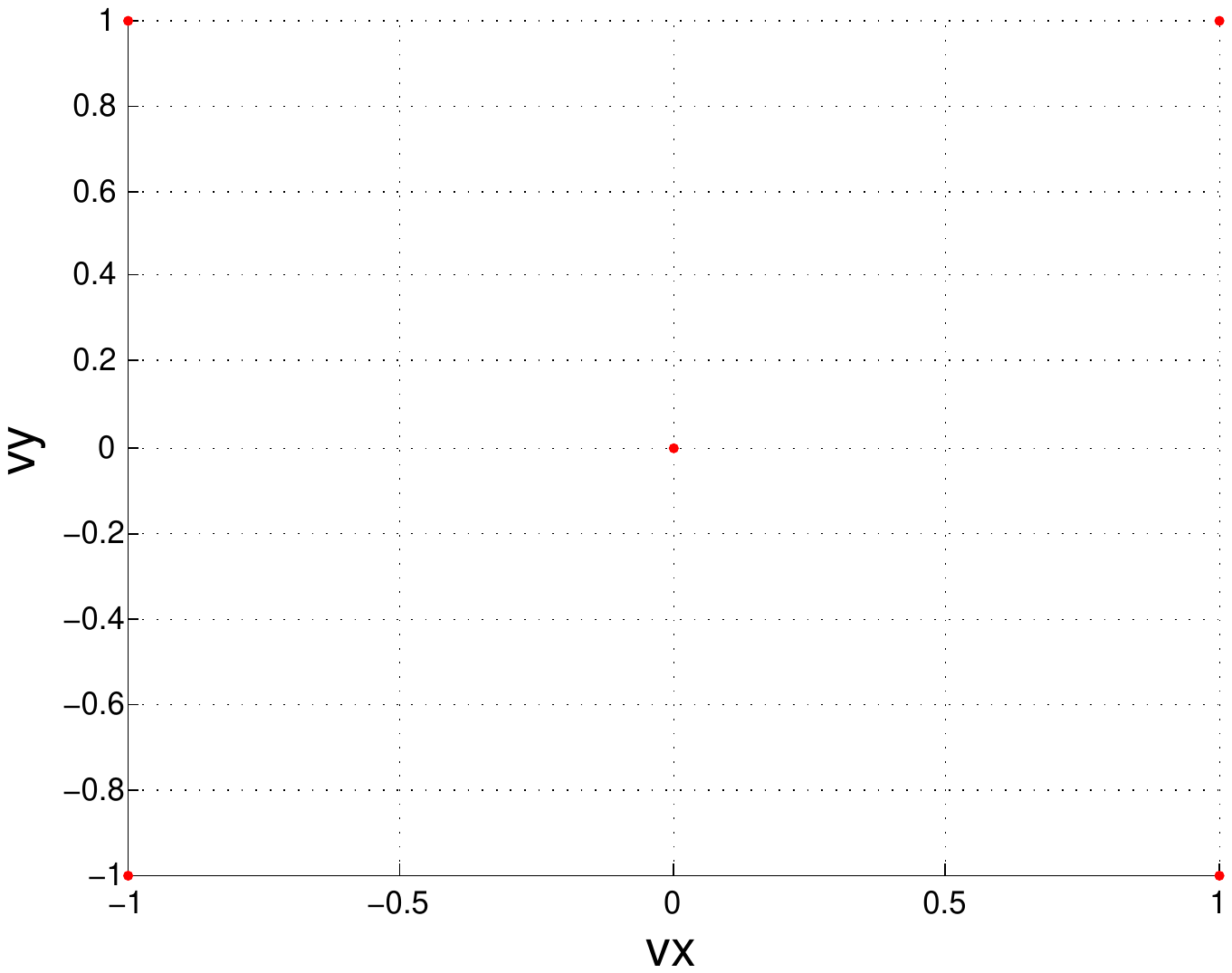}
\end{center}
\end{minipage}\hfill
\begin{minipage}{0.45\linewidth} 
\begin{center}
\includegraphics[trim = 40mm 100mm 40mm 90mm,width=1.\textwidth]{images/d2q4_t_nis_V_12.pdf}
\end{center}
\end{minipage}\\
\caption{Velocities $\vectV$ stable in $\lambda$ scale for the scheme relative to $\utilde=\vectz$ (MRT on the left), $\utilde=\vectV$ (on the right), for a non intrinsic diffusion with $\sk[q]=1$ and $\sk[xy]=2$.}
\label{fig:vpddqq_3b}
\end{figure}

\begin{figure}
  \begin{minipage}{0.45\linewidth} 
\begin{center}
\includegraphics[trim = 40mm 100mm 40mm 90mm,width=1.\textwidth]{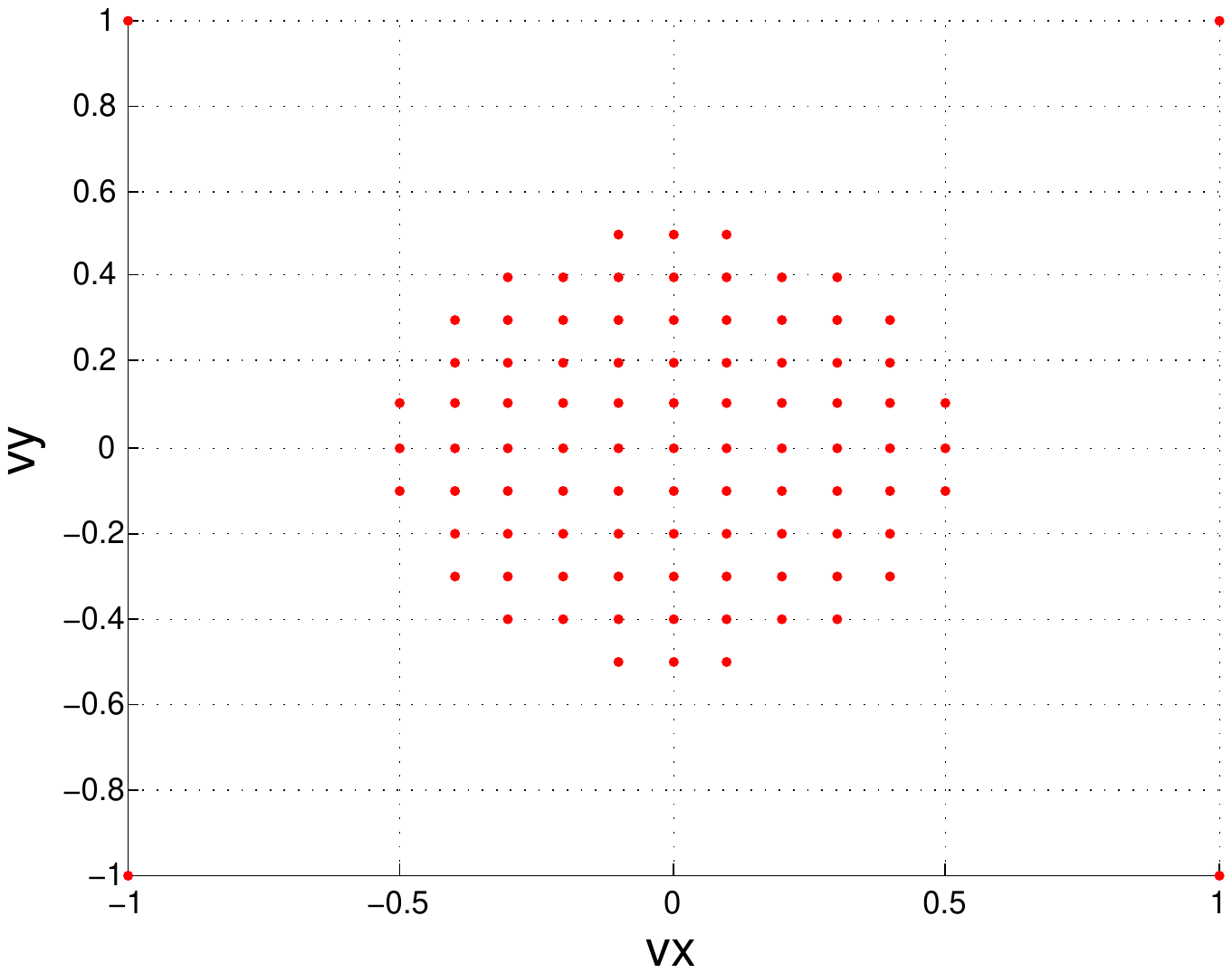}
\end{center}
\end{minipage}\hfill
\begin{minipage}{0.45\linewidth} 
\begin{center}
\includegraphics[trim = 40mm 100mm 40mm 90mm,width=1.\textwidth]{images/d2q4_t_nis_V_12.pdf}
\end{center}
\end{minipage}\\
\caption{Velocities $\vectV$ stable in $\lambda$ scale for the scheme relative to $\utilde=\vectz$ (MRT on the left), $\utilde=\vectV$ (on the right), for a non intrinsic diffusion with $\sk[q]=1.9$ and $\sk[xy]=1$.}
\label{fig:vpddqq_4}
\end{figure}

\begin{figure}
  \begin{minipage}{0.45\linewidth} 
\begin{center}
\includegraphics[trim = 40mm 100mm 40mm 90mm,width=1.\textwidth]{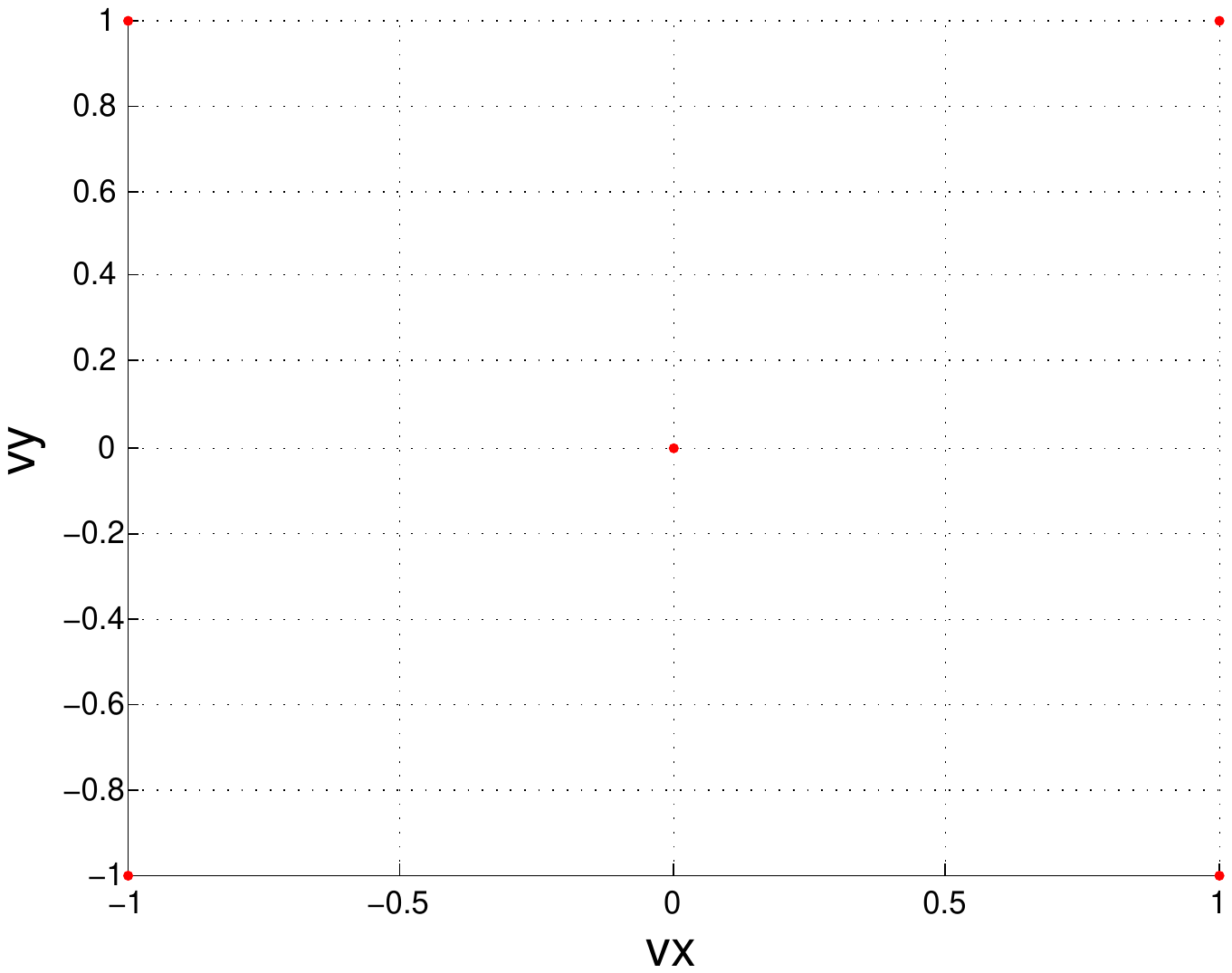}
\end{center}
\end{minipage}\hfill
\begin{minipage}{0.45\linewidth} 
\begin{center}
\includegraphics[trim = 40mm 100mm 40mm 90mm,width=1.\textwidth]{images/d2q4_t_nis_V_12.pdf}
\end{center}
\end{minipage}\\
\caption{Velocities $\vectV$ stable in $\lambda$ scale for the scheme relative to $\utilde=\vectz$ (MRT on the left), $\utilde=\vectV$ (on the right), for a non intrinsic diffusion with $\sk[q]=2$ and $\sk[xy]=1$.}
\label{fig:vpddqq_5}
\end{figure}

Firstly, the draws corresponding to a BGK scheme ($\sk[q]=\sk[xy]=1$) are independent of $\utilde$ (figure \ref{fig:vpddqq_1}). This result was expected because the velocity field $\utilde$ does not appear in the scheme for one single relaxation parameter since (\ref{eq:mueq}) is verified. Secondly, the scheme relative to $\utilde=\vectV$ verifies the necessary condition of stability on the biggest area $\normi\leq\lambda$ (CFL) for all the $\vects$ chosen (figure \ref{fig:vpddqq_1} to \ref{fig:vpddqq_5}). The constance and the optimality of these areas confirm the phenomena observed for the advected spot (tables \ref{table:dispdifix} and \ref{table:dispdifix2}). Finally, the stability areas of the MRT scheme decrease as the two relaxation parameters are moving away from each other. The stability area is reduced to $\vectV=\vectz$ for $\sk[q]=2$ or $\sk[xy]=2$. Whatever the choice of $\vects$, the areas of the MRT scheme are included in those of the scheme relative to $\utilde=\vectV$. These results are also consistent with the test case of the advected spot: the deterioration of the areas are due to the third order dispersive terms of the equivalent equations (proposition \ref{th:ordre3difix}).

\begin{figure}
  \begin{minipage}{0.45\linewidth} 
\begin{center}
\includegraphics[trim = 40mm 100mm 40mm 90mm,width=1.\textwidth]{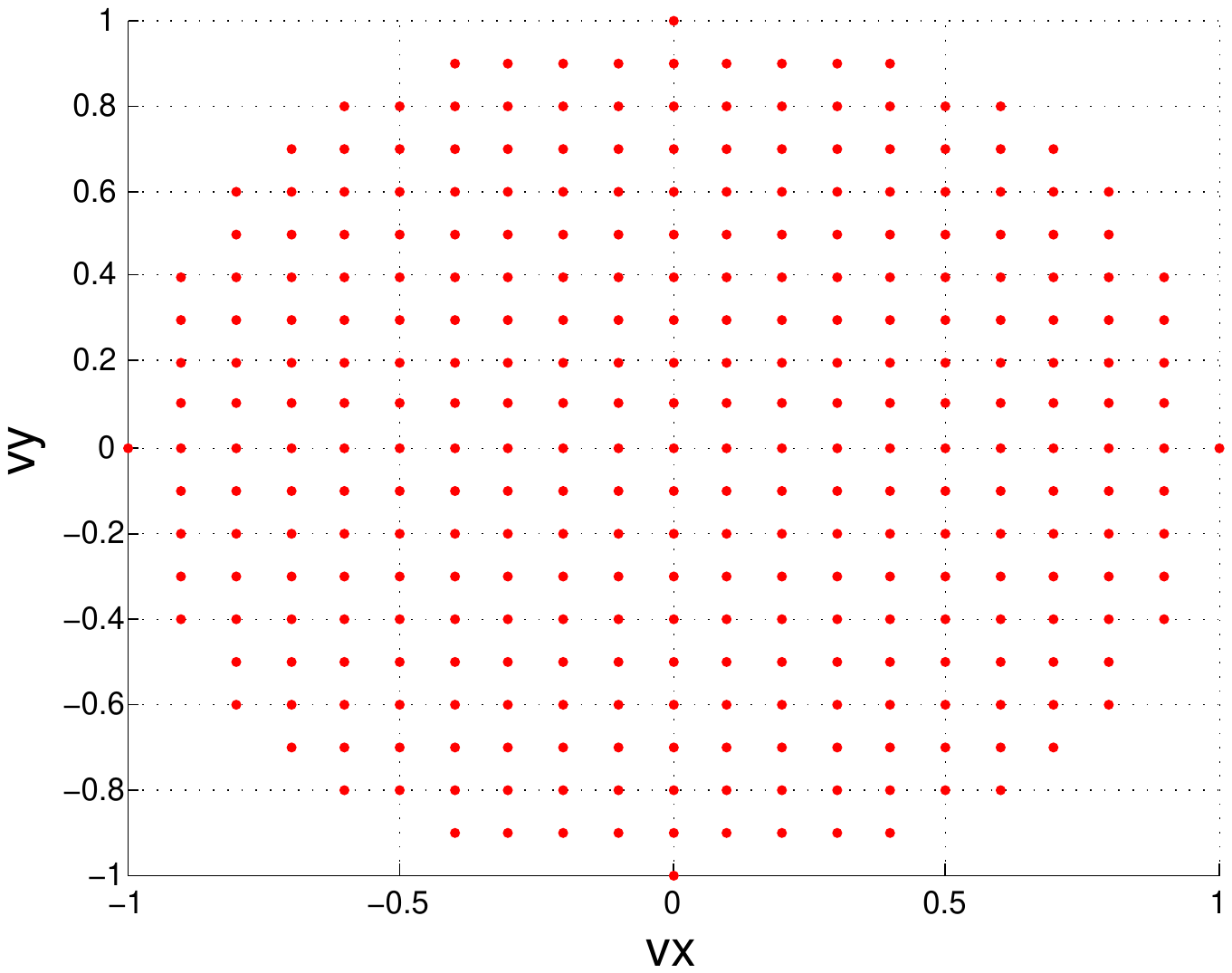}
\end{center}
\end{minipage}\hfill
\begin{minipage}{0.45\linewidth} 
\begin{center}
\includegraphics[trim = 40mm 100mm 40mm 90mm,width=1.\textwidth]{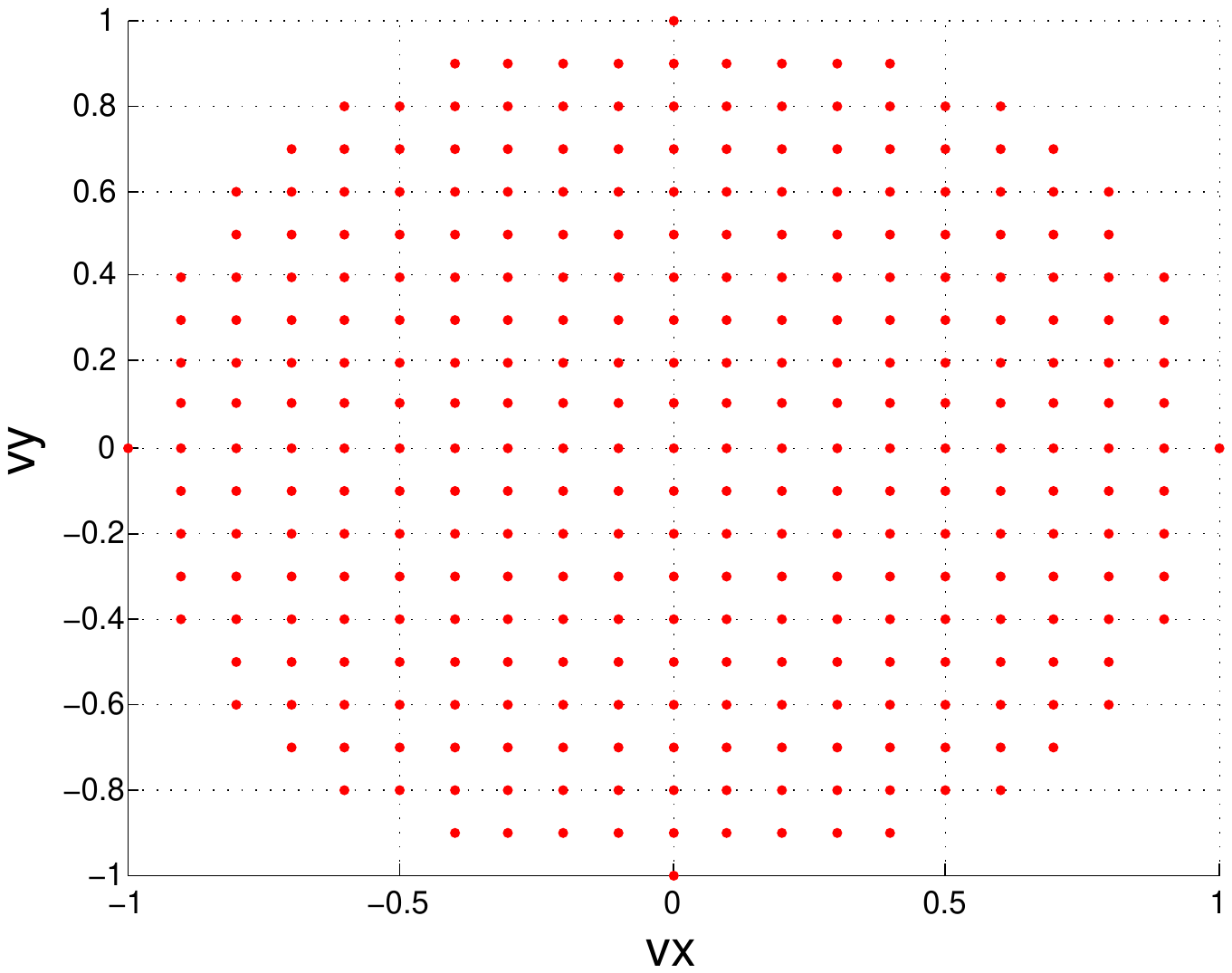}
\end{center}
\end{minipage}\\
\caption{Velocities $\vectV$ stable in $\lambda$ scale for the scheme relative to $\utilde=\vectz$ (MRT on the left), $\utilde=\vectV$ (on the right), for an intrinsic diffusion with $\sk[q]=1$ and $\sk[xy]=1$.}
\label{fig:vpddqq_6}
\end{figure}

\begin{figure}
  \begin{minipage}{0.45\linewidth} 
\begin{center}
\includegraphics[trim = 40mm 100mm 40mm 90mm,width=1.\textwidth]{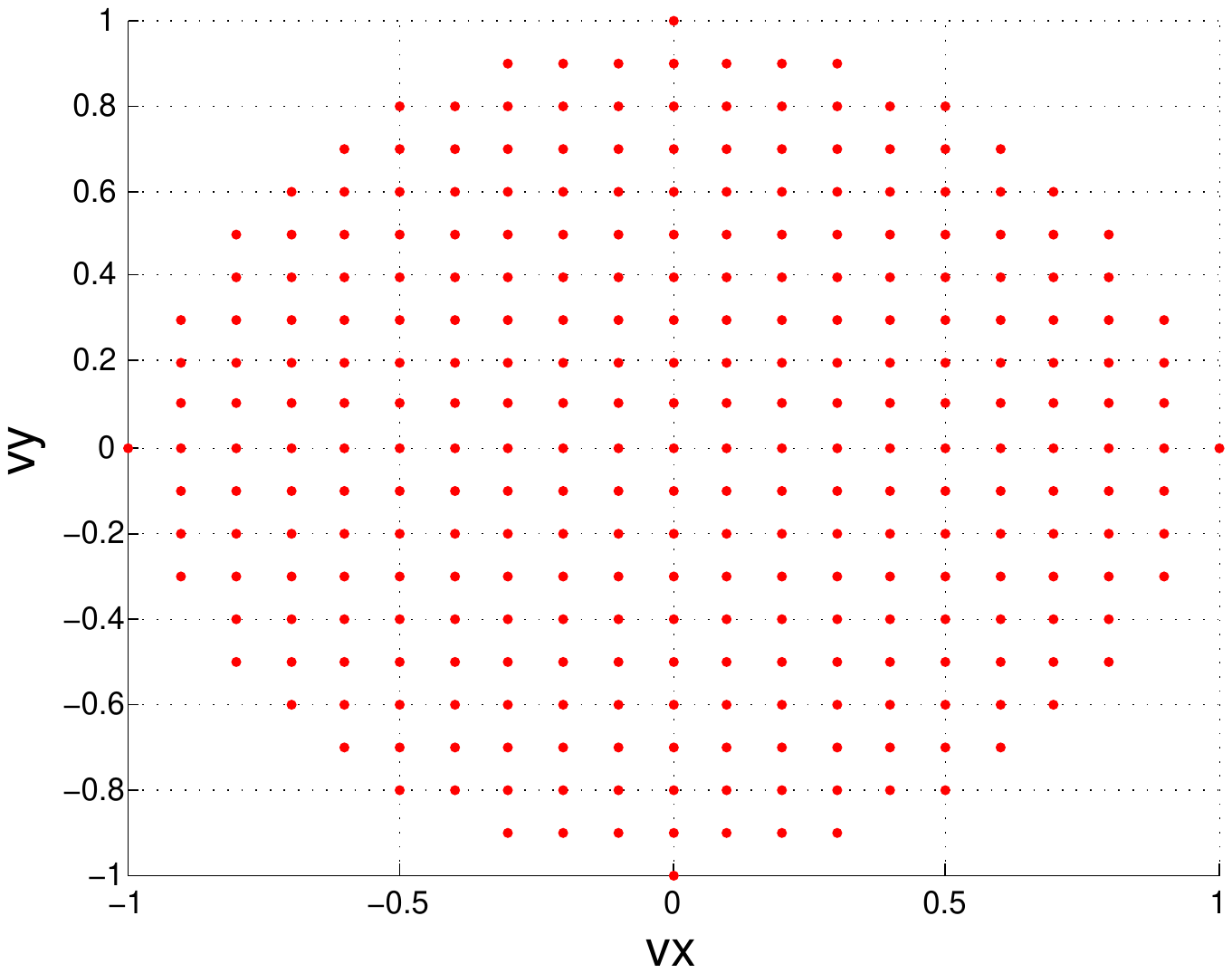}
\end{center}
\end{minipage}\hfill
\begin{minipage}{0.45\linewidth} 
\begin{center}
\includegraphics[trim = 40mm 100mm 40mm 90mm,width=1.\textwidth]{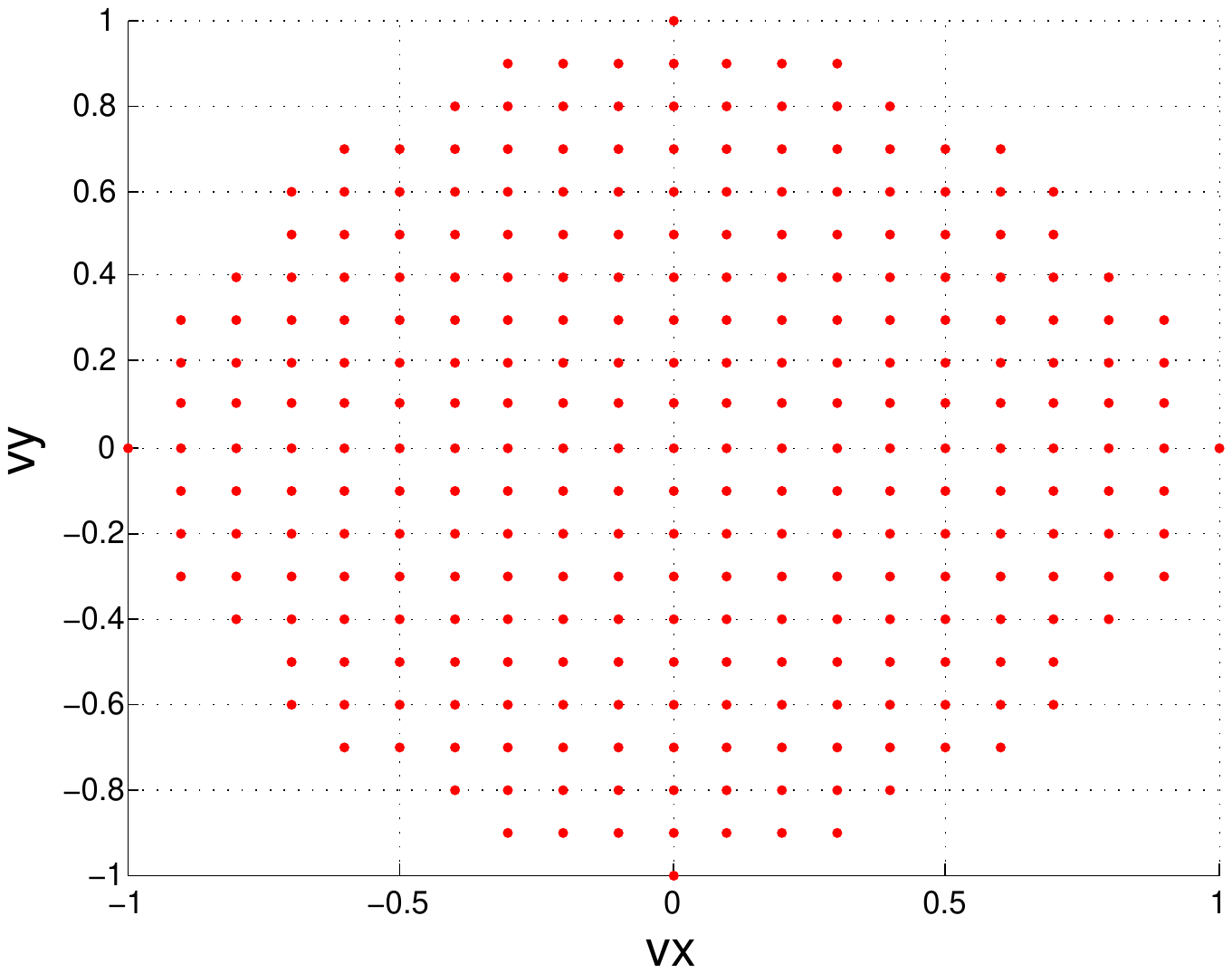}
\end{center}
\end{minipage}\\
\caption{Velocities $\vectV$ stable in $\lambda$ scale for the scheme relative to $\utilde=\vectz$ (MRT on the left), $\utilde=\vectV$ (on the right), for an intrinsic diffusion with $\sk[q]=1$ and $\sk[xy]=1.5$.}
\label{fig:vpddqq_7}
\end{figure}

\begin{figure}
  \begin{minipage}{0.45\linewidth} 
\begin{center}
\includegraphics[trim = 40mm 100mm 40mm 90mm,width=1.\textwidth]{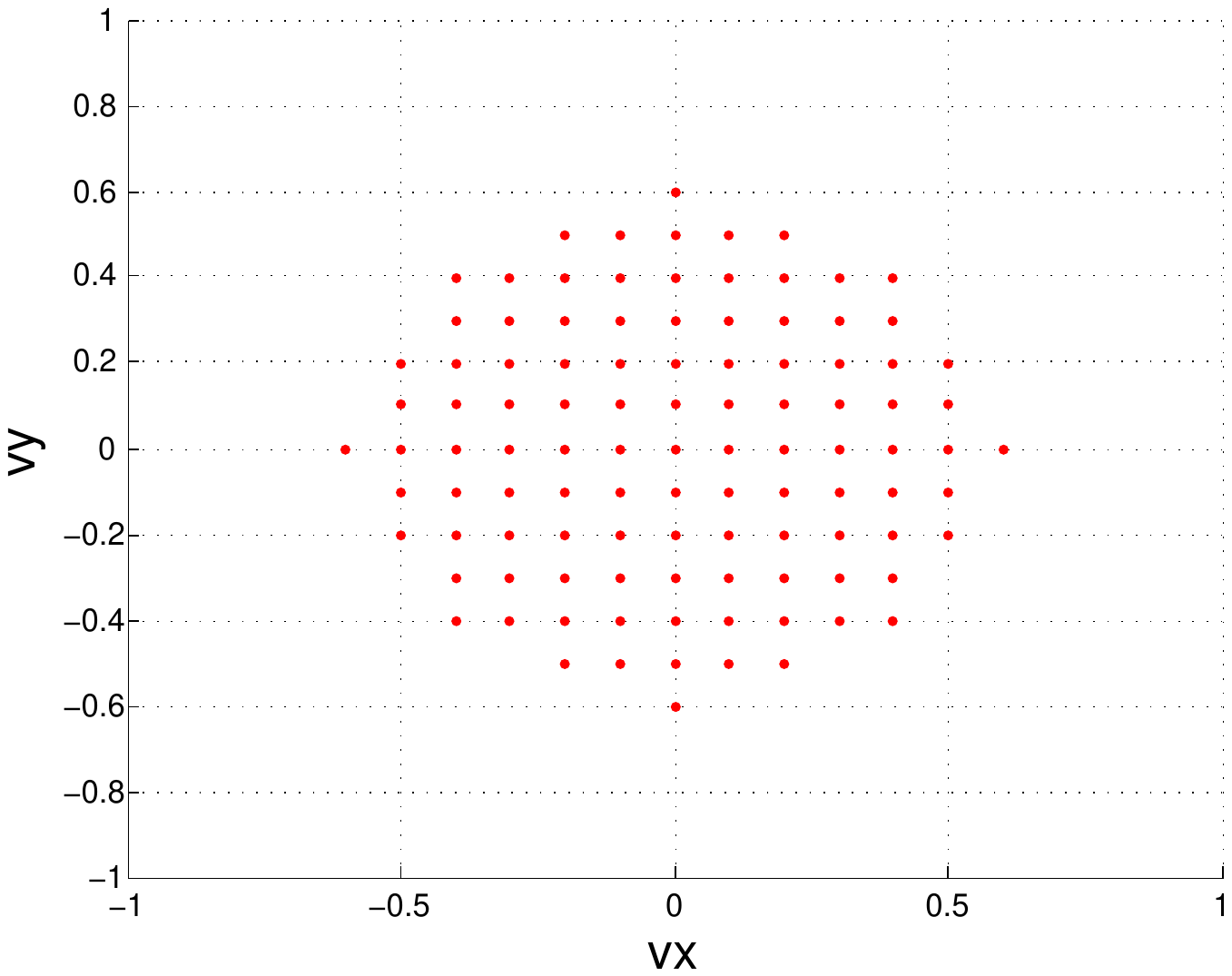}
\end{center}
\end{minipage}\hfill
\begin{minipage}{0.45\linewidth} 
\begin{center}
\includegraphics[trim = 40mm 100mm 40mm 90mm,width=1.\textwidth]{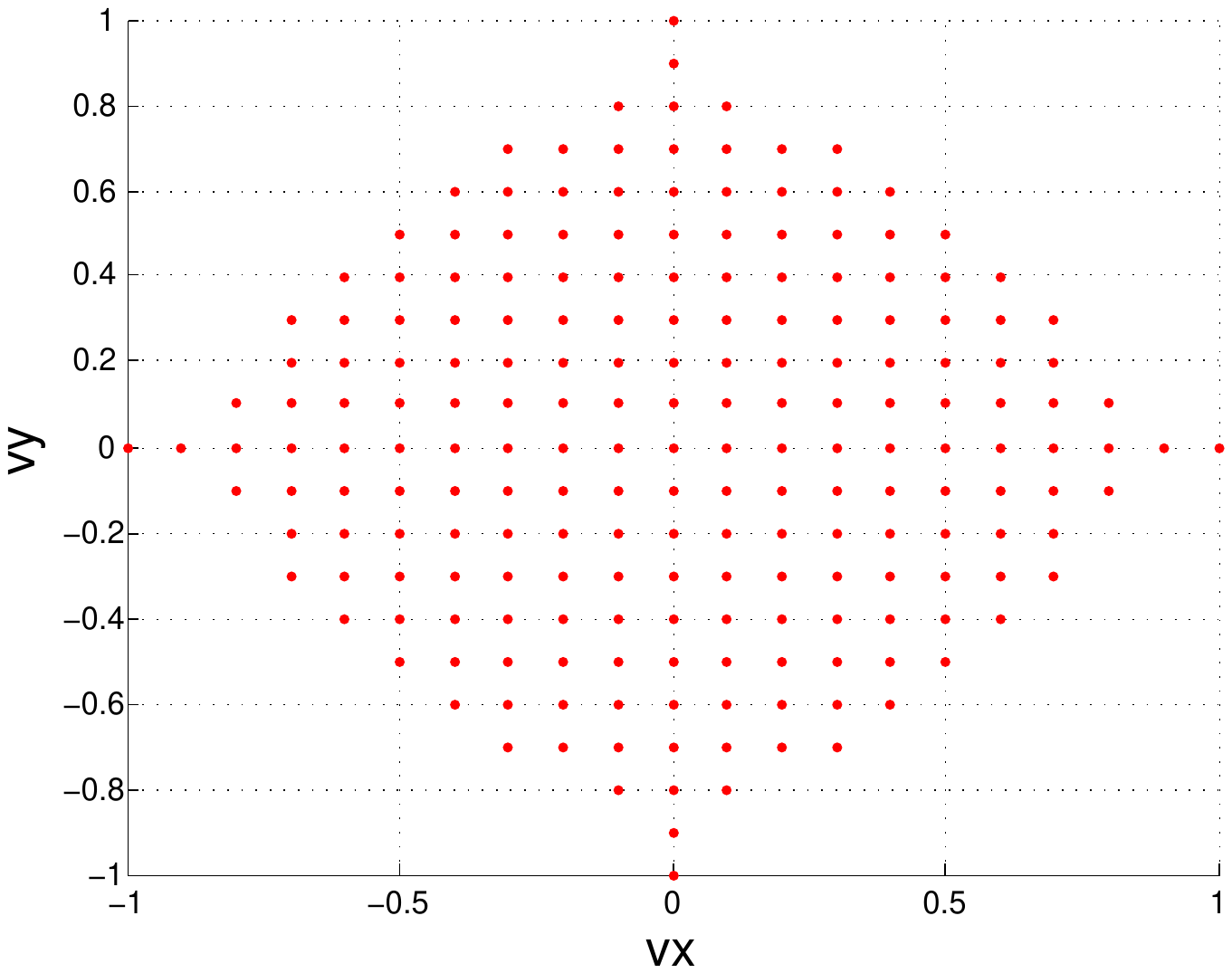}
\end{center}
\end{minipage}\\
\caption{Velocities $\vectV$ stable in $\lambda$ scale for the scheme relative to $\utilde=\vectz$ (MRT on the left), $\utilde=\vectV$ (on the right), for an intrinsic diffusion with $\sk[q]=1$ and $\sk[xy]=1.9$.}
\label{fig:vpddqq_8}
\end{figure}

\begin{figure}
  \begin{minipage}{0.45\linewidth} 
\begin{center}
\includegraphics[trim = 40mm 100mm 40mm 90mm,width=1.\textwidth]{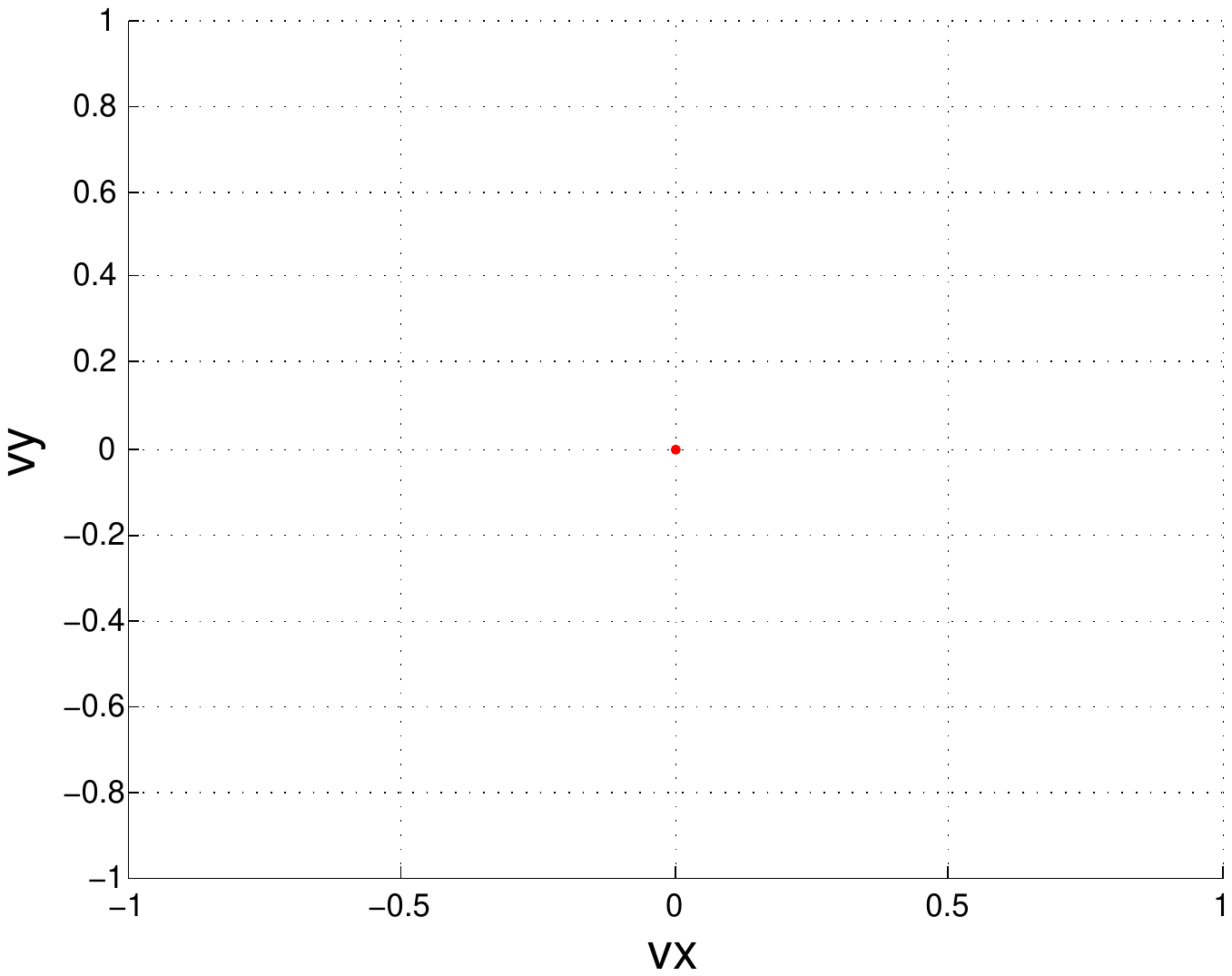}
\end{center}
\end{minipage}\hfill
\begin{minipage}{0.45\linewidth} 
\begin{center}
\includegraphics[trim = 40mm 100mm 40mm 90mm,width=1.\textwidth]{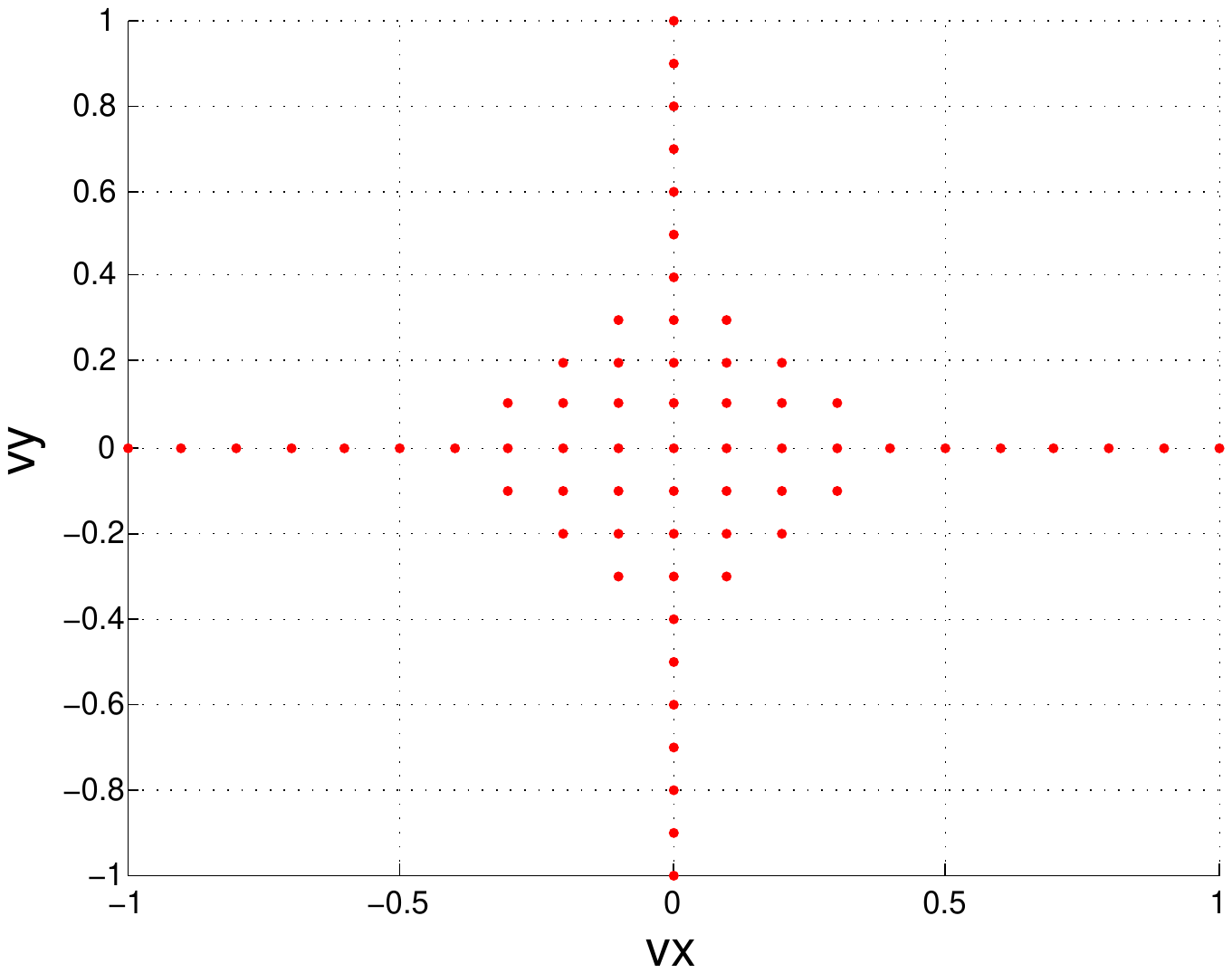}
\end{center}
\end{minipage}\\
\caption{Velocities $\vectV$ stable in $\lambda$ scale for the scheme relative to $\utilde=\vectz$ (MRT on the left), $\utilde=\vectV$ (on the right), for an intrinsic diffusion with $\sk[q]=1$ and $\sk[xy]=2$.}
\label{fig:vpddqq_8b}
\end{figure}

\begin{figure}
  \begin{minipage}{0.45\linewidth} 
\begin{center}
\includegraphics[trim = 40mm 100mm 40mm 90mm,width=1.\textwidth]{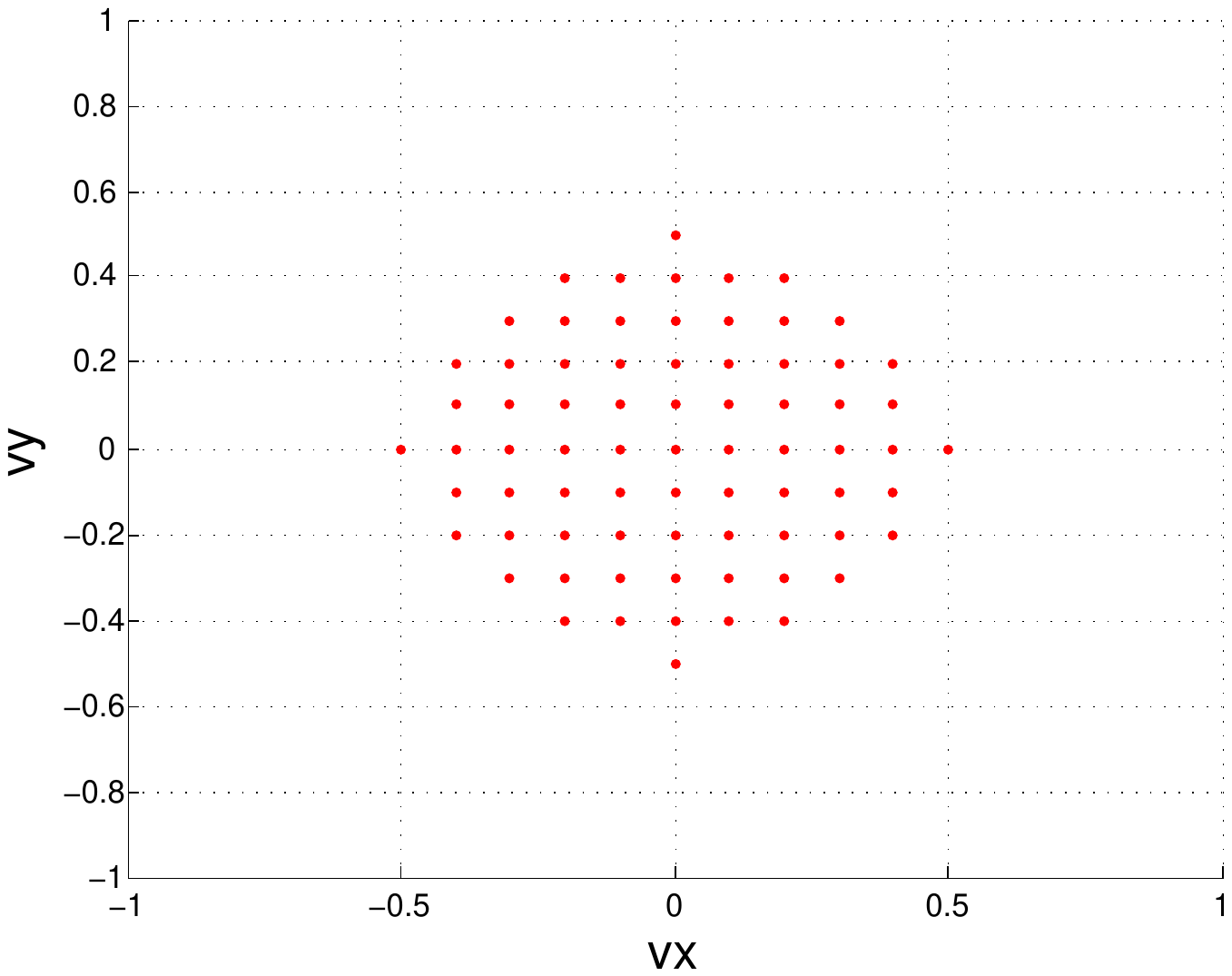}
\end{center}
\end{minipage}\hfill
\begin{minipage}{0.45\linewidth} 
\begin{center}
\includegraphics[trim = 40mm 100mm 40mm 90mm,width=1.\textwidth]{images/d2q4_t_is_V_11,9.pdf}
\end{center}
\end{minipage}\\
\caption{Velocities $\vectV$ stable in $\lambda$ scale for the scheme relative to $\utilde=\vectz$ (MRT on the left), $\utilde=\vectV$ (on the right), for an intrinsic diffusion with $\sk[q]=1.9$ and $\sk[xy]=1$.}
\label{fig:vpddqq_9}
\end{figure}

\begin{figure}
  \begin{minipage}{0.45\linewidth} 
\begin{center}
\includegraphics[trim = 40mm 100mm 40mm 90mm,width=1.\textwidth]{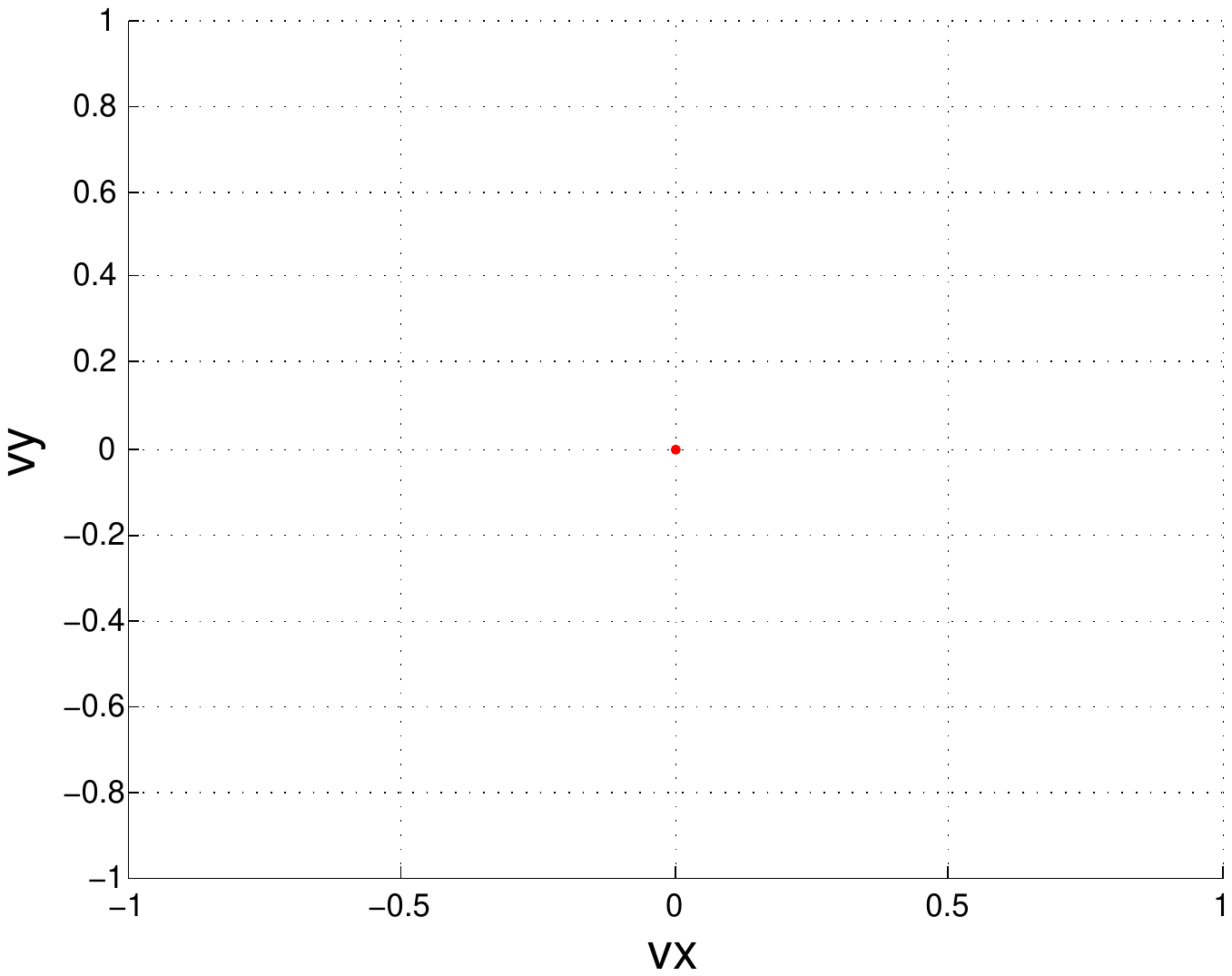}
\end{center}
\end{minipage}\hfill
\begin{minipage}{0.45\linewidth} 
\begin{center}
\includegraphics[trim = 40mm 100mm 40mm 90mm,width=1.\textwidth]{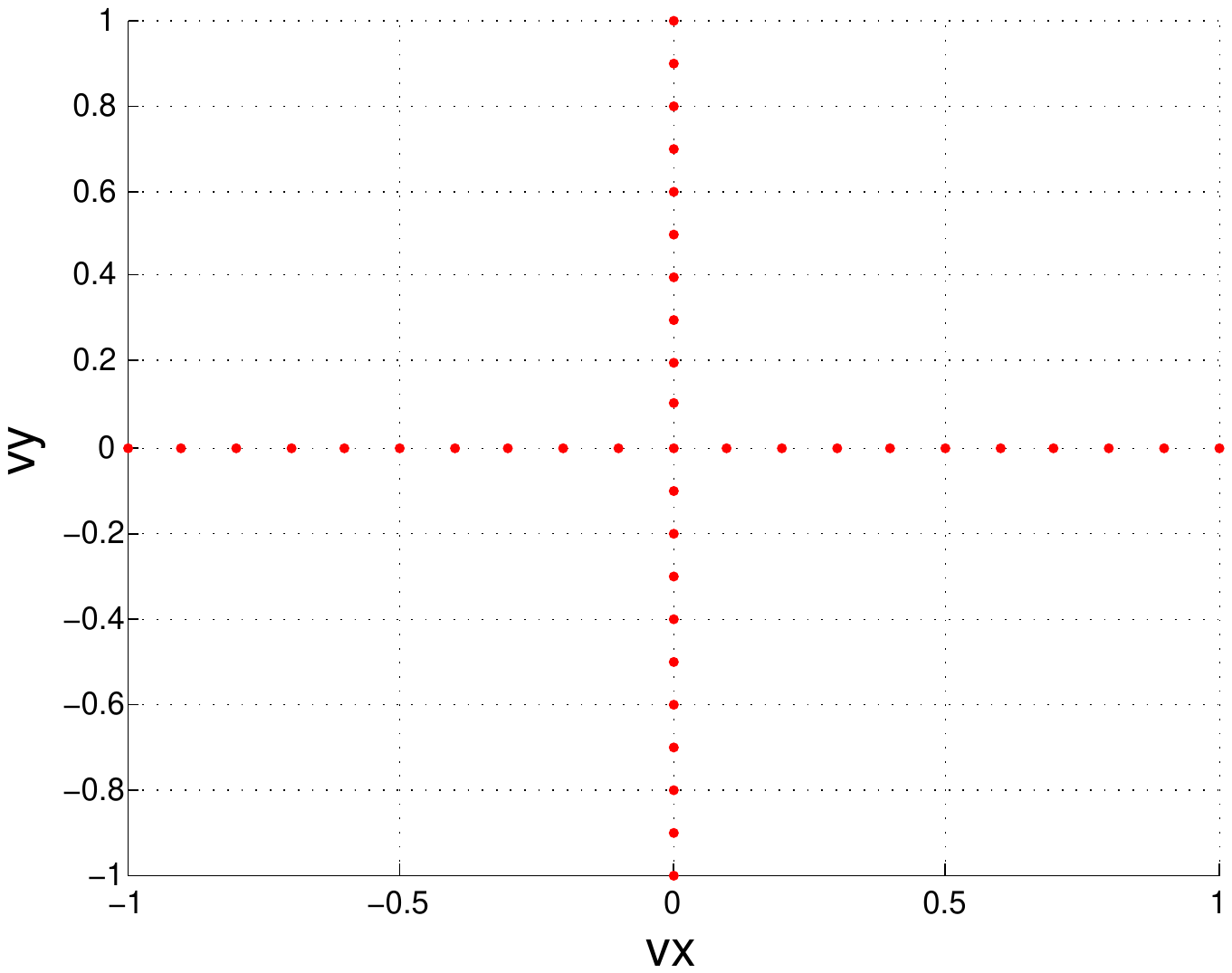}
\end{center}
\end{minipage}\\
\caption{Velocities $\vectV$ stable in $\lambda$ scale for the scheme relative to $\utilde=\vectz$ (MRT on the left), $\utilde=\vectV$ (on the right), for an intrinsic diffusion with $\sk[q]=2$ and $\sk[xy]=1$.}
\label{fig:vpddqq_10}
\end{figure}

For the $\ddqqn$ scheme with an intrinsic diffusion (figures \ref{fig:vpddqq_6} to \ref{fig:vpddqq_10}), the results are different but the general trend is the same. The scheme relative to $\utilde=\vectV$ verifies (\ref{eq:CN}) for larger sets of $\vectV$ than the MRT scheme, excepted when $\sk[q]=1$, $\sk[xy]=1.5$ (figure \ref{fig:vpddqq_7}).
However, the areas associated to $\utilde=\vectV$ are not optimal any more: they decrease when $\sk[q]$ and $\sk[xy]$ move away from each other. Both velocity fields ($\utilde=\vectz$ and $\utilde=\vectV$) undergo this phenomenon but the effect of the dispersion is bigger for the MRT scheme. These results confirm the advected spot ones: the same deterioration appears when $\sig[q]$ tends to $0$ (table \ref{table:dispdifix} and \ref{table:dispdifix2}) to a lesser extent for $\utilde=\vectV$ than for $\utilde=\vectz$.

The main conclusions of this section are the following. Choosing $\utilde=\vectV$ cancels some dispersive terms becoming important as $\vectV$ grows. This cancellation does not occur for $\utilde=\vectz$. We must view these results in parallel with the fact that the scheme is more stable (numerical $L^2$ notion) for $\utilde=\vectV$ than for $\utilde=\vectz$.

\section{Theoretical $\Li$ stability}\label{sub:stabLi}

We study in this section the stability of the relative velocity $\ddqqn$ scheme with respect to a notion of $\Li$ stability presented in \cite{Dell:2013:0,Gra:2014:0}. We use it to demonstrate some maximum principles for $\utilde$ equal to $\vectz$ (MRT scheme) and $\vectV$, the advection velocity (``cascaded like'' scheme). The $\Li$ stability area is fully described in terms of relaxation parameters and advection velocity for these two choices of relative velocity. We show that this $\Li$ notion does not differentiate $\utilde=\vectz$ from $\utilde=\vectV$ in terms of stable behaviour. In each case, the parameters $(\sk[q],\sk[xy])$ corresponding to a non empty area of $\Li$ stability in $\vectV$ are included in the square $[0,2]^2$.

\begin{definition}[$L^{\infty}$ stability]\label{th:stabli}
Let us consider a $\ddqq$ scheme on the cartesian lattice $\mathcal{L}\in\R^d$. The mass at time $t$ is given by 
\begin{equation*}\label{eq:masstot}
\mtot(t)=\sum_{\vectx\in\mathcal{L}}\sum_{j=0}^{q-1}\fj(\vectx,t),\quad t\in\R.
\end{equation*}
Let $C\in\varp{\R}{+}{\star}$, suppose that the initial distributions of particles are nonnegative and that the initial mass is bounded by $C$: 
\begin{equation*}
\fj(\vectx,0)\geqslant0,\quad 0\leq j\leq q-1,~ \vectx\in\mL,\quad and\quad \mtot(0)\leq C,
\end{equation*}
the scheme is said to be $\Li$ stable if for all time $t^n\in\R$, $n\in\N$,
\begin{equation*}
\fj(\vectx,t^n)\geqslant0,\quad 0\leq j\leq q-1,~ \vectx\in\mL,\quad and\quad
\mtot(t^n)\leq C.
\end{equation*}
\end{definition}

The proofs of this {\it a priori} non linear $\Li$ stability notion are only depending on the relaxation phase. Indeed, the transport exchanges the distributions between each other: it does not act on their positivity and on their mass. Thus it is sufficient to show that the postcollision distributions remain nonnegative and the mass bounded.
Note also that if $\mtot(0)$ is bounded, $\mtot(t^n)$ remains bounded for all time $t^n$ because this mass is conserved by the scheme. So there is only to study the positivity of the postcollision distributions functions.

To do so, we express the postcollision distributions of particles $\vectfe$ as a function of the precollision one $\vectf$ for a general $\ddqq$ scheme with one conservation law. These postcollision distributions are first expressed as functions of the postcollision moments thanks to (\ref{eq:mtof}). After applying the relaxation identity (\ref{eq:relaxationu}), the precollision moments are written in the frame of the particle distributions with (\ref{eq:ftomu}). These calculations lead to the following expression of the relaxation.
\begin{multline}\label{eq:relfgen}
\fje[l]=\Big(1-\sum_{k=1}^{q-1}\sk\Mijinvu[lk]\Miju[kl]\Big)\fj[l]-\sum_{\substack{j=0\\j\neq l}}^{q-1}\Big(\sum_{k=1}^{q-1}\sk\Mijinvu[lk]\Miju[kj]\Big)\fj\\+\sum_{k=1}^{q-1}\sk\Mijinvu[lk]\mkueq,\quad 0\leq l\leq q-1. 
\end{multline}

The framework of study is linear because we are approaching a linear advection equation. As a consequence, the equilibrium term of (\ref{eq:relfgen}) is a linear function of the conserved variable $\rho$ and thus of the distributions $\fj$. Consequently, there is a matricial relation between $\vectfe$ and $\vectf$: sufficient conditions of $\Li$ stability are obtained checking when this matrix is nonnegative.

\begin{definition}
For $q\in\N^{\star}$, a square matrix of $\espM(\R)$ is said to be nonnegative if all its coefficients are nonnegative.
\end{definition}

Note that this notion is different from a positive matrix that have all its eigenvalues nonnegative \cite{LasThe:1993:0}. In the following, we prove the results in the twisted case: the analogous ones for the $\ddqqn$ scheme are available in the appendix \ref{se:app4} and use the proposition \ref{th:trstab} of the appendix \ref{se:app3}. We first consider the relative velocity $\ddqqn$ scheme with a non intrinsic diffusion (equilibrium (\ref{eq:eqd2q4tre})). The proofs are quite technical and involves geometric inequalities on $\vectV$ and $\vects$.

\subsection{The non intrinsic diffusion case}

\begin{proposition}[$\Li$ stability areas for the MRT scheme]\label{th:stlitrt0}
Let $\vectV\in\R^2$, $(\sk[q],\sk[xy])\in\R^2$, consider the twisted $\ddqqn$ MRT scheme $(\utilde=\vectz)$ with the relaxation parameters $(0,\sk[q],\sk[q],\sk[xy])$, associated with the equilibrium (\ref{eq:eqd2q4},\ref{eq:eqd2q4tre}) $$\vectmeqz = \rho(1,\Vx,\Vy,\Vx\Vy).$$ Note $\vectVV=(\VVx,\VVy)$ where $\VVx=\Vx+\Vy$, $\VVy=\Vx-\Vy$ and $\gamma=\sk[q]/\sk[xy]$ :
\begin{itemize}
\item if $0<\sk[xy]\leq\min(\sk[q],2-\sk[q])$ (area BCD on the figure \ref{fig:zonzst_dhtw}), the scheme is $\Li$ stable for all $\vectV$ so that 
\begin{subequations}
\begin{gather}
(\VVx\pm2\lambda\gamma)^2-(\VVy)^2\geqslant 4\lambda^2(\gamma^2-1),\label{eq:TRTst1}\\
(\VVy\pm2\lambda\gamma)^2-(\VVx)^2\geqslant 4\lambda^2(\gamma^2-1).\label{eq:TRTst2}
\end{gather}
\end{subequations}
\item if $\sk[q]\leq\sk[xy]\leq2\sk[q]$ and $\sk[q]\leq1$ (area ABC on the figure \ref{fig:zonzst_dhtw}), the scheme is $\Li$ stable for all $\vectV$ so that 
\begin{subequations}
\begin{gather}
(\VVx\pm2\lambda\gamma)^2-(\VVy)^2\geqslant 4\lambda^2(\gamma-1)^2,\label{eq:TRTst3}\\
(\VVy\pm2\lambda\gamma)^2-(\VVx)^2\geqslant 4\lambda^2(\gamma-1)^2.\label{eq:TRTst4}
\end{gather}
\end{subequations}
\item if $2-\sk[q]\leq\sk[xy]\leq2(2-\sk[q])$ and $\sk[q]\geqslant1$ (area ACD on the figure \ref{fig:zonzst_dhtw}), the scheme is $\Li$ stable for all $\vectV$ so that  
\begin{subequations}
\begin{gather}
(\VVx\pm2\lambda\gamma)^2-(\VVy)^2\geqslant  4\lambda^2\Big((\gamma+1)^2-\frac{4}{\sk[xy]}\Big),\label{eq:TRTst5}\\
(\VVy\pm2\lambda\gamma)^2-(\VVx)^2\geqslant  4\lambda^2\Big((\gamma+1)^2-\frac{4}{\sk[xy]}\Big).\label{eq:TRTst6}
\end{gather}
\end{subequations}
\item if $\sk[xy]=0$ and $0<\sk[q]\leq2$ (ray ]BD] on the figure \ref{fig:zonzst_dhtw}), the scheme is $\Li$ stable for $\vectV=\vectz$.
\item if $\sk[xy]=\sk[q]=0$ (point B on the figure \ref{fig:zonzst_dhtw}), the scheme is unconditionally $\Li$ stable.
\item For all other $\vects$, there is no $\vectV$ corresponding to a $\Li$ stable scheme.
\end{itemize}
In particular, the parameters $(\sk[q],\sk[xy])$ corresponding to a non empty area of $\Li$ stability in $\vectV$ are included in the square $[0,2]^2$.
\end{proposition}

The stability areas in the plan $(\sk[q],\sk[xy])$ are represented on the figure \ref{fig:zonzst_dhtw}.The limit cases for the parameters $\sk[q]$ and $\sk[xy]$ lead to consistent inequalities for the velocity conditions.

\begin{figure}
\begin{center}
\definecolor{ffqqtt}{rgb}{1,0,0.2}
\definecolor{ttccqq}{rgb}{0.2,0.8,0}
\definecolor{xdxdff}{rgb}{0.49,0.49,1}
\definecolor{qqttff}{rgb}{0,0.2,1}
\definecolor{uququq}{rgb}{0.25,0.25,0.25}
\definecolor{qqqqff}{rgb}{0,0,1}
\begin{tikzpicture}[scale=0.7,line cap=round,line join=round,>=triangle 45,x=4.0cm,y=4.0cm]
\draw[->,color=black] (-0.5,0) -- (2.5,0);
\foreach \x in {,1,2}
\draw[shift={(\x,0)},color=black] (0pt,2pt) -- (0pt,-2pt) node[below] {\footnotesize $\x$};
\draw[->,color=black] (0,-0.5) -- (0,2.5);
\foreach \y in {,1,2}
\draw[shift={(0,\y)},color=black] (2pt,0pt) -- (-2pt,0pt) node[left] {\footnotesize $\y$};
\draw[color=black] (0pt,-10pt) node[right] {\footnotesize $0$};
\clip(-0.5,-0.5) rectangle (2.5,2.5);
\fill[color=black,fill=black,fill opacity=0.1] (1,2) -- (0,0) -- (1,1) -- cycle;
\fill[color=black,fill=black,fill opacity=0.1] (1,2) -- (1,1) -- (2,0) -- cycle;
\fill[color=black,fill=black,fill opacity=0.1] (0,0) -- (2,0) -- (1,1) -- cycle;
\draw [color=black] (1,2)-- (0,0);
\draw [color=black] (0,0)-- (1,1);
\draw [color=black] (1,1)-- (1,2);
\draw [color=black] (1,2)-- (1,1);
\draw [color=black] (1,1)-- (2,0);
\draw [color=black] (2,0)-- (1,2);
\draw [color=black] (0,0)-- (2,0);
\draw [color=black] (2,0)-- (1,1);
\draw [color=black] (1,1)-- (0,0);
\begin{scriptsize}
\draw (2.35,-0.05) node[anchor=north west] {$\sk[q]$};
\draw (-0.25,2.5) node[anchor=north west] {$\sk[xy]$};
\fill [color=black] (1,2) circle (1.5pt);
\draw[color=black] (1.08,2.13) node {$A$};
\fill [color=black] (0,0) circle (1.5pt);
\draw[color=black] (-0.08,0.13) node {$B$};
\fill [color=black] (1,1) circle (1.5pt);
\draw[color=black] (1.08,1.13) node {$C$};
\fill [color=black] (2,0) circle (1.5pt);
\draw[color=black] (2.08,0.13) node {$D$};
\end{scriptsize}
\end{tikzpicture}
\end{center}
\caption{$\Li$ stability areas in $\vects$ of the twisted $\ddqqn$ MRT scheme with anistropic diffusion.  BCD : (\ref{eq:TRTst1},\ref{eq:TRTst2}), ABC : (\ref{eq:TRTst3},\ref{eq:TRTst4}), ACD :  (\ref{eq:TRTst5},\ref{eq:TRTst6}).} 
\label{fig:zonzst_dhtw}
\end{figure}
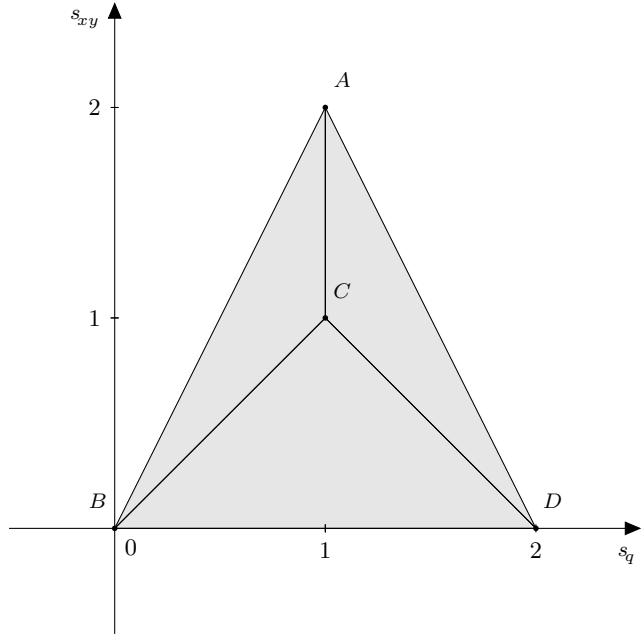

\begin{proof}
The first step consists to express $\vectfe$ as a function of $\vectf$ thanks to the relation (\ref{eq:relfgen}). If the matrix sending $\vectf$ on $\vectfe$ is nonnegative, the scheme is $\Li$ stable.
A formal calculus software guarantees that it corresponds to ensure the following relations:
\begin{equation*}
\min\big(\frac{\sk[xy]}{4},\frac{2\sk[q]-\sk[xy]}{4},1-\frac{2\sk[q]+\sk[xy]}{4}\big)\pm\frac{\sk[q]}{4\lambda}(\Vx+(-1)^i\Vy)+(-1)^i\frac{\sk[xy]}{4\lambda^2}\Vx\Vy\geqslant0,\  i=0,1.
\end{equation*}

If $\sk[xy]=0$, these relations imply $\vectV=\vectz$ and $0\leq\sk[q]\leq2$. Otherwise, we factorise by $\sk[xy]/4\lambda^2$ and express these identities as functions of the coefficient $\gamma=\sk[q]/\sk[xy]$. This leads to
\begin{equation}\label{eq:ineq1}
\lambda^2\min\big(1,2\gamma-1,\frac{4}{\sk[xy]}-2\gamma-1\big)\pm\lambda\gamma(\Vx+(-1)^i\Vy)+(-1)^i\Vx\Vy\geqslant0,\quad i=0,1,
\end{equation}
if $\sk[xy]>0$ or 
\begin{equation*}\label{eq:ineq2}
\lambda^2\max\big(1,2\gamma-1,\frac{4}{\sk[xy]}-2\gamma-1\big)\pm\lambda\gamma(\Vx+(-1)^i\Vy)+(-1)^i\Vx\Vy\leq0,\quad i=0,1,
\end{equation*}
if $\sk[xy]<0$.
These inequalities are expressed thanks to quadratic forms in $\vectV$. To characterize the associated geometry, we write them in an adapted basis. To do so, we use the variable change $$\Vx\Vy=\frac{(\VVx)^2-(\VVy)^2}{4},\quad {\rm for}\quad \VVx=\Vx+\Vy,~ \VVy=\Vx-\Vy.$$
The inequalities (\ref{eq:ineq1}) then read
\begin{equation*}
4\lambda^2\min\big(1,2\gamma-1,\frac{4}{\sk[xy]}-2\gamma-1\big)\pm4\lambda\gamma\VVx+(\VVx)^2-(\VVy)^2\geqslant0,
\end{equation*}
completed by the ones obtained when the roles of $\VVx$ and $\VVy$ are exchanged.
We then center them in the adapted frame to obtain if $\sk[xy]>0$

\begin{equation}\label{eq:trtrtmax}
(\VVx\pm2\lambda\gamma)^2-(\VVy)^2\geqslant 4\lambda^2\max\big(\gamma^2-1,(\gamma-1)^2,(\gamma+1)^2-\frac{4}{\sk[xy]}\big),
\end{equation}
plus the inequalities obtained exchanging the roles of $\VVx$ and $\VVy$.
The case $\sk[xy]<0$ yields to
\begin{equation}\label{eq:trtrtmax2}
(\VVx\pm2\lambda\gamma)^2-(\VVy)^2\leq 4\lambda^2\min\big(\gamma^2-1,(\gamma-1)^2,(\gamma+1)^2-\frac{4}{\sk[xy]}\big).
\end{equation}

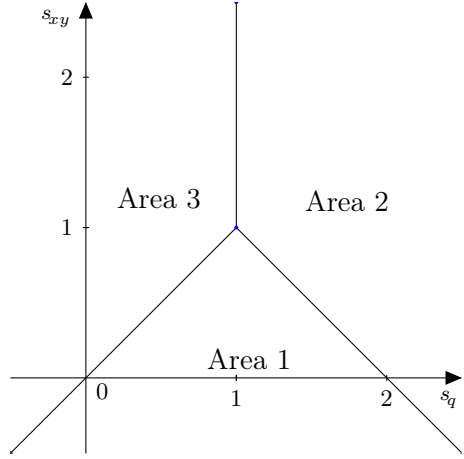
\begin{figure}
\begin{center}
\definecolor{xdxdff}{rgb}{0.49,0.49,1}
\definecolor{uququq}{rgb}{0.25,0.25,0.25}
\definecolor{qqqqff}{rgb}{0,0,1}
\begin{tikzpicture}[scale=0.5,line cap=round,line join=round,>=triangle 45,x=4.0cm,y=4.0cm]
\draw[->,color=black] (-0.5,0) -- (2.5,0);
\foreach \x in {,1,2}
\draw[shift={(\x,0)},color=black] (0pt,2pt) -- (0pt,-2pt) node[below] {\footnotesize $\x$};
\draw[->,color=black] (0,-0.5) -- (0,2.5);
\foreach \y in {,1,2}
\draw[shift={(0,\y)},color=black] (2pt,0pt) -- (-2pt,0pt) node[left] {\footnotesize $\y$};
\draw[color=black] (0pt,-10pt) node[right] {\footnotesize $0$};
\clip(-0.5,-0.5) rectangle (2.5,2.5);
\draw (-0.5,-0.5)-- (1,1);
\draw (2.5,-0.5)-- (1,1);
\draw (1,1)-- (1,2.5);
\draw (0.74,0.26) node[anchor=north west] {Area 1};
\draw (1.39,1.31) node[anchor=north west] {Area 2};
\draw (0.14,1.33) node[anchor=north west] {Area 3};
\begin{scriptsize}
\draw (2.3,-0.03) node[anchor=north west] {$\sk[q]$};
\draw (-0.35,2.5) node[anchor=north west] {$\sk[xy]$};
\fill [color=qqqqff] (1,1) circle (1.5pt);
\fill [color=uququq] (-0.5,-0.5) circle (1.5pt);
\fill [color=xdxdff] (2.5,-0.5) circle (1.5pt);
\fill [color=qqqqff] (1,2.5) circle (1.5pt);
\end{scriptsize}
\end{tikzpicture}
\end{center}
\caption{Value of the maximum of (\ref{eq:trtrtmax}). Area $1$ : $\gamma^2-1$, area $2$ : $(\gamma+1)^2-4/\sk[xy]$, area $3$ : $(\gamma-1)^2$.} 
\label{fig:zonemax_dhtw}
\end{figure}

We begin by showing that there is no velocity stable in the case $\sk[xy]<0$. The minimum of (\ref{eq:trtrtmax2}) is equal to $(\gamma-1)^2$ (nonnegative) if $\sk[q]\leq\sk[xy]<0$, to $\gamma^2-1$ (negative) if $\sk[xy]\leq\min(\sk[q],2-\sk[q])$ and $(\gamma+1)^2-4/\sk[xy]$ (nonnegative) if $2-\sk[q]\leq\sk[xy]<0$. In the first and third cases, (\ref{eq:trtrtmax2})  would have a solution if and only if the left pole of the hyperbole centered in $2\lambda|\gamma|$ is negative. We then should have
$$2\lambda|\gamma|-2\lambda|\gamma-1|\leq0,$$ in the first case and 
$$2\lambda|\gamma|-2\lambda((\gamma+1)^2-4/\sk[xy])^{\frac{1}{2}}\leq0,$$ in the third one.
Considering that $\sk[xy]<0$, the first case is equivalent to $\sk[xy]\leq2\sk[q]$ that has no intersection with $\sk[q]\leq\sk[xy]<0$ where $(\gamma-1)^2$ is the minimum.
The third one is equivalent to $\sk[xy]\leq2(2-\sk[q])$ whose intersection with $2-\sk[q]\leq\sk[xy]<0$ is empty. It remains the case 
\begin{equation*}
(\VVy)^2-(\VVx\pm2\lambda\gamma)^2\geqslant 4\lambda^2(1-\gamma^2),
\end{equation*}
when $\sk[xy]\leq\min(\sk[q],2-\sk[q])$. These inequations combined with the ones obtained exchanging the roles of $\VVx$ and $\VVy$ have no common intersection. This closes the case $\sk[xy]<0$.

We exhibit now the conditions for which (\ref{eq:trtrtmax}) has at least a solution $\vectV$. 
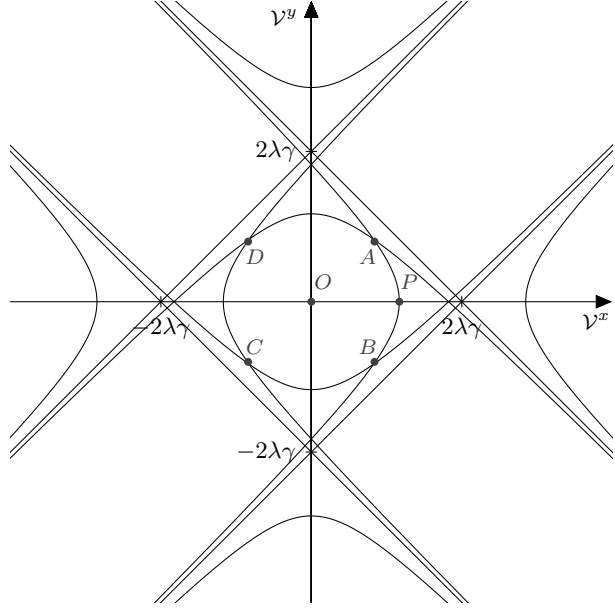
\begin{figure}
\begin{center}
\definecolor{qqqqcc}{rgb}{0,0,0.8}
\definecolor{ttffff}{rgb}{0.2,1,1}
\definecolor{uququq}{rgb}{0.25,0.25,0.25}
\definecolor{qqccqq}{rgb}{0,0.8,0}
\definecolor{ffqqqq}{rgb}{1,0,0}
\definecolor{qqqqff}{rgb}{0,0,1}
\definecolor{xdxdff}{rgb}{0.49,0.49,1}
\begin{tikzpicture}[scale=1.,line cap=round,line join=round,>=triangle 45,x=1.0cm,y=1.0cm]
\draw[->,color=black] (-4,0) -- (4,0);
\foreach \x in {-5,-4,-3,-2,-1,1,2,3,4}
\draw[->,color=black] (0,-4) -- (0,4);
\foreach \y in {-5,-4,-3,-2,-1,1,2,3,4}
\draw[color=black] (0pt,-10pt) node[right] {};
\draw[shift={(2,0)},color=black] (0pt,2pt) -- (0pt,-2pt) node[below] {\footnotesize $2\lambda\gamma$};
\draw[shift={(-2,0)},color=black] (0pt,2pt) -- (0pt,-2pt) node[below] {\footnotesize $-2\lambda\gamma$};
\draw[shift={(0,2)},color=black] (2pt,0pt) -- (-2pt,0pt) node[left] {\footnotesize $2\lambda\gamma$};
\draw[shift={(0,-2)},color=black] (2pt,0pt) -- (-2pt,0pt) node[left] {\footnotesize $-2\lambda\gamma$};
\clip(-4,-4) rectangle (4,4);
\draw [domain=-5:5] plot(\x,{(-4--2*\x)/-2});
\draw [domain=-5:5] plot(\x,{(-4--2*\x)/2});
\draw [domain=-5:5] plot(\x,{(--4--2*\x)/-2});
\draw [domain=-5:5] plot(\x,{(-4-2*\x)/-2});
\draw [samples=50,domain=-0.99:0.99,rotate around={0:(2.01,0)},xshift=2.01cm,yshift=0cm,color=black] plot ({0.84*(1+\x*\x)/(1-\x*\x)},{0.84*2*\x/(1-\x*\x)});
\draw [samples=50,domain=-0.99:0.99,rotate around={0:(2.01,0)},xshift=2.01cm,yshift=0cm,color=black] plot ({0.84*(-1-\x*\x)/(1-\x*\x)},{0.84*(-2)*\x/(1-\x*\x)});
\draw [samples=50,domain=-0.99:0.99,rotate around={180:(-2.01,0)},xshift=-2.01cm,yshift=0cm,color=black] plot ({0.84*(1+\x*\x)/(1-\x*\x)},{0.84*2*\x/(1-\x*\x)});
\draw [samples=50,domain=-0.99:0.99,rotate around={180:(-2.01,0)},xshift=-2.01cm,yshift=0cm,color=black] plot ({0.84*(-1-\x*\x)/(1-\x*\x)},{0.84*(-2)*\x/(1-\x*\x)});
\draw [samples=50,domain=-0.99:0.99,rotate around={90:(0,2.01)},xshift=0cm,yshift=2.01cm,color=black] plot ({0.84*(1+\x*\x)/(1-\x*\x)},{0.84*2*\x/(1-\x*\x)});
\draw [samples=50,domain=-0.99:0.99,rotate around={90:(0,2.01)},xshift=0cm,yshift=2.01cm,color=black] plot ({0.84*(-1-\x*\x)/(1-\x*\x)},{0.84*(-2)*\x/(1-\x*\x)});
\draw [samples=50,domain=-0.99:0.99,rotate around={-90:(0,-2.01)},xshift=0cm,yshift=-2.01cm,color=black] plot ({0.84*(1+\x*\x)/(1-\x*\x)},{0.84*2*\x/(1-\x*\x)});
\draw [samples=50,domain=-0.99:0.99,rotate around={-90:(0,-2.01)},xshift=0cm,yshift=-2.01cm,color=black] plot ({0.84*(-1-\x*\x)/(1-\x*\x)},{0.84*(-2)*\x/(1-\x*\x)});
\begin{scriptsize}
\fill [color=uququq] (0,0) circle (1.5pt);
\draw[color=uququq] (0.16,0.26) node {$O$};
\fill [color=uququq] (1.17,0) circle (1.5pt);
\draw[color=uququq] (1.3,0.26) node {$P$};
\fill [color=uququq] (0.84,0.8) circle (1.5pt);
\draw[color=uququq] (0.75,0.6) node {$A$};
\fill [color=uququq] (0.84,-0.8) circle (1.5pt);
\draw[color=uququq] (0.75,-0.6) node {$B$};
\fill [color=uququq] (-0.84,-0.8) circle (1.5pt);
\draw[color=uququq] (-0.75,-0.6) node {$C$};
\fill [color=uququq] (-0.84,0.8) circle (1.5pt);
\draw[color=uququq] (-0.75,0.6) node {$D$};
\draw (3.5,-0.05) node[anchor=north west] {$\VVx$};
\draw (-0.65,4.) node[anchor=north west] {$\VVy$};
\end{scriptsize}
\end{tikzpicture}
\end{center}
\caption{Layout of the bounds of the stability areas in $\vectV$ for the twisted $\ddqqn$ MRT scheme.} 
\label{fig:zonehyp_dhtw}
\end{figure}
The eventual stability area is located between the two hyperboles
\begin{equation*}
(\VVx\pm2\lambda\gamma)^2-(\VVy)^2= 4\lambda^2\max\big(\gamma^2-1,(\gamma-1)^2,(\gamma+1)^2-\frac{4}{\sk[xy]}\big),
\end{equation*}
plus the ones obtained exchanging the roles of $\VVx$ and $\VVy$. Let us note that this maximum is always nonnegative because $(\gamma-1)^2$ is nonnegative. This area is delimited by the points $A,B,C,D$ on the figure \ref{fig:zonehyp_dhtw}.
The inequalities (\ref{eq:trtrtmax}) have a solution $\vectV$ if and only if the left pole $P$ of the hyperbole centered in $2\lambda\gamma$ has a nonnegative abscissa (figure \ref{fig:zonehyp_dhtw}).
When the maximum is $\gamma^2-1$, this abscissa is $2\lambda(\gamma-|\gamma^2-1|^{\frac{1}{2}})$, that is nonnegative if $\sk[q]\geqslant0$. 
When it is $(\gamma-1)^2$, the abscissa $2\lambda(2\gamma-1)$ is nonnegative if and only if $\sk[xy]\leq2\sk[q]$.
Finally when the maximum is $(\gamma+1)^2-4/\sk[xy]$, the abscissa is equal to $2\lambda(\gamma-((\gamma+1)^2-4/\sk[xy])^{\frac{1}{2}})$. The positivity of this coefficient is equivalent to $\sk[xy]\leq2(2-\sk[q])$.

It remains to determine the maximum of (\ref{eq:trtrtmax}) in the case where there is at least one velocity $\vectV$ stable: three cases, represented on the figure \ref{fig:zonemax_dhtw}, appear. If $\sk[xy]\leq\sk[q]\leq2-\sk[xy]$, the maximum is $\gamma^2-1$, if $\sk[xy]\geqslant2-\sk[q]$ and  $\sk[q]\geqslant1$, it is $(\gamma+1)^2-4/\sk[xy]$, if $\sk[q]\leq\min(\sk[xy],1)$, it is $(\gamma-1)^2$.
\end{proof}

From the proposition \ref{th:stlitrt0}, we deduce the simpler case of the BGK scheme corresponding to one single relaxation parameter.
 
\begin{corollary}[The BGK case]\label{th:stlibgk0}
Let $\vectV\in\R^2$, $\utilde\in\R^2$, $s\in\R$, consider the twisted $\ddqqn$ scheme relative to $\utilde$, BGK of parameter $s$, associated with the equilibrium (\ref{eq:eqd2q4},\ref{eq:eqd2q4tre}) $$\vectmequ =  \rho\big(1,\Vx-\utx,\Vy-\uty,(\Vx-\utx)(\Vy-\uty)\big).$$ 
Note $\vectVV=(\VVx,\VVy)$ where $\VVx=\Vx+\Vy$, $\VVy=\Vx-\Vy$:
\begin{itemize}
\item if $0\leq s\leq1$, the scheme is $\Li$ stable for $\normi\leq\lambda$.
\item if $1\leq s\leq4/3$,  the scheme is $\Li$ stable for all $\vectV$ so that
\begin{subequations}
\begin{gather}
(\VVx\pm2\lambda)^2-(\VVy)^2\geqslant  16\lambda^2\Big(1-\frac{1}{s}\Big),\label{eq:BGKst1}\\
(\VVy\pm2\lambda)^2-(\VVx)^2\geqslant  16\lambda^2\Big(1-\frac{1}{s}\Big).\label{eq:BGKst2}
\end{gather}
\end{subequations}
\item If $s>4/3$, no $\vectV$ corresponds to a $\Li$ stable scheme.
\end{itemize}
In particular, the parameters $s$ corresponding to a non empty area of $\Li$ stability in $\vectV$ are included in $[0,2]$.
\end{corollary}

\begin{remarque}
The BGK scheme does not depend on the velocity field $\utilde$ since the relation (\ref{eq:mueq}) is verified.
\end{remarque}

\begin{proof}
In this context, the parameter $\gamma$ is equal to $1$. Thus the inequalities (\ref{eq:trtrtmax}) summarize in 
\begin{equation*}
(\VVx\pm2\lambda)^2-(\VVy)^2\geqslant  \max\big(0,16\lambda^2\big(1-\frac{1}{s}\big)\big),
\end{equation*}
plus the inequality obtained exchanging the roles of $\VVx$ and $\VVy$.
If $s\leq1$, it is equivalent to
\begin{equation*}
(\VVx\pm2\lambda)^2-(\VVy)^2\geqslant 0,
\end{equation*}
and so to $\normi\leq\lambda$, else we obtain (\ref{eq:BGKst1},\ref{eq:BGKst2}).

The inequalities (\ref{eq:BGKst1},\ref{eq:BGKst2}) admit a non empty set of stable velocities $\vectV$ if the abscissa of the left pole of the hyperbole $$(\VVx-2\lambda)^2-(\VVy)^2=16\lambda^2\Big(1-\frac{1}{s}\Big),$$ is nonnegative. This abscissa, equal to $2\lambda(1-2(1-1/s)^{\frac{1}{2}})$ for $s\geqslant1$, is nonnegative if and only if $s\leq4/3$.
\end{proof}

For $s\leq1$, the BGK scheme is optimal in terms of $\Li$ stability (CFL reached). When $\sk[q]$ and $\sk[xy]$ go far from each other (TRT scheme), the stability area decreases until it reaches $\vectV=\vectz$ at the boundaries of the triangle represented on the figure \ref{fig:zonzst_dhtw}. There is no stable velocities outside of this triangle. 

We now focus on the scheme relative to $\utilde=\vectV$ with a non intrinsic diffusion and two relaxation parameters. We show that the optimal area of stability spreads on some TRT schemes contrary to the MRT scheme.

\begin{proposition}[$\Li$ stability areas for a relative velocity scheme]\label{th:stlitrtv}
Let $\vectV\in\R^2$, $(\sk[q],\sk[xy])\in\R^2$, consider the twisted $\ddqqn$ scheme relative to $\utilde=\vectV$ associated with the relaxation parameters $(0,\sk[q],\sk[q],\sk[xy])$, and the equilibrium (\ref{eq:eqd2q4},\ref{eq:eqd2q4tre}) $$\vectmeq(\vectV) = \rho(1,0,0,0).$$ Let us note $\vectVV=(\VVx,\VVy)$ where $\VVx=\Vx+\Vy$, $\VVy=\Vx-\Vy$ and $\gamma=\sk[xy]/(2\sk[q]-\sk[xy])$, then: 
\begin{itemize}
\item if $\sk[q]\leq\sk[xy]\leq \min(1,2\sk[q])$ (area ABC on the figure \ref{fig:zonzst_rectw}), the scheme is $\Li$ stable on 
\begin{equation}\label{eq:carrei}
\normi\leq\lambda.
\end{equation} 
\item if $\sk[xy]<2\sk[q]$, $\max(\sk[q],1)\leq\sk[xy]\leq2(2-\sk[q])$ (area BCED on the figure \ref{fig:zonzst_rectw}), the scheme is $\Li$ stable on the intersection of (\ref{eq:carrei}) with 
\begin{subequations} \label{eq:ega}
\begin{gather}
(\VVx\pm2\lambda\gamma)^2-(\VVy)^2\geqslant\frac{16\lambda^2}{(2\sk[q]-\sk[xy])^2}(\sk[xy]-\sk[q](2-\sk[q])),\label{eq:ega3}\\
(\VVy\pm2\lambda\gamma)^2-(\VVx)^2\geqslant\frac{16\lambda^2}{(2\sk[q]-\sk[xy])^2}(\sk[xy]-\sk[q](2-\sk[q]))\label{eq:ega32}.
\end{gather} 
\end{subequations}
\item if $0\leq\sk[xy]\leq\min(\sk[q],\sk[q](2-\sk[q]))$ (area ACF on the figure \ref{fig:zonzst_rectw}), the scheme is $\Li$ stable on
\begin{equation}\label{eq:carreig}
\normi\leq\lambda\gamma.
\end{equation}
\item if $\sk[q](2-\sk[q])\leq\sk[xy]\leq\min(\sk[q],2(2-\sk[q])) $(area CEF on the figure \ref{fig:zonzst_rectw}), the scheme is $\Li$ stable on (\ref{eq:ega3},\ref{eq:ega32}).
\item For all other $\vects$, no $\vectV$ corresponds to a $\Li$ stable scheme.
\item if $\sk[xy]=2\sk[q]$ and $0<\sk[xy]\leq2$ (ray ]AD] on the figure \ref{fig:zonzst_rectw}), the scheme is $\Li$ stable on the intersection of (\ref{eq:carrei}) and 
\begin{equation}\label{eq:nor1s}
\normu\leq2\lambda\Big(\frac{2-\sk[xy]}{\sk[xy]}\Big).
\end{equation}

\item if $\sk[xy]=\sk[q]=0$ (A on the figure \ref{fig:zonzst_rectw}), the scheme is unconditionally $\Li$ stable.
\end{itemize}
In particular, the parameters $(\sk[q],\sk[xy])$ corresponding to a non empty area of $\Li$ stability in $\vectV$ are included in the square $[0,2]^2$.
\end{proposition}

\begin{figure}
\begin{center}
\definecolor{zzttqq}{rgb}{0.6,0.2,0}
\definecolor{ffqqqq}{rgb}{1,0,0}
\definecolor{ttccqq}{rgb}{0.2,0.8,0}
\definecolor{uququq}{rgb}{0.25,0.25,0.25}
\definecolor{qqqqff}{rgb}{0,0,1}
\definecolor{xdxdff}{rgb}{0.49,0.49,1}
\begin{tikzpicture}[scale=0.7,line cap=round,line join=round,>=triangle 45,x=4.0cm,y=4.0cm]
\draw[->,color=black] (-0.5,0) -- (2.5,0);
\foreach \x in {,1,2}
\draw[shift={(\x,0)},color=black] (0pt,2pt) -- (0pt,-2pt) node[below] {\footnotesize $\x$};
\draw[->,color=black] (0,-0.5) -- (0,2.5);
\foreach \y in {,1,2}
\draw[shift={(0,\y)},color=black] (2pt,0pt) -- (-2pt,0pt) node[left] {\footnotesize $\y$};
\draw[color=black] (0pt,-10pt) node[right] {\footnotesize $0$};
\clip(-0.5,-0.5) rectangle (2.5,2.5);
\fill[color=black,fill=black,fill opacity=0.1] (0,0) -- (0.5,1) -- (1,1) -- cycle;
\fill[color=black,fill=black,fill opacity=0.1] (0.5,1) -- (1,2) -- (1.34,1.32) -- (1,1) -- cycle;
\draw[color=black,fill=black,fill opacity=0.1, smooth,samples=50,domain=0.0:2.0] plot(\x,{(\x+\x*(2-\x)-abs(\x-\x*(2-\x)))/2}) -- (2,0) -- (0,0) -- cycle;
\draw[color=black,fill=black,fill opacity=0.1] {[smooth,samples=50,domain=1.0:2.0] plot(\x,{\x*(2-\x)})} -- (2,0) {[smooth,samples=50,domain=2.0:1.0] -- plot(\x,{(\x+4-2*\x-abs(\x-(4-2*\x)))/2})} -- (1,1) -- cycle;
\draw[smooth,samples=100,domain=0.0:1.3333333333333333] plot(\x,{(\x)});
\draw[smooth,samples=100,domain=1.0:2.0] plot(\x,{4-2*(\x)});
\draw[smooth,samples=100,domain=0.0:1.0] plot(\x,{2*(\x)});
\draw[smooth,samples=100,domain=1.0:2.0] plot(\x,{(\x)*(2-(\x))});
\draw[smooth,samples=100,domain=0.5:1.0] plot(\x,{1});
\draw [color=black] (0,0)-- (0.5,1);
\draw [color=black] (0.5,1)-- (1,1);
\draw [color=black] (1,1)-- (0,0);
\draw [color=black] (0.5,1)-- (1,2);
\draw [color=black] (1,2)-- (1.34,1.32);
\draw [color=black] (1.34,1.32)-- (1,1);
\draw [color=black] (1,1)-- (0.5,1);
\begin{scriptsize}
\draw (2.35,-0.05) node[anchor=north west] {$\sk[q]$};
\draw (-0.25,2.5) node[anchor=north west] {$\sk[xy]$};
\fill [color=black] (0,0) circle (1.5pt);
\draw[color=black] (-0.09,0.1) node {$A$};
\fill [color=black] (0.5,1) circle (1.5pt);
\draw[color=black] (0.48,1.1) node {$B$};
\fill [color=black] (1,1) circle (1.5pt);
\draw[color=black] (1.0,1.1) node {$C$};
\fill [color=black] (1,2) circle (1.5pt);
\draw[color=black] (1.09,2.15) node {$D$};
\fill [color=black] (1.34,1.32) circle (1.5pt);
\draw[color=black] (1.42,1.47) node {$E$};
\fill [color=black] (2,0) circle (1.5pt);
\draw[color=black] (2.08,0.15) node {$F$};
\end{scriptsize}
\end{tikzpicture}
\end{center}
\caption{$\Li$ stability area in $\vects$ of the twisted $\ddqqn$ scheme relative to $\utilde=\vectV$ with a non intrinsic diffusion. ACF : (\ref{eq:carreig}), ABC : (\ref{eq:carrei}), BCED :  (\ref{eq:ega3},\ref{eq:ega32}), CEF : (\ref{eq:ega3}).} 
\label{fig:zonzst_rectw}
\end{figure}
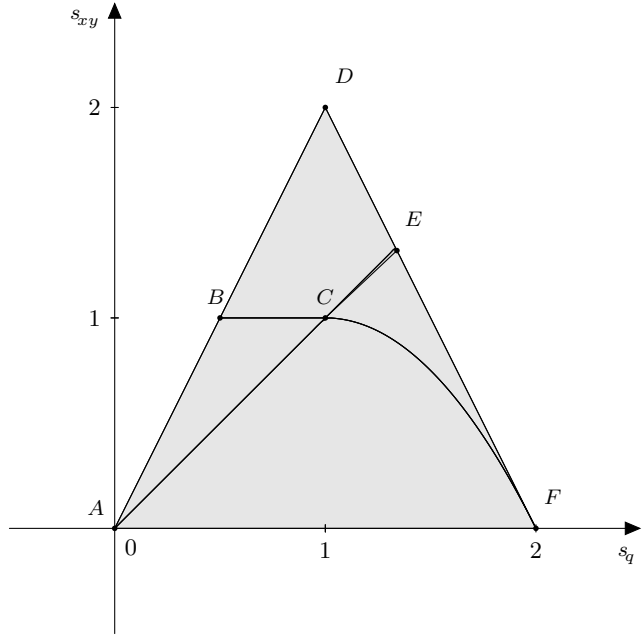

\begin{proof}
We determine the sufficient conditions for the positivity of the matrix sending $\vectf$ on $\vectfe$. A formal calculus software leads to the following inequalities
\begin{gather}
(2\sk[q]-\sk[xy])(\lambda\pm\Vx)(\lambda\pm\Vy)\geqslant0,\label{eq:twtrtr1}\\
(\lambda\pm\Vx)(\pm(2\sk[q]-\sk[xy])\Vy+\lambda\sk[xy])\geqslant0,\label{eq:twtrtr2}\\
4\lambda^2{-}((2\sk[q]{+}\sk[xy])\lambda^2{\pm}\sk[xy]\lambda(\Vx{+}(-1)^i\Vy){+}(-1)^i(\sk[xy]{-}2\sk[q])\Vx\Vy)\geqslant0,\quad i =0,1,\label{eq:twtrtr4}
\end{gather}
plus the inequality coming from (\ref{eq:twtrtr2}) when the roles of $\Vx$ and $\Vy$ are exchanged. We begin by the study of (\ref{eq:twtrtr1}).

{\bf Case 1 : $2\sk[q]<\sk[xy]$.} 
These inequalities are equivalent to 
\begin{equation*}\label{eq:incomp}
(\lambda\pm\Vx)(\lambda\pm\Vy)\leq0.
\end{equation*}
No  $\vectV$ verifies this inequality: such a set of $\vectV$ would be the intersection of fourth areas that do not intersect. We must assume that $\sk[xy]\leq2\sk[q]$ to eventually have a non empty stability area in $\vectV$.

{\bf Case 2 : $\sk[xy]=2\sk[q]$.}

If $\sk[xy]=0$, all the inequalities are obviously verified and the scheme is unconditionally stable in $\vectV$. If $\sk[xy]<0$, $\lambda\pm\Vx\leq0$ that is impossible. If $\sk[xy]>0$, (\ref{eq:twtrtr2}) becomes $(\lambda\pm\Vx)\lambda\geqslant0$, plus the analogous inequality in $\Vy$: this is equivalent to (\ref{eq:carrei}). The inequalities (\ref{eq:twtrtr4}) read 

\begin{equation*}
1-\frac{\sk[xy]}{4\lambda}(2\lambda\pm(\Vx\pm\Vy))\geqslant0,
\end{equation*}
that is equivalent to (\ref{eq:nor1s}). The final stability area is the intersection of (\ref{eq:carrei}) and (\ref{eq:nor1s}). Three cases, represented on the figures \ref{fig:3cas_rectw3}, \ref{fig:3cas_rectw2} and \ref{fig:3cas_rectw1}, are then possible.

\begin{figure}
\begin{center}
\definecolor{uququq}{rgb}{0.25,0.25,0.25}
\definecolor{ffqqqq}{rgb}{1,0,0}
\definecolor{qqqqff}{rgb}{0,0,1}
\begin{tikzpicture}[scale=0.7,line cap=round,line join=round,>=triangle 45,x=2.0cm,y=2.0cm]
\draw[->,color=black] (-2.4,0) -- (2.4,0);
\draw[->,color=black] (0,-2.4) -- (0,2.4);
\draw[shift={(1,0)},color=black] (0pt,2pt) -- (0pt,-2pt) node[below] {\footnotesize $\lambda$};
\draw[shift={(-1,0)},color=black] (0pt,2pt) -- (0pt,-2pt) node[below] {\footnotesize $-\lambda$};
\draw[shift={(2,0)},color=black] (0pt,2pt) -- (0pt,-2pt) node[below] {\footnotesize $2\lambda$};
\draw[shift={(-2,0)},color=black] (0pt,2pt) -- (0pt,-2pt) node[below] {\footnotesize $-2\lambda$};
\draw[shift={(0,1)},color=black] (2pt,0pt) -- (-2pt,0pt) node[left] {\footnotesize $\lambda$};
\draw[shift={(0,-1)},color=black] (2pt,0pt) -- (-2pt,0pt) node[left] {\footnotesize $-\lambda$};
\draw[shift={(0,2)},color=black] (2pt,0pt) -- (-2pt,0pt) node[left] {\footnotesize $2\lambda$};
\draw[shift={(0,-2)},color=black] (2pt,0pt) -- (-2pt,0pt) node[left] {\footnotesize $-2\lambda$};
\draw[color=black] (0pt,-10pt) node[right] {\footnotesize $0$};
\clip(-2.4,-2.4) rectangle (2.4,2.4);
\fill[color=black,fill=black,fill opacity=0.15] (-1,1) -- (-1,-1) -- (1,-1) -- (1,1) -- cycle;
\fill[color=black,fill=black,fill opacity=0.1] (0,2.1) -- (-2.1,0) -- (0,-2.1) -- (2.1,0) -- cycle;
\draw [color=black] (-1,1)-- (-1,-1);
\draw [color=black] (-1,-1)-- (1,-1);
\draw [color=black] (1,-1)-- (1,1);
\draw [color=black] (1,1)-- (-1,1);
\draw [color=black] (0,2.1)-- (-2.1,0);
\draw [color=black] (-2.1,0)-- (0,-2.1);
\draw [color=black] (0,-2.1)-- (2.1,0);
\draw [color=black] (2.1,0)-- (0,2.1);
\begin{scriptsize}
\draw[color=black] (2.15,0.24) node {$P$};
\draw (2.05,-0.05) node[anchor=north west] {$\Vx$};
\draw (-0.45,2.4) node[anchor=north west] {$\Vy$};
\end{scriptsize}
\end{tikzpicture}
\end{center}
\caption{$\Li$ stability area for the scheme relative to $\utilde=\vectV$ with a non intrinsic diffusion for $\sk[xy]=2\sk[q]$. The point $P$ is of abscissa $\dsp{2\lambda(2-\sk[xy])/\sk[xy]\geqslant2\lambda}$.} 
\label{fig:3cas_rectw3}
\end{figure}

\begin{figure}
\begin{center}
\definecolor{xdxdff}{rgb}{0.49,0.49,1}
\definecolor{uququq}{rgb}{0.25,0.25,0.25}
\definecolor{ffqqqq}{rgb}{1,0,0}
\definecolor{qqqqff}{rgb}{0,0,1}
\begin{tikzpicture}[scale=0.7,line cap=round,line join=round,>=triangle 45,x=2.0cm,y=2.0cm]
\draw[->,color=black] (-2.4,0) -- (2.4,0);
\draw[->,color=black] (0,-2.4) -- (0,2.4);
\draw[shift={(1,0)},color=black] (0pt,2pt) -- (0pt,-2pt) node[below] {\footnotesize $\lambda$};
\draw[shift={(-1,0)},color=black] (0pt,2pt) -- (0pt,-2pt) node[below] {\footnotesize $-\lambda$};
\draw[shift={(2,0)},color=black] (0pt,2pt) -- (0pt,-2pt) node[below] {\footnotesize $2\lambda$};
\draw[shift={(-2,0)},color=black] (0pt,2pt) -- (0pt,-2pt) node[below] {\footnotesize $-2\lambda$};
\draw[shift={(0,1)},color=black] (2pt,0pt) -- (-2pt,0pt) node[left] {\footnotesize $\lambda$};
\draw[shift={(0,-1)},color=black] (2pt,0pt) -- (-2pt,0pt) node[left] {\footnotesize $-\lambda$};
\draw[shift={(0,2)},color=black] (2pt,0pt) -- (-2pt,0pt) node[left] {\footnotesize $2\lambda$};
\draw[shift={(0,-2)},color=black] (2pt,0pt) -- (-2pt,0pt) node[left] {\footnotesize $-2\lambda$};
\draw[color=black] (0pt,-10pt) node[right] {\footnotesize $0$};
\clip(-2.4,-2.4) rectangle (2.4,2.4);
\fill[color=black,fill=black,fill opacity=0.1] (-1,1) -- (-1,-1) -- (1,-1) -- (1,1) -- cycle;
\fill[color=black,fill=black,fill opacity=0.1] (0,1.5) -- (-1.5,0) -- (0,-1.5) -- (1.5,0) -- cycle;
\draw [color=black] (-1,1)-- (-1,-1);
\draw [color=black] (-1,-1)-- (1,-1);
\draw [color=black] (1,-1)-- (1,1);
\draw [color=black] (1,1)-- (-1,1);
\draw [color=black] (0,1.5)-- (-1.5,0);
\draw [color=black] (-1.5,0)-- (0,-1.5);
\draw [color=black] (0,-1.5)-- (1.5,0);
\draw [color=black] (1.5,0)-- (0,1.5);
\begin{scriptsize}
\draw[color=uququq] (1.59,0.14) node {$P$};
\draw (2.05,-0.05) node[anchor=north west] {$\Vx$};
\draw (-0.45,2.4) node[anchor=north west] {$\Vy$};
\end{scriptsize}
\end{tikzpicture}
\end{center}
\caption{$\Li$ stability area for the scheme relative to $\utilde=\vectV$ with a non intrinsic diffusion for $\sk[xy]=2\sk[q]$. The point $P$ is of abscissa $\dsp{\lambda\leq2\lambda(2-\sk[xy])/\sk[xy]\leq2\lambda}$.} 
\label{fig:3cas_rectw2}
\end{figure}
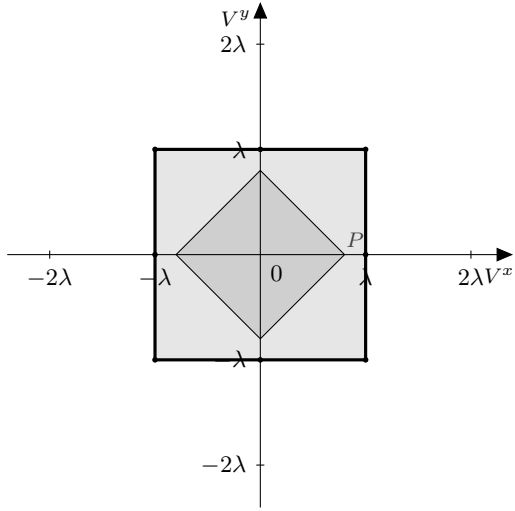
\begin{figure}
\begin{center}
\definecolor{xdxdff}{rgb}{0.49,0.49,1}
\definecolor{uququq}{rgb}{0.25,0.25,0.25}
\definecolor{ffqqqq}{rgb}{1,0,0}
\definecolor{qqqqff}{rgb}{0,0,1}
\begin{tikzpicture}[scale=0.7,line cap=round,line join=round,>=triangle 45,x=2.0cm,y=2.0cm]
\draw[->,color=black] (-2.4,0) -- (2.4,0);
\draw[shift={(1,0)},color=black] (0pt,2pt) -- (0pt,-2pt) node[below] {\footnotesize $\lambda$};
\draw[shift={(-1,0)},color=black] (0pt,2pt) -- (0pt,-2pt) node[below] {\footnotesize $-\lambda$};
\draw[shift={(2,0)},color=black] (0pt,2pt) -- (0pt,-2pt) node[below] {\footnotesize $2\lambda$};
\draw[shift={(-2,0)},color=black] (0pt,2pt) -- (0pt,-2pt) node[below] {\footnotesize $-2\lambda$};
\draw[shift={(0,1)},color=black] (2pt,0pt) -- (-2pt,0pt) node[left] {\footnotesize $\lambda$};
\draw[shift={(0,-1)},color=black] (2pt,0pt) -- (-2pt,0pt) node[left] {\footnotesize $-\lambda$};
\draw[shift={(0,2)},color=black] (2pt,0pt) -- (-2pt,0pt) node[left] {\footnotesize $2\lambda$};
\draw[shift={(0,-2)},color=black] (2pt,0pt) -- (-2pt,0pt) node[left] {\footnotesize $-2\lambda$};
\draw[->,color=black] (0,-2.4) -- (0,2.4);
\draw[color=black] (0pt,-10pt) node[right] {\footnotesize $0$};
\clip(-2.4,-2.4) rectangle (2.4,2.4);
\fill[color=black,fill=black,fill opacity=0.1] (-1,1) -- (-1,-1) -- (1,-1) -- (1,1) -- cycle;
\fill[color=black,fill=black,fill opacity=0.1] (0,0.8) -- (-0.8,0) -- (0,-0.8) -- (0.8,0) -- cycle;
\draw [line width=1.2pt,color=black] (-1,1)-- (-1,-1);
\draw [line width=1.2pt,color=black] (-1,-1)-- (1,-1);
\draw [line width=1.2pt,color=black] (1,-1)-- (1,1);
\draw [line width=1.2pt,color=black] (1,1)-- (-1,1);
\draw [color=black] (0,0.8)-- (-0.8,0);
\draw [color=black] (-0.8,0)-- (0,-0.8);
\draw [color=black] (0,-0.8)-- (0.8,0);
\draw [color=black] (0.8,0)-- (0,0.8);
\begin{scriptsize}
\draw[color=uququq] (0.9,0.14) node {$P$};
\fill [color=black] (-1,1) circle (1.5pt);
\fill [color=black] (-1,-1) circle (1.5pt);
\fill [color=black] (1,-1) circle (1.5pt);
\fill [color=black] (1,1) circle (1.5pt);
\fill [color=black] (0,1) circle (1.5pt);
\fill [color=black] (-1,0) circle (1.5pt);
\fill [color=black] (0,-1) circle (1.5pt);
\fill [color=black] (1,0) circle (1.5pt);
\draw (2.05,-0.05) node[anchor=north west] {$\Vx$};
\draw (-0.45,2.4) node[anchor=north west] {$\Vy$};
\end{scriptsize}
\end{tikzpicture}
\end{center}
\caption{$\Li$ stability area for the scheme relative to $\utilde=\vectV$ with a non intrinsic diffusion for $\sk[xy]=2\sk[q]$. The point $P$ is of abscissa $\dsp{2\lambda(2-\sk[xy])/\sk[xy]\leq\lambda}$.} 
\label{fig:3cas_rectw1}
\end{figure}
Let $\sk[xy]\leq1$, then $$\dsp{2\lambda\frac{(2-\sk[xy])}{\sk[xy]}\geqslant2\lambda},$$ and the stability area is given by (\ref{eq:carrei}) (figure \ref{fig:3cas_rectw3}), let $1\leq\sk[xy]\leq4/3$, then $$\dsp{\lambda\leq2\lambda\frac{(2-\sk[xy])}{\sk[xy]}\leq2\lambda},$$ and the stability area is the intersection of (\ref{eq:carrei}) and (\ref{eq:nor1s}) represented on the figure \ref{fig:3cas_rectw2}, let $\sk[xy]\geqslant4/3$, then $$\dsp{2\lambda\frac{(2-\sk[xy])}{\sk[xy]}\leq\lambda},$$ and the stability area is given by (\ref{eq:nor1s}) (figure \ref{fig:3cas_rectw1}).
This area is reduced to $\vectV=\vectz$ when $\sk[xy]=2$.

{\bf Case 3 : $\sk[xy]<2\sk[q]$.}

The inequalities (\ref{eq:twtrtr1}) read
\begin{equation*}
(\lambda\pm\Vx)(\lambda\pm\Vy)\geqslant0,
\end{equation*}
that is equivalent to (\ref{eq:carrei}). The inequalities (\ref{eq:twtrtr2}) write
\begin{equation}\label{eq:inegprov}
(\lambda\pm\Vx)(\lambda\gamma\pm\Vy)\geqslant0,
\end{equation}

for $\gamma=\sk[xy]/(2\sk[q]-\sk[xy])$. The inequalities (\ref{eq:inegprov}) have solutions if $\gamma\geqslant0$, that is equivalent to $\sk[xy]\geqslant0$. If $\sk[xy]\geqslant\sk[q]$, (\ref{eq:inegprov}) contains (\ref{eq:carrei}), that correspond to the same constraints that those imposed by (\ref{eq:twtrtr1}). If $\sk[xy]\leq\sk[q]$, then the stability area reduces to (\ref{eq:carreig}).

There is still to study the influence of (\ref{eq:twtrtr4}) on the stability area. These identities read
\begin{equation*}
\frac{\lambda^2}{2\sk[q]-\sk[xy]}(4-(2\sk[q]+\sk[xy]))\pm\gamma\lambda(\Vx+(-1)^i\Vy)+(-1)^i\Vx\Vy\geqslant0,\quad i=0,1,
\end{equation*}
after division by $2\sk[q]-\sk[xy]$. We change of frame, $\VVx=\Vx+\Vy,~ \VVy=\Vx-\Vy$, and obtain
\begin{equation*}
\frac{4\lambda^2}{2\sk[q]-\sk[xy]}(4-(2\sk[q]+\sk[xy]))\pm4\gamma\lambda\VVx+(\VVx)^2-(\VVy)^2\geqslant0,
\end{equation*}
and the analogous inequality exchanging the roles of $\VVx$ and $\VVy$. We center these equations to obtain (\ref{eq:ega3}) and (\ref{eq:ega32}).

{\bf Sub-case 1 : $\sk[xy]\leq\sk[q](2-\sk[q])$.}

The inequality (\ref{eq:ega3}) reads
\begin{equation}\label{eq:rectwhypnuy}
(\VVy)^2-(\VVx\pm2\lambda\gamma)^2\leq\frac{16\lambda^2}{(2\sk[q]-\sk[xy])^2}(\sk[q](2-\sk[q])-\sk[xy]).
\end{equation}
We represent the areas corresponding to the inequations (\ref{eq:carreig}) and (\ref{eq:rectwhypnuy}) on the figure \ref{fig:zonehyp_rectw}. The relation (\ref{eq:carreig}) is equivalent in the new frame to $\normunu\leq2\lambda\gamma$.
It is contained in the area delimited by (\ref{eq:rectwhypnuy}) (figure \ref{fig:zonehyp_rectw}). The final stability area is then given by (\ref{eq:carrei}) if $\gamma\geqslant1$ and by (\ref{eq:carreig}) otherwise.

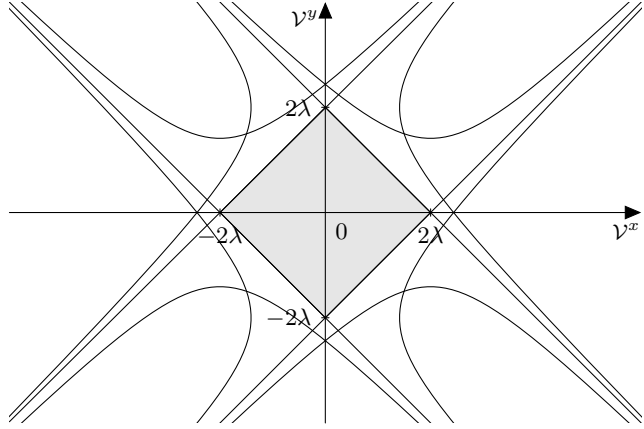
\begin{figure}
\begin{center}
\definecolor{ttffff}{rgb}{0.2,1,1}
\definecolor{uququq}{rgb}{0.25,0.25,0.25}
\definecolor{ffqqqq}{rgb}{1,0,0}
\definecolor{ttqqff}{rgb}{0.2,0,1}
\definecolor{qqccqq}{rgb}{0,0.8,0}
\definecolor{qqqqff}{rgb}{0,0,1}
\definecolor{xdxdff}{rgb}{0.49,0.49,1}
\begin{tikzpicture}[scale=0.7,line cap=round,line join=round,>=triangle 45,x=1.0cm,y=1.0cm]
\draw[->,color=black] (-6,0) -- (6,0);
\draw[->,color=black] (0,-4) -- (0,4);
\draw[color=black] (0pt,-10pt) node[right] {\footnotesize $0$};
\draw[shift={(2,0)},color=black] (0pt,2pt) -- (0pt,-2pt) node[below] {\footnotesize $2\lambda$};
\draw[shift={(-2,0)},color=black] (0pt,2pt) -- (0pt,-2pt) node[below] {\footnotesize $-2\lambda$};
\draw[shift={(0,2)},color=black] (2pt,0pt) -- (-2pt,0pt) node[left] {\footnotesize $2\lambda$};
\draw[shift={(0,-2)},color=black] (2pt,0pt) -- (-2pt,0pt) node[left] {\footnotesize $-2\lambda$};
\clip(-6,-4) rectangle (6,4);
\fill[color=black,fill=black,fill opacity=0.1] (0,2) -- (-2,0) -- (0,-2) -- (2,0) -- cycle;
\draw [domain=-6:6] plot(\x,{(-4--2*\x)/-2});
\draw [domain=-6:6] plot(\x,{(--3.36--1.68*\x)/1.69});
\draw [domain=-6:6] plot(\x,{(-4-2*\x)/2});
\draw [domain=-6:6] plot(\x,{(--4-2*\x)/-2});
\draw [samples=50,domain=-0.99:0.99,rotate around={90:(-2,0)},xshift=-2cm,yshift=0cm,color=black] plot ({1.41*(1+\x*\x)/(1-\x*\x)},{1.42*2*\x/(1-\x*\x)});
\draw [samples=50,domain=-0.99:0.99,rotate around={90:(-2,0)},xshift=-2cm,yshift=0cm,color=black] plot ({1.41*(-1-\x*\x)/(1-\x*\x)},{1.42*(-2)*\x/(1-\x*\x)});
\draw [samples=50,domain=-0.99:0.99,rotate around={90:(2,0)},xshift=2cm,yshift=0cm,color=black] plot ({1.41*(1+\x*\x)/(1-\x*\x)},{1.42*2*\x/(1-\x*\x)});
\draw [samples=50,domain=-0.99:0.99,rotate around={90:(2,0)},xshift=2cm,yshift=0cm,color=black] plot ({1.41*(-1-\x*\x)/(1-\x*\x)},{1.42*(-2)*\x/(1-\x*\x)});
\draw [color=black] (0,2)-- (-2,0);
\draw [color=black] (-2,0)-- (0,-2);
\draw [color=black] (0,-2)-- (2,0);
\draw [color=black] (2,0)-- (0,2);
\draw [samples=50,domain=-0.99:0.99,rotate around={180:(0,2)},xshift=0cm,yshift=2cm,color=black] plot ({1.41*(1+\x*\x)/(1-\x*\x)},{1.42*2*\x/(1-\x*\x)});
\draw [samples=50,domain=-0.99:0.99,rotate around={180:(0,2)},xshift=0cm,yshift=2cm,color=black] plot ({1.41*(-1-\x*\x)/(1-\x*\x)},{1.42*(-2)*\x/(1-\x*\x)});
\draw [samples=50,domain=-0.99:0.99,rotate around={180:(0,-2)},xshift=0cm,yshift=-2cm,color=black] plot ({1.41*(1+\x*\x)/(1-\x*\x)},{1.42*2*\x/(1-\x*\x)});
\draw [samples=50,domain=-0.99:0.99,rotate around={180:(0,-2)},xshift=0cm,yshift=-2cm,color=black] plot ({1.41*(-1-\x*\x)/(1-\x*\x)},{1.42*(-2)*\x/(1-\x*\x)});
\begin{scriptsize}
\draw (5.3,-0.05) node[anchor=north west] {$\VVx$};
\draw (-0.8,4) node[anchor=north west] {$\VVy$};
\end{scriptsize}
\end{tikzpicture}
\end{center}
\caption{Layout of the areas given by (\ref{eq:carreig}) and (\ref{eq:rectwhypnuy}).} 
\label{fig:zonehyp_rectw}
\end{figure}

{\bf Sub-case 2 : $\sk[xy]\geqslant\sk[q](2-\sk[q])$.}

The inequality (\ref{eq:ega3}) has a solution $\vectV$ if the abscissa of the left pole of the hyperbole centered in $2\lambda\gamma$ is nonnegative.
This abscissa, equal to $$2\lambda\gamma-\frac{4\lambda}{(2\sk[q]-\sk[xy])}(\sk[xy]-\sk[q](2-\sk[q]))^{\frac{1}{2}},$$ is nonnegative if
\begin{equation}\label{eq:cone}
|2-\sk[xy]|\geqslant2|\sk[q]-1|.
\end{equation}
Assuming that $\sk[xy]\geqslant 2$, then $\sk[xy]\geqslant 2|\sk[q]-1|+2\geqslant 2\sk[q]$, that gives an empty $\Li$ stability area in $\vectV$ (refer to the beginning of the proof). As a consequence $\sk[xy]\leq2$.
If $\sk[q]\leq1$, the inequality (\ref{eq:cone}) is equivalent to $\sk[xy]\leq2\sk[q]$, otherwise to 
$\sk[xy]\leq2(2-\sk[q])$.
These conditions carrying on $\vects$ ensure the existence of a non empty stability area in $\vectV$.

 To obtain the final stability area in $\vectV$ given by the set of inequations from (\ref{eq:twtrtr1}) to (\ref{eq:twtrtr4}), we compare those associated with (\ref{eq:ega3}) to (\ref{eq:carrei}) if $\gamma\geqslant1$ or (\ref{eq:carreig}) if $\gamma\leq1$. The two areas (\ref{eq:carrei}) and (\ref{eq:carreig}) read respectively $\normunu\leq2\lambda$,
 and
 \begin{equation}\label{eq:nor1mV}
\normunu\leq2\lambda\gamma.
 \end{equation}
 in the frame $(0,\VVx,\VVy)$.
 If $\gamma\leq1$, the stability area is included in the one given by (\ref{eq:nor1mV}): it is represented by the points $A,B,C,D$ on the figure \ref{fig:zonehyp_rectw3}.
 If $\gamma\geqslant1$ and $\sk[xy]\leq1$, the area (\ref{eq:carrei}) is included in the one given by (\ref{eq:ega3}): the abscissa of the left pole $P$ of the hyperbole centered $2\lambda\gamma$ is greater than $2\lambda$ (figure \ref{fig:zonehyp_rectw2}).
 If $\sk[xy]\geqslant1$, it is not the case any more and the area decreases when $\sk[xy]$ increases. \qedhere
 
 \begin{figure}
\begin{center}
\definecolor{qqffff}{rgb}{0,1,1}
\definecolor{qqttff}{rgb}{0,0.2,1}
\definecolor{zzttqq}{rgb}{0.6,0.2,0}
\definecolor{qqccqq}{rgb}{0,0.8,0}
\definecolor{uququq}{rgb}{0.25,0.25,0.25}
\definecolor{ffqqqq}{rgb}{1,0,0}
\definecolor{qqqqff}{rgb}{0,0,1}
\definecolor{xdxdff}{rgb}{0.49,0.49,1}
\begin{tikzpicture}[scale=1.,line cap=round,line join=round,>=triangle 45,x=1.0cm,y=1.0cm]
\draw[->,color=black] (-4,0) -- (4,0);
\draw[->,color=black] (0,-4) -- (0,4);
\draw[shift={(2,0)},color=black] (0pt,2pt) -- (0pt,-2pt) node[below] {\footnotesize $2\lambda\gamma$};
\draw[shift={(-2,0)},color=black] (0pt,2pt) -- (0pt,-2pt) node[below] {\footnotesize $-2\lambda\gamma$};
\draw[shift={(0,2)},color=black] (2pt,0pt) -- (-2pt,0pt) node[left] {\footnotesize $2\lambda\gamma$};
\draw[shift={(0,-2)},color=black] (2pt,0pt) -- (-2pt,0pt) node[left] {\footnotesize $-2\lambda\gamma$};
\draw[shift={(1,0)},color=black] (0pt,2pt) -- (0pt,-2pt) node[below] {\footnotesize $2\lambda$};
\draw[shift={(-1,0)},color=black] (0pt,2pt) -- (0pt,-2pt) node[below] {\footnotesize $-2\lambda$};
\draw[shift={(0,1)},color=black] (2pt,0pt) -- (-2pt,0pt) node[left] {\footnotesize $2\lambda$};
\draw[shift={(0,-1)},color=black] (2pt,0pt) -- (-2pt,0pt) node[left] {\footnotesize $-2\lambda$};
\draw[color=black] (0pt,-10pt) node[right] {\footnotesize $0$};
\clip(-4,-4) rectangle (4,4);
\fill[color=black,fill=black,fill opacity=0.1] (0,1) -- (-1,0) -- (0,-1) -- (1,0) -- cycle;
\draw [domain=-5:5] plot(\x,{(--4--2*\x)/2});
\draw [domain=-5:5] plot(\x,{(--4-2*\x)/2});
\draw [domain=-5:5] plot(\x,{(--4-2*\x)/-2});
\draw [domain=-5:5] plot(\x,{(--4--2*\x)/-2});
\draw [color=black] (0,1)-- (-1,0);
\draw [color=black] (-1,0)-- (0,-1);
\draw [color=black] (0,-1)-- (1,0);
\draw [color=black] (1,0)-- (0,1);
\draw [samples=50,domain=-0.99:0.99,rotate around={0:(-2,0)},xshift=-2cm,yshift=0cm,color=black] plot ({0.5*(1+\x*\x)/(1-\x*\x)},{0.5*2*\x/(1-\x*\x)});
\draw [samples=50,domain=-0.99:0.99,rotate around={0:(-2,0)},xshift=-2cm,yshift=0cm,color=black] plot ({0.5*(-1-\x*\x)/(1-\x*\x)},{0.5*(-2)*\x/(1-\x*\x)});
\draw [samples=50,domain=-0.99:0.99,rotate around={180:(2,0)},xshift=2cm,yshift=0cm,color=black] plot ({0.5*(1+\x*\x)/(1-\x*\x)},{0.5*2*\x/(1-\x*\x)});
\draw [samples=50,domain=-0.99:0.99,rotate around={180:(2,0)},xshift=2cm,yshift=0cm,color=black] plot ({0.5*(-1-\x*\x)/(1-\x*\x)},{0.5*(-2)*\x/(1-\x*\x)});
\draw [color=zzttqq] (1,0)-- (0,-1);
\draw [samples=50,domain=-0.99:0.99,rotate around={-90:(0,2)},xshift=0cm,yshift=2cm,color=black] plot ({0.5*(1+\x*\x)/(1-\x*\x)},{0.5*2*\x/(1-\x*\x)});
\draw [samples=50,domain=-0.99:0.99,rotate around={-90:(0,2)},xshift=0cm,yshift=2cm,color=black] plot ({0.5*(-1-\x*\x)/(1-\x*\x)},{0.5*(-2)*\x/(1-\x*\x)});
\draw [samples=50,domain=-0.99:0.99,rotate around={90:(0,-2)},xshift=0cm,yshift=-2cm,color=black] plot ({0.5*(1+\x*\x)/(1-\x*\x)},{0.5*2*\x/(1-\x*\x)});
\draw [samples=50,domain=-0.99:0.99,rotate around={90:(0,-2)},xshift=0cm,yshift=-2cm,color=black] plot ({0.5*(-1-\x*\x)/(1-\x*\x)},{0.5*(-2)*\x/(1-\x*\x)});
\begin{scriptsize}
\fill [color=uququq] (1.5,0) circle (1.5pt);
\draw[color=uququq] (1.6,0.17) node {$P$};
\draw (3.5,-0.05) node[anchor=north west] {$\VVx$};
\draw (-0.65,4.) node[anchor=north west] {$\VVy$};
\end{scriptsize}
\end{tikzpicture}
\end{center}
\caption{Layout of the bounds of the $\Li$ stability areas of the scheme relative to $\utilde=\vectV$ with a non intrinsic diffusion for $\gamma\geqslant1$ and $\sk[q](2-\sk[q])\leq\sk[xy]\leq1$.} 
\label{fig:zonehyp_rectw2}
\end{figure}

 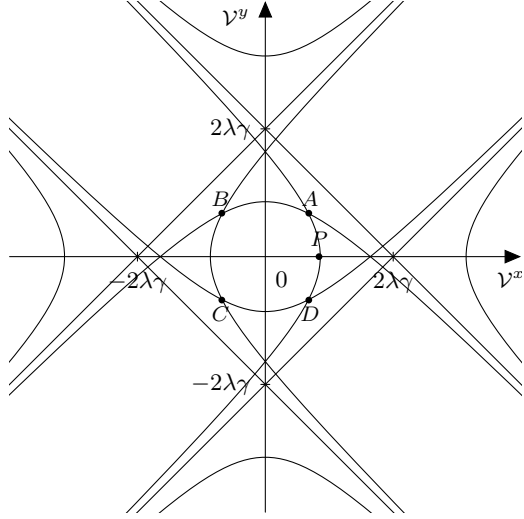
\begin{figure}
\begin{center}
\definecolor{qqttff}{rgb}{0,0.2,1}
\definecolor{qqffff}{rgb}{0,1,1}
\definecolor{uququq}{rgb}{0.25,0.25,0.25}
\definecolor{ffqqqq}{rgb}{1,0,0}
\definecolor{qqzzqq}{rgb}{0,0.6,0}
\definecolor{qqqqff}{rgb}{0,0,1}
\definecolor{xdxdff}{rgb}{0.49,0.49,1}
\begin{tikzpicture}[scale=0.85,line cap=round,line join=round,>=triangle 45,x=2.0cm,y=2.0cm]
\draw[->,color=black] (-2,0) -- (2,0);
\draw[->,color=black] (0,-2) -- (0,2);
\draw[shift={(1,0)},color=black] (0pt,2pt) -- (0pt,-2pt) node[below] {\footnotesize $2\lambda\gamma$};
\draw[shift={(-1,0)},color=black] (0pt,2pt) -- (0pt,-2pt) node[below] {\footnotesize $-2\lambda\gamma$};
\draw[shift={(0,1)},color=black] (2pt,0pt) -- (-2pt,0pt) node[left] {\footnotesize $2\lambda\gamma$};
\draw[shift={(0,-1)},color=black] (2pt,0pt) -- (-2pt,0pt) node[left] {\footnotesize $-2\lambda\gamma$};
\draw[color=black] (0pt,-10pt) node[right] {\footnotesize $0$};
\clip(-2,-2) rectangle (2,2);
\draw [domain=-3:3] plot(\x,{(--1--1*\x)/1});
\draw [domain=-3:3] plot(\x,{(--1-1*\x)/-1});
\draw [domain=-3:3] plot(\x,{(--1--1*\x)/-1});
\draw [domain=-3:3] plot(\x,{(--1-1*\x)/1});
\draw [samples=50,domain=-0.99:0.99,rotate around={0:(1,0)},xshift=2cm,yshift=0cm,color=black] plot ({0.57*(1+\x*\x)/(1-\x*\x)},{0.57*2*\x/(1-\x*\x)});
\draw [samples=50,domain=-0.99:0.99,rotate around={0:(1,0)},xshift=2cm,yshift=0cm,color=black] plot ({0.57*(-1-\x*\x)/(1-\x*\x)},{0.57*(-2)*\x/(1-\x*\x)});
\draw [samples=50,domain=-0.99:0.99,rotate around={180:(-1,0)},xshift=-2cm,yshift=0cm,color=black] plot ({0.57*(1+\x*\x)/(1-\x*\x)},{0.57*2*\x/(1-\x*\x)});
\draw [samples=50,domain=-0.99:0.99,rotate around={180:(-1,0)},xshift=-2cm,yshift=0cm,color=black] plot ({0.57*(-1-\x*\x)/(1-\x*\x)},{0.57*(-2)*\x/(1-\x*\x)});
\draw [samples=50,domain=-0.99:0.99,rotate around={-90:(0,-1)},xshift=0cm,yshift=-2cm,color=black] plot ({0.57*(1+\x*\x)/(1-\x*\x)},{0.57*2*\x/(1-\x*\x)});
\draw [samples=50,domain=-0.99:0.99,rotate around={-90:(0,-1)},xshift=0cm,yshift=-2cm,color=black] plot ({0.57*(-1-\x*\x)/(1-\x*\x)},{0.57*(-2)*\x/(1-\x*\x)});
\draw [samples=50,domain=-0.99:0.99,rotate around={90:(0,1)},xshift=0cm,yshift=2cm,color=black] plot ({0.57*(1+\x*\x)/(1-\x*\x)},{0.57*2*\x/(1-\x*\x)});
\draw [samples=50,domain=-0.99:0.99,rotate around={90:(0,1)},xshift=0cm,yshift=2cm,color=black] plot ({0.57*(-1-\x*\x)/(1-\x*\x)},{0.57*(-2)*\x/(1-\x*\x)});
\begin{scriptsize}
\draw (1.75,-0.05) node[anchor=north west] {$\VVx$};
\draw (-0.4,2.) node[anchor=north west] {$\VVy$};
\fill [color=black] (0.42,0) circle (1.5pt);
\draw[color=black] (0.42,0.14) node {$P$};
\fill [color=black] (0.34,0.34) circle (1.5pt);
\draw[color=black] (0.35,0.45) node {$A$};
\fill [color=black] (-0.34,0.34) circle (1.5pt);
\draw[color=black] (-0.35,0.45) node {$B$};
\fill [color=black] (0.34,-0.34) circle (1.5pt);
\draw[color=black] (0.35,-0.45) node {$D$};
\fill [color=black] (-0.34,-0.34) circle (1.5pt);
\draw[color=black] (-0.35,-0.45) node {$C$};
\end{scriptsize}
\end{tikzpicture}
\end{center}
\caption{Layout of the bounds of the $\Li$ stability areas of the scheme relative to $\utilde=\vectV$ with a non intrinsic diffusion for $\gamma\leq1$ and $\sk[q](2-\sk[q])\leq\sk[xy]$.} 
\label{fig:zonehyp_rectw3}
\end{figure}
\end{proof}

The domain of $\Li$ stability in $\vects$ is identical for $\utilde=\vectz$ and $\utilde=\vectV$.
It is given by the triangle of the figures \ref{fig:zonzst_dhtw} and \ref{fig:zonzst_rectw}. However, the choice $\utilde=\vectV$ increases the set of the $\vects$ verifying the optimal stability area $\normi\leq\lambda$: indeed the area (\ref{eq:carrei}) spreads to TRT schemes consistent with $\sk[q]\leq\sk[xy]\leq \min(1,2\sk[q])$.

We now compare the areas in $\vectV$ associated with $\utilde=\vectz$ and $\utilde=\vectV$ for three $\vects$ that don't verify $\sk[q]\leq\sk[xy]\leq \min(1,2\sk[q])$. We want to know if $\utilde=\vectV$ provides a better $\Li$ stability behaviour than $\utilde=\vectz$. The associated inequalities are presented in the table \ref{tab:compzone1}. The solutions $\vectV$ are represented on the figures from \ref{fig:compzone1} to \ref{fig:compzone3} in the frame $(O,\VVx,\VVy)$: the areas delimited by the points $A,B,C,D$ are associated with the scheme relative to $\utilde=\vectV$, and those bounded by $E,F,G,H$ correspond to the MRT scheme ($\utilde=\vectz$).

\begin{table}
\centering\small
\grandtraittop
\begin{tabular}{@{}p{2.2cm}p{3.3cm}p{3.8cm}p{3.6cm}@{}}
Choice of $\vects$&$(0,1,1,1/2)$&$(0,1,1,3/2)$&$(0,3/2,3/2,3/4)$\\
Area for $\utilde{=}\vectz$&$(\VVx\pm4)^2-(\VVy)^2\geqslant12$

$(\VVy\pm4)^2-(\VVx)^2\geqslant12$ &$(\VVx\pm4/3)^2-(\VVy)^2\geqslant4/9$

$(\VVy\pm4/3)^2-(\VVx)^2\geqslant4/9$&$(\VVx\pm4)^2-(\VVy)^2\geqslant44/3$

$(\VVy\pm4)^2-(\VVx)^2\geqslant44/3$\\
Area for $\utilde{=}\vectV$&$\normi\leq1/3$&$(\VVx\pm6)^2-(\VVy)^2\geqslant32$

$(\VVy\pm6)^2-(\VVx)^2\geqslant32$&$\normi\leq1/3$\\
&&$\normi\leq1$& 
\end{tabular}
\grandtraitbottom
 \caption{$\Li$ stability areas of the schemes with a non intrinsic diffusion for $\lambda=1$.}
 \label{tab:compzone1}
\end{table}

\begin{figure}
\begin{center}
\definecolor{uququq}{rgb}{0.25,0.25,0.25}
\definecolor{xdxdff}{rgb}{0.49,0.49,1}
\definecolor{ffqqqq}{rgb}{1,0,0}
\begin{tikzpicture}[scale=0.45,line cap=round,line join=round,>=triangle 45,x=4.0cm,y=4.0cm]
\draw[->,color=black] (-2.2,0) -- (2.2,0);
\draw[->,color=black] (0,-2.2) -- (0,2.2);
\clip(-2.2,-2.2) rectangle (2.2,2.2);
\draw [samples=50,domain=-0.3:0.3,rotate around={0:(4,0)},xshift=16cm,yshift=0cm,color=ffqqqq] plot ({3.46*(1+\x*\x)/(1-\x*\x)},{3.46*2*\x/(1-\x*\x)});
\draw [samples=50,domain=-0.3:0.3,rotate around={0:(4,0)},xshift=16cm,yshift=0cm,color=ffqqqq] plot ({3.46*(-1-\x*\x)/(1-\x*\x)},{3.46*(-2)*\x/(1-\x*\x)});
\draw [samples=50,domain=-0.2:0.2,rotate around={180:(-4,0)},xshift=-16cm,yshift=0cm,color=ffqqqq] plot ({3.46*(1+\x*\x)/(1-\x*\x)},{3.46*2*\x/(1-\x*\x)});
\draw [samples=50,domain=-0.3:0.3,rotate around={180:(-4,0)},xshift=-16cm,yshift=0cm,color=ffqqqq] plot ({3.46*(-1-\x*\x)/(1-\x*\x)},{3.46*(-2)*\x/(1-\x*\x)});
\draw [samples=50,domain=-0.2:0.2,rotate around={90:(0,4)},xshift=0cm,yshift=16cm,color=ffqqqq] plot ({3.46*(1+\x*\x)/(1-\x*\x)},{3.46*2*\x/(1-\x*\x)});
\draw [samples=50,domain=-0.3:0.3,rotate around={90:(0,4)},xshift=0cm,yshift=16cm,color=ffqqqq] plot ({3.46*(-1-\x*\x)/(1-\x*\x)},{3.46*(-2)*\x/(1-\x*\x)});
\draw [samples=50,domain=-0.2:0.2,rotate around={-90:(0,-4)},xshift=0cm,yshift=-16cm,color=ffqqqq] plot ({3.46*(1+\x*\x)/(1-\x*\x)},{3.46*2*\x/(1-\x*\x)});
\draw [samples=50,domain=-0.3:0.3,rotate around={-90:(0,-4)},xshift=0cm,yshift=-16cm,color=ffqqqq] plot ({3.46*(-1-\x*\x)/(1-\x*\x)},{3.46*(-2)*\x/(1-\x*\x)});
\draw [domain=-2.2:2.2] plot(\x,{(-0.67--1*\x)/1});
\draw [domain=-2.2:2.2] plot(\x,{(-0.67--1*\x)/-1});
\draw [domain=-2.2:2.2] plot(\x,{(-0.67-1*\x)/-1});
\draw [domain=-2.2:2.2] plot(\x,{(-0.67-1*\x)/1});
\begin{scriptsize}
\fill [color=xdxdff] (0,0) circle (1.5pt);
\draw[color=black] (0.06,0.1) node {$O$};
\fill [color=uququq] (-0.67,0) circle (1.5pt);
\draw[color=uququq] (-0.61,0.1) node {$A$};
\fill [color=uququq] (0,0.67) circle (1.5pt);
\draw[color=uququq] (0.06,0.77) node {$B$};
\fill [color=uququq] (0.67,0) circle (1.5pt);
\draw[color=uququq] (0.73,0.1) node {$C$};
\fill [color=uququq] (0,-0.67) circle (1.5pt);
\draw[color=uququq] (0.06,-0.56) node {$D$};
\fill [color=uququq] (-0.5,0.5) circle (1.5pt);
\draw[color=uququq] (-0.44,0.6) node {$E$};
\fill [color=uququq] (0.5,0.5) circle (1.5pt);
\draw[color=uququq] (0.57,0.6) node {$F$};
\fill [color=uququq] (0.5,-0.5) circle (1.5pt);
\draw[color=uququq] (0.57,-0.4) node {$G$};
\fill [color=uququq] (-0.5,-0.5) circle (1.5pt);
\draw[color=uququq] (-0.46,-0.4) node {$H$};
\draw (1.9,-0.05) node[anchor=north west] {$\VVx$};
\draw (-0.4,2.2) node[anchor=north west] {$\VVy$};
\end{scriptsize}
\end{tikzpicture}
\end{center}
\caption{$\Li$ stability areas for $\vects=(0,1,1,1/2)$. Scheme relative to $\utilde=\vectV$ ($A,B,C,D$), to $\utilde=\vectz$ ($E,F,G,H$), with a non intrinsic diffusion.} 
\label{fig:compzone1}
\end{figure}
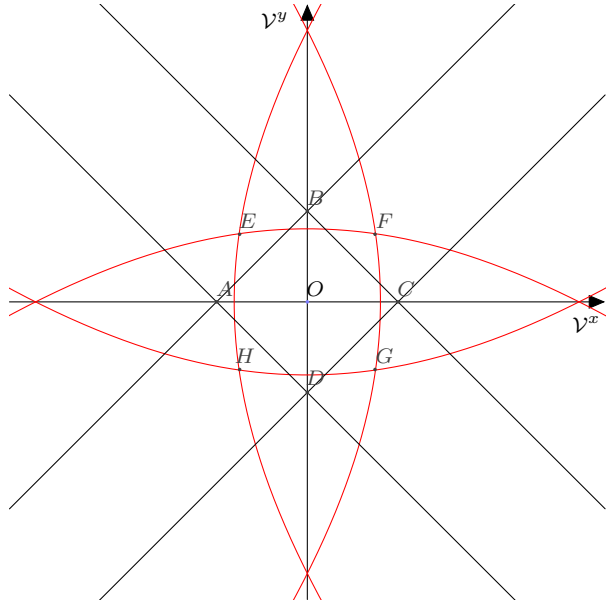

\begin{figure}
\begin{center}
\definecolor{uququq}{rgb}{0.25,0.25,0.25}
\definecolor{ffqqqq}{rgb}{1,0,0}
\begin{tikzpicture}[scale=0.45,line cap=round,line join=round,>=triangle 45,x=4.0cm,y=4.0cm]
\draw[->,color=black] (-2.2,0) -- (2.2,0);
\draw[->,color=black] (0,-2.2) -- (0,2.2);
\clip(-2.2,-2.2) rectangle (2.2,2.2);
\draw [samples=50,domain=-0.99:0.99,rotate around={0:(1.33,0)},xshift=5.33cm,yshift=0cm,color=ffqqqq] plot ({0.67*(1+\x*\x)/(1-\x*\x)},{0.67*2*\x/(1-\x*\x)});
\draw [samples=50,domain=-0.99:0.99,rotate around={0:(1.33,0)},xshift=5.33cm,yshift=0cm,color=ffqqqq] plot ({0.67*(-1-\x*\x)/(1-\x*\x)},{0.67*(-2)*\x/(1-\x*\x)});
\draw [samples=50,domain=-0.99:0.99,rotate around={0:(-1.33,0)},xshift=-5.33cm,yshift=0cm,color=ffqqqq] plot ({0.67*(1+\x*\x)/(1-\x*\x)},{0.67*2*\x/(1-\x*\x)});
\draw [samples=50,domain=-0.99:0.99,rotate around={0:(-1.33,0)},xshift=-5.33cm,yshift=0cm,color=ffqqqq] plot ({0.67*(-1-\x*\x)/(1-\x*\x)},{0.67*(-2)*\x/(1-\x*\x)});
\draw [samples=50,domain=-0.99:0.99,rotate around={90:(0,-1.33)},xshift=0cm,yshift=-5.33cm,color=ffqqqq] plot ({0.67*(1+\x*\x)/(1-\x*\x)},{0.67*2*\x/(1-\x*\x)});
\draw [samples=50,domain=-0.99:0.99,rotate around={90:(0,-1.33)},xshift=0cm,yshift=-5.33cm,color=ffqqqq] plot ({0.67*(-1-\x*\x)/(1-\x*\x)},{0.67*(-2)*\x/(1-\x*\x)});
\draw [samples=50,domain=-0.99:0.99,rotate around={90:(0,1.33)},xshift=0cm,yshift=5.33cm,color=ffqqqq] plot ({0.67*(1+\x*\x)/(1-\x*\x)},{0.67*2*\x/(1-\x*\x)});
\draw [samples=50,domain=-0.99:0.99,rotate around={90:(0,1.33)},xshift=0cm,yshift=5.33cm,color=ffqqqq] plot ({0.67*(-1-\x*\x)/(1-\x*\x)},{0.67*(-2)*\x/(1-\x*\x)});
\draw [samples=50,domain=-0.2:0.2,rotate around={0:(6,0)},xshift=24cm,yshift=0cm] plot ({5.66*(1+\x*\x)/(1-\x*\x)},{5.66*2*\x/(1-\x*\x)});
\draw [samples=50,domain=-0.3:0.3,rotate around={0:(6,0)},xshift=24cm,yshift=0cm] plot ({5.66*(-1-\x*\x)/(1-\x*\x)},{5.66*(-2)*\x/(1-\x*\x)});
\draw [samples=50,domain=-0.2:0.2,rotate around={90:(0,6)},xshift=0cm,yshift=24cm] plot ({5.66*(1+\x*\x)/(1-\x*\x)},{5.66*2*\x/(1-\x*\x)});
\draw [samples=50,domain=-0.2:0.2,rotate around={90:(0,6)},xshift=0cm,yshift=24cm] plot ({5.66*(-1-\x*\x)/(1-\x*\x)},{5.66*(-2)*\x/(1-\x*\x)});
\draw [samples=50,domain=-0.2:0.2,rotate around={180:(-6,0)},xshift=-24cm,yshift=0cm] plot ({5.66*(1+\x*\x)/(1-\x*\x)},{5.66*2*\x/(1-\x*\x)});
\draw [samples=50,domain=-0.2:0.2,rotate around={180:(-6,0)},xshift=-24cm,yshift=0cm] plot ({5.66*(-1-\x*\x)/(1-\x*\x)},{5.66*(-2)*\x/(1-\x*\x)});
\draw [samples=50,domain=-0.2:0.2,rotate around={-90:(0,-6)},xshift=0cm,yshift=-24cm] plot ({5.66*(1+\x*\x)/(1-\x*\x)},{5.66*2*\x/(1-\x*\x)});
\draw [samples=50,domain=-0.2:0.2,rotate around={-90:(0,-6)},xshift=0cm,yshift=-24cm] plot ({5.66*(-1-\x*\x)/(1-\x*\x)},{5.66*(-2)*\x/(1-\x*\x)});
\begin{scriptsize}
\fill [color=uququq] (0,0) circle (1.5pt);
\draw[color=uququq] (0.04,0.06) node {$O$};
\fill [color=uququq] (-0.33,0.33) circle (1.5pt);
\draw[color=uququq] (-0.3,0.39) node {$A$};
\fill [color=uququq] (0.33,0.33) circle (1.5pt);
\draw[color=uququq] (0.37,0.39) node {$B$};
\fill [color=uququq] (0.33,-0.33) circle (1.5pt);
\draw[color=uququq] (0.37,-0.27) node {$C$};
\fill [color=uququq] (-0.33,-0.33) circle (1.5pt);
\draw[color=uququq] (-0.3,-0.27) node {$D$};
\fill [color=uququq] (-0.5,0.5) circle (1.5pt);
\draw[color=uququq] (-0.47,0.56) node {$E$};
\fill [color=uququq] (0.5,0.5) circle (1.5pt);
\draw[color=uququq] (0.54,0.56) node {$F$};
\fill [color=uququq] (0.5,-0.5) circle (1.5pt);
\draw[color=uququq] (0.54,-0.44) node {$G$};
\fill [color=uququq] (-0.5,-0.5) circle (1.5pt);
\draw[color=uququq] (-0.48,-0.44) node {$H$};
\draw (1.9,-0.05) node[anchor=north west] {$\VVx$};
\draw (-0.4,2.2) node[anchor=north west] {$\VVy$};
\end{scriptsize}
\end{tikzpicture}
\end{center}
\caption{$\Li$ stability areas for $\vects=(0,1,1,3/2)$. Scheme relative to $\utilde=\vectV$ ($A,B,C,D$), to $\utilde=\vectz$ ($E,F,G,H$), with a non intrinsic diffusion.} 
\label{fig:compzone2}
\end{figure}
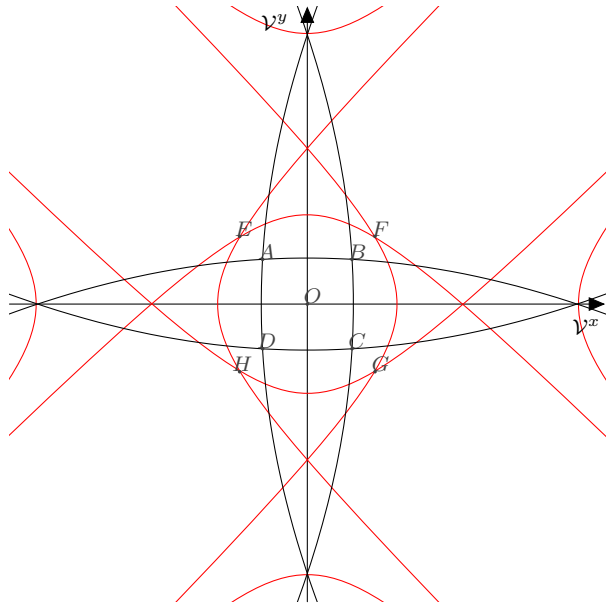

\begin{figure}
\begin{center}
\definecolor{ffqqqq}{rgb}{1,0,0}
\definecolor{uququq}{rgb}{0.25,0.25,0.25}
\definecolor{xdxdff}{rgb}{0.49,0.49,1}
\begin{tikzpicture}[scale=0.7,line cap=round,line join=round,>=triangle 45,x=4.0cm,y=4.0cm]
\draw[->,color=black] (-1.5,0) -- (1.5,0);
\draw[->,color=black] (0,-1.5) -- (0,1.5);
\clip(-1.5,-1.5) rectangle (1.5,1.5);
\draw [domain=-1.5:1.5] plot(\x,{(-0.67--1*\x)/1});
\draw [domain=-1.5:1.5] plot(\x,{(-0.67-1*\x)/-1});
\draw [domain=-1.5:1.5] plot(\x,{(-0.67-1*\x)/1});
\draw [domain=-1.5:1.5] plot(\x,{(-0.67--1*\x)/-1});
\draw [samples=50,domain=-0.2:0.2,rotate around={0:(4,0)},xshift=16cm,yshift=0cm,color=ffqqqq] plot ({3.83*(-1-\x*\x)/(1-\x*\x)},{3.83*(-2)*\x/(1-\x*\x)});
\draw [samples=50,domain=-0.2:0.2,rotate around={0:(-4,0)},xshift=-16cm,yshift=0cm,color=ffqqqq] plot ({3.83*(1+\x*\x)/(1-\x*\x)},{3.83*2*\x/(1-\x*\x)});
\draw [samples=50,domain=-0.2:0.2,rotate around={90:(0,-4)},xshift=0cm,yshift=-16cm,color=ffqqqq] plot ({3.83*(1+\x*\x)/(1-\x*\x)},{3.83*2*\x/(1-\x*\x)});
\draw [samples=50,domain=-0.2:0.2,rotate around={90:(0,4)},xshift=0cm,yshift=16cm,color=ffqqqq] plot ({3.83*(-1-\x*\x)/(1-\x*\x)},{3.83*(-2)*\x/(1-\x*\x)});
\begin{scriptsize}
\fill [color=uququq] (0,0.67) circle (1.5pt);
\draw[color=uququq] (0.1,0.7) node {$A$};
\fill [color=uququq] (-0.67,0) circle (1.5pt);
\draw[color=uququq] (-0.65,0.1) node {$B$};
\fill [color=uququq] (0,-0.67) circle (1.5pt);
\draw[color=uququq] (0.1,-0.66) node {$C$};
\fill [color=uququq] (0.67,0) circle (1.5pt);
\draw[color=uququq] (0.69,0.1) node {$D$};
\fill [color=uququq] (-0.17,0.17) circle (1.5pt);
\draw[color=uququq] (-0.15,0.21) node {$E$};
\fill [color=uququq] (0.17,0.17) circle (1.5pt);
\draw[color=uququq] (0.19,0.21) node {$F$};
\fill [color=uququq] (0.17,-0.17) circle (1.5pt);
\draw[color=uququq] (0.19,-0.13) node {$G$};
\fill [color=uququq] (-0.17,-0.17) circle (1.5pt);
\draw[color=uququq] (-0.14,-0.13) node {$H$};
\draw (1.3,-0.05) node[anchor=north west] {$\VVx$};
\draw (-0.2,1.5) node[anchor=north west] {$\VVy$};
\end{scriptsize}
\end{tikzpicture}
\end{center}
\caption{$\Li$ stability areas for $\vects=(0,3/2,3/2,3/4)$. Scheme relative to $\utilde=\vectV$ ($A,B,C,D$), to $\utilde=\vectz$ ($E,F,G,H$), with a non intrinsic diffusion.}  
\label{fig:compzone3}
\end{figure}

We can not say that a scheme is better than the other in terms of $\Li$ stability: generally the areas just intersect.
In the section \ref{sub:stabeqeq}, we saw that the velocity field $\utilde=\vectV$ improve the stability associated with the blow up of the scheme that rather corresponds to a $L^2$ notion than to a $\Li$ notion. We confirm this fact theoretically in the section \ref{sub:weiL2}.

\subsection{The intrinsic diffusion case}

As for the non intrinsic case, we first deal with the MRT scheme corresponding to $\utilde=\vectz$.

\begin{proposition}[$\Li$ stability areas for the MRT scheme]\label{th:stlitrtis0}
Let $\vectV\in\R^2$, $(\sk[q],\sk[xy])\in\R^2$, consider the twisted $\ddqqn$ MRT scheme $(\utilde=\vectz)$ associated with the relaxation parameters $(0,\sk[q],\sk[q],\sk[xy])$, and the equilibrium (\ref{eq:eqd2q4},\ref{eq:eqd2q4treis}) $$\vectmeqz = \rho(1,\Vx,\Vy,0).$$ Note $\gamma=\sk[xy]/\sk[q]$, 
\begin{itemize}
\item if $0\leq\sk[xy]\leq\min(\sk[q],2-\sk[q])$ (area BCD on the figure \ref{fig:zonzst_dhtwis}), the scheme is $\Li$ stable for all $\vectV$ such that 
\begin{equation}\label{eq:IS1}
\normu\leq\lambda\gamma.
\end{equation}
\item if $\sk[q]\leq\min(1,\sk[xy])$ and $\sk[xy]\leq2\sk[q]$ (area ABC on the figure \ref{fig:zonzst_dhtwis}), the scheme is $\Li$ stable for all $\vectV$ such that 
\begin{equation}\label{eq:IS2}
\normu\leq\lambda(2-\gamma).
\end{equation}
\item if $\sk[q]\geqslant\max(1,2-\sk[xy])$ and $\sk[xy]\leq2(2-\sk[q])$ (area ACD on the figure \ref{fig:zonzst_dhtwis}), the scheme is $\Li$ stable for all $\vectV$ such that 
\begin{equation}\label{eq:IS3}
\normu\leq\lambda\big(\frac{4}{\sk[q]}-2-\gamma\big).
\end{equation}
\item if $\sk[xy]=\sk[q]=0$ (point B on the figure \ref{fig:zonzst_dhtwis}), the scheme is unconditionally $\Li$ stable.
\item For all other $\vects$, no $\vectV$ corresponds to a $\Li$ stable scheme.
\end{itemize}
In particular, the parameters $(\sk[q],\sk[xy])$ corresponding to a non empty area of $\Li$ stability in $\vectV$ are included in the square $[0,2]^2$.
\end{proposition}

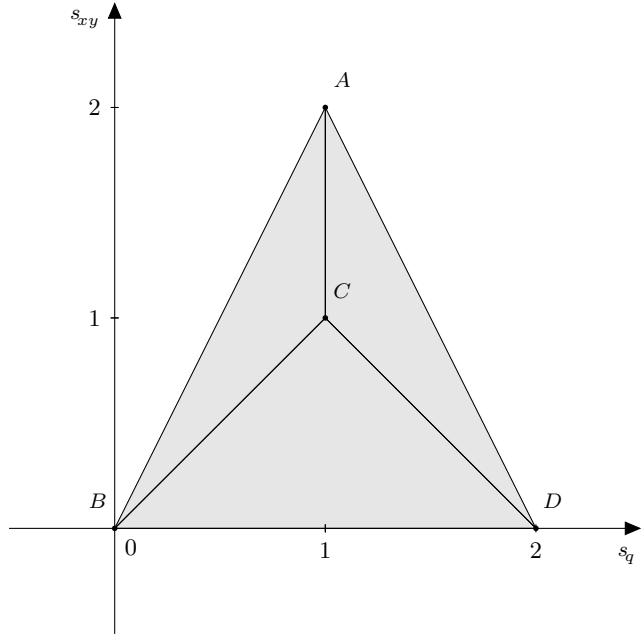
\begin{figure}
\begin{center}
\definecolor{ffqqtt}{rgb}{1,0,0.2}
\definecolor{ttccqq}{rgb}{0.2,0.8,0}
\definecolor{xdxdff}{rgb}{0.49,0.49,1}
\definecolor{qqttff}{rgb}{0,0.2,1}
\definecolor{uququq}{rgb}{0.25,0.25,0.25}
\definecolor{qqqqff}{rgb}{0,0,1}
\begin{tikzpicture}[scale=0.7,line cap=round,line join=round,>=triangle 45,x=4.0cm,y=4.0cm]
\draw[->,color=black] (-0.5,0) -- (2.5,0);
\foreach \x in {,1,2}
\draw[shift={(\x,0)},color=black] (0pt,2pt) -- (0pt,-2pt) node[below] {\footnotesize $\x$};
\draw[->,color=black] (0,-0.5) -- (0,2.5);
\foreach \y in {,1,2}
\draw[shift={(0,\y)},color=black] (2pt,0pt) -- (-2pt,0pt) node[left] {\footnotesize $\y$};
\draw[color=black] (0pt,-10pt) node[right] {\footnotesize $0$};
\clip(-0.5,-0.5) rectangle (2.5,2.5);
\fill[color=black,fill=black,fill opacity=0.1] (1,2) -- (0,0) -- (1,1) -- cycle;
\fill[color=black,fill=black,fill opacity=0.1] (1,2) -- (1,1) -- (2,0) -- cycle;
\fill[color=black,fill=black,fill opacity=0.1] (0,0) -- (2,0) -- (1,1) -- cycle;
\draw [color=black] (1,2)-- (0,0);
\draw [color=black] (0,0)-- (1,1);
\draw [color=black] (1,1)-- (1,2);
\draw [color=black] (1,2)-- (1,1);
\draw [color=black] (1,1)-- (2,0);
\draw [color=black] (2,0)-- (1,2);
\draw [color=black] (0,0)-- (2,0);
\draw [color=black] (2,0)-- (1,1);
\draw [color=black] (1,1)-- (0,0);
\begin{scriptsize}
\draw (2.35,-0.05) node[anchor=north west] {$\sk[q]$};
\draw (-0.25,2.5) node[anchor=north west] {$\sk[xy]$};
\fill [color=black] (1,2) circle (1.5pt);
\draw[color=black] (1.08,2.13) node {$A$};
\fill [color=black] (0,0) circle (1.5pt);
\draw[color=black] (-0.08,0.13) node {$B$};
\fill [color=black] (1,1) circle (1.5pt);
\draw[color=black] (1.08,1.13) node {$C$};
\fill [color=black] (2,0) circle (1.5pt);
\draw[color=black] (2.08,0.13) node {$D$};
\end{scriptsize}
\end{tikzpicture}
\end{center}
\caption{$\Li$ stability area in $\vects$ of the twisted $\ddqqn$ MRT scheme ($\utilde=\vectz$) with an intrinsic diffusion.
BCD : (\ref{eq:IS1}), ABC : (\ref{eq:IS2}), ACD : (\ref{eq:IS3}).} 
\label{fig:zonzst_dhtwis}
\end{figure}

\begin{proof}
The positivity of the matrix sending $\vectf$ on $\vectfe$ reads
\begin{equation}\label{eq:is1}
\min(\frac{\sk[xy]}{4},\frac{2\sk[q]-\sk[xy]}{4},1-\frac{2\sk[q]+\sk[xy]}{4})\pm\frac{\sk[q]}{4\lambda}(\Vx\pm\Vy)\geqslant0.
\end{equation}
If $\sk[q]=0$, then the inequalities (\ref{eq:is1}) impose $\sk[xy]=0$ and there is unconditional stability in $\vectV$. Otherwise these inequalities are equivalent to
\begin{equation}\label{eq:is7}
\normu\leq\min\big(\lambda\gamma,\lambda(2-\gamma),\lambda(\frac{4}{\sk[q]}-2-\gamma)\big).
\end{equation}

The conditions to have a non empty set of $\vectV$ are to be determined: for that we must study when the minimum in (\ref{eq:is7}) is nonnegative.
\begin{align*}
\gamma\geqslant0&\iff\sk[xy]\geqslant0,\\
2-\gamma\geqslant0&\iff\sk[xy]\leq2\sk[q],\\
\frac{4}{\sk[q]}-2-\gamma\geqslant0&\iff\sk[xy]\leq2(2-\sk[q]).
\end{align*}
We now suppose that these conditions are verified and we determine the minimum according to the choice of $\vects$:
if $\sk[xy]\leq\min(\sk[q],2-\sk[q])$, the area is given by (\ref{eq:IS1}),
if $\sk[q]\leq\min(1,\sk[xy])$, it is given by (\ref{eq:IS2}),
if $\sk[q]\geqslant\max(1,2-\sk[xy])$, we obtain (\ref{eq:IS3}).
This closes the proof.
\end{proof}

The areas are represented on the figure \ref{fig:zonzst_dhtwis}.
We can deduce the $\Li$ stability areas for the BGK scheme.

\begin{proposition}[The BGK case]\label{th:stlitrtisbgk}
Let $\vectV\in\R^2$, $s\in\R$, consider the twisted relative velocity $\ddqqn$ scheme, BGK of relaxation parameters $s$, and of equilibrium given by (\ref{eq:eqd2q4},\ref{eq:eqd2q4treis}) 
\begin{itemize}
\item if $0\leq s\leq1$, the scheme is $\Li$ stable for all $\vectV$ such that 
\begin{equation*}
\normu\leq\lambda.
\end{equation*}
\item if $1\leq s\leq 4/3$, the scheme is $\Li$ stable for all $\vectV$ such that 
\begin{equation*}
\normu\leq\lambda\big(\frac{4}{s}-3\big).
\end{equation*}
\item For all other $\vects$, no $\vectV$ corresponds to a $\Li$ stable scheme.
\end{itemize}
In particular, the parameter $s$ corresponding to a non empty area of $\Li$ stability in $\vectV$ are included in $[0,2]$.
\end{proposition}

We obtain a constant area of stability for $s\leq1$. When $s$ becomes larger than one (overrelaxation) the area decreases as $s$ increases. For $s$ greater than $4/3$, there is no velocity $\Li$ stable. Now we do the same job for  the TRT scheme with $\utilde=\vectV$.

\begin{proposition}[$\Li$ stability areas for a relative velocity scheme]\label{th:stlitrtisV}
Let $\vectV\in\R^2$, $(\sk[q],\sk[xy])\in\R^2$, consider the twisted $\ddqqn$ scheme relative to $\utilde=\vectV$ associated with the relaxation parameters $(0,\sk[q],\sk[q],\sk[xy])$, and the equilibrium (\ref{eq:eqd2q4},\ref{eq:eqd2q4treis}) $$\vectmeq(\vectV) = \rho(1,0,0,-\Vx\Vy).$$ Note $\gu[1]=(2\sk[q]-\sk[xy])/(\sk[q]-\sk[xy])$ and $\gu[2]=\sk[xy]/(\sk[q]-\sk[xy])$,
\begin{itemize}
\item if $\sk[q]<\sk[xy]\leq\min(2\sk[q],2(2-\sk[q]))$ (area ACD on the figure \ref{fig:zonzst_rectwis}), the scheme is stable for all $\vectV$ such that 
\begin{subequations}
\begin{gather}
(\VVx\pm\lambda\gu[1])^2-(\VVy)^2\leq \lambda^2\gu[1]\gu[2],\label{eq:hypisleq1}\\
(\VVx\pm\lambda\frac{\gu[1]+\gu[2]}{2})^2-(\VVy\pm\lambda)^2\leq\lambda^2\varp{\gamma}{2}{2},\label{eq:hypisleq3}\\
(\VVx\pm\lambda\frac{\gu[1]+\gu[2]}{2})^2-(\VVy\mp\lambda)^2\leq\lambda^2\varp{\gamma}{2}{2},\label{eq:hypisleq2}\\
(\VVx\pm\lambda\gu[2])^2-(\VVy)^2\leq\frac{\lambda^2}{(\sk[q]-\sk[xy])^2}(4\varp{s}{q}{2}-2\sk[q]\sk[xy]-\varp{s}{xy}{2}+8(\sk[xy]-\sk[q])), \label{eq:hypisleq4}
\end{gather} 
\end{subequations}
plus the analogous inequalities exchanging $\VVx$ and $\VVy$.
\item if $\sk[xy]\leq\min(\sk[q],2(2-\sk[q]))$ (area ABD on the figure \ref{fig:zonzst_rectwis}), the scheme is stable for all $\vectV$ such that
\begin{subequations}
\begin{gather}
(\VVx\pm\lambda\gu[1])^2-(\VVy)^2\geqslant \lambda^2\gu[1]\gu[2],\label{eq:hypisgeq1}\\
(\VVx\pm\lambda\frac{\gu[1]+\gu[2]}{2})^2-(\VVy\pm\lambda)^2\geqslant\lambda^2\varp{\gamma}{2}{2},\label{eq:hypisgeq3}\\
(\VVx\pm\lambda\frac{\gu[1]+\gu[2]}{2})^2-(\VVy\mp\lambda)^2\geqslant\lambda^2\varp{\gamma}{2}{2},\label{eq:hypisgeq2}\\
(\VVx\pm\lambda\gu[2])^2-(\VVy)^2\geqslant\frac{\lambda^2}{(\sk[q]-\sk[xy])^2}(4\varp{s}{q}{2}-2\sk[q]\sk[xy]-\varp{s}{xy}{2}+8(\sk[xy]-\sk[q])),\label{eq:hypisgeq4}
\end{gather} 
\end{subequations}
plus the analogous inequalities exchanging $\VVx$ and $\VVy$.
\item if $\sk[xy]=\sk[q]$ (segment [AD] on the figure \ref{fig:zonzst_rectwis}), the $\Li$ stability area is given by the proposition \ref{th:stlitrtisbgk}.
\item For all other $\vects$, no $\vectV$ corresponds to a $\Li$ stable scheme.
\end{itemize}
In particular, the parameters $(\sk[q],\sk[xy])$ corresponding to a non empty area of $\Li$ stability in $\vectV$ are included in the square $[0,2]^2$.
\end{proposition}

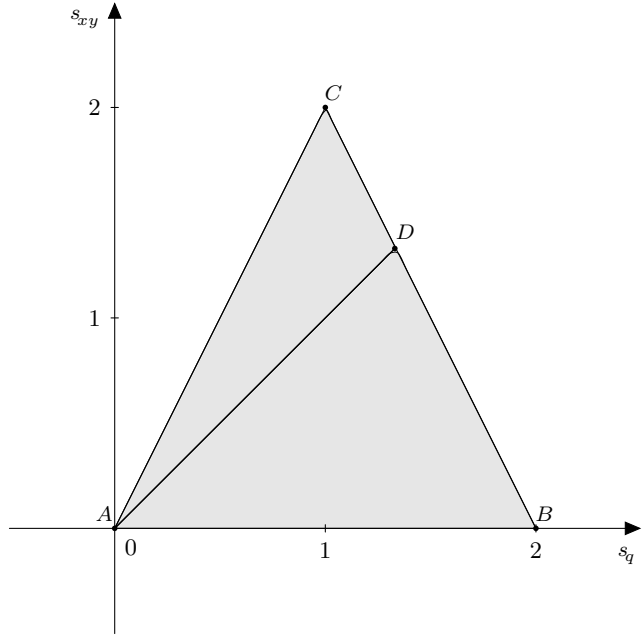
\begin{figure}
\begin{center}
\definecolor{uququq}{rgb}{0.25,0.25,0.25}
\definecolor{xdxdff}{rgb}{0.49,0.49,1}
\definecolor{qqqqff}{rgb}{0,0,1}
\definecolor{ffqqqq}{rgb}{1,0,0}
\begin{tikzpicture}[scale=0.7,line cap=round,line join=round,>=triangle 45,x=4.0cm,y=4.0cm]
\draw[->,color=black] (-0.5,0) -- (2.5,0);
\foreach \x in {1,2}
\draw[shift={(\x,0)},color=black] (0pt,2pt) -- (0pt,-2pt) node[below] {\footnotesize $\x$};
\draw[->,color=black] (0,-0.5) -- (0,2.5);
\foreach \y in {1,2}
\draw[shift={(0,\y)},color=black] (2pt,0pt) -- (-2pt,0pt) node[left] {\footnotesize $\y$};
\draw[color=black] (0pt,-10pt) node[right] {\footnotesize $0$};
\clip(-0.5,-0.5) rectangle (2.5,2.5);
\draw[color=black,fill=black,fill opacity=0.1] {[smooth,samples=50,domain=0.0:1.33333] plot(\x,{(4-2*\x+2*\x-abs(4-2*\x-2*\x))/2})} -- (1.33,1.33) {[smooth,samples=50,domain=1.33333:0.0] -- plot(\x,{\x})} -- (0,0) -- cycle;
\draw[color=black,fill=black,fill opacity=0.1, smooth,samples=50,domain=0.0:2.0] plot(\x,{(4-2*\x+\x-abs(4-2*\x-\x))/2}) -- (2,0) -- (0,0) -- cycle;
\draw[smooth,samples=100,domain=1.0:2.0] plot(\x,{4-2*(\x)});
\draw[smooth,samples=100,domain=0.0:1.0] plot(\x,{2*(\x)});
\draw[smooth,samples=100,domain=0.0:1.3333333333333333] plot(\x,{(\x)});
\begin{scriptsize}
\draw (2.35,-0.05) node[anchor=north west] {$\sk[q]$};
\draw (-0.25,2.5) node[anchor=north west] {$\sk[xy]$};
\fill [color=black] (0,0) circle (1.5pt);
\draw[color=black] (-0.05,0.07) node {$A$};
\fill [color=black] (2,0) circle (1.5pt);
\draw[color=black] (2.04,0.07) node {$B$};
\fill [color=black] (1,2) circle (1.5pt);
\draw[color=black] (1.04,2.07) node {$C$};
\fill [color=black] (1.33,1.33) circle (1.5pt);
\draw[color=black] (1.38,1.41) node {$D$};
\end{scriptsize}
\end{tikzpicture}
\end{center}
\caption{$\Li$ stability area in $\vects$ of the twisted $\ddqqn$ scheme relative to $\utilde=\vectV$ with an intrinsic diffusion.
ACD : (\ref{eq:hypisleq1}) to (\ref{eq:hypisleq4}), ABD : (\ref{eq:hypisgeq1}) to (\ref{eq:hypisgeq4}).} 
\label{fig:zonzst_rectwis}
\end{figure}

We now compare the $\Li$ stability areas of the schemes relative to the velocities $\utilde=\vectz$ or $\utilde=\vectV$ with an intrinsic diffusion for different $\vects$. We draw the areas given by the propositions \ref{th:stlitrtis0} and \ref{th:stlitrtisV} on the figures \ref{fig:compzoneis1} and \ref{fig:compzoneis2}.

\begin{figure}
\begin{center}
\definecolor{ffqqqq}{rgb}{1,0,0}
\definecolor{xdxdff}{rgb}{0.49,0.49,1}
\definecolor{uququq}{rgb}{0.25,0.25,0.25}
\begin{tikzpicture}[scale=0.45,line cap=round,line join=round,>=triangle 45,x=4.0cm,y=4.0cm]
\draw[->,color=black] (-2,0) -- (2,0);
\draw[->,color=black] (0,-2) -- (0,2);
\clip(-2,-2) rectangle (2,2);
\draw [samples=50,domain=-0.5:0.5,rotate around={0:(-1,0)},xshift=-4cm,yshift=0cm] plot ({1.73*(1+\x*\x)/(1-\x*\x)},{1.73*2*\x/(1-\x*\x)});
\draw [samples=50,domain=-0.5:0.5,rotate around={0:(-1,0)},xshift=-4cm,yshift=0cm] plot ({1.73*(-1-\x*\x)/(1-\x*\x)},{1.73*(-2)*\x/(1-\x*\x)});
\draw [samples=50,domain=-0.5:0.5,rotate around={180:(1,0)},xshift=4cm,yshift=0cm] plot ({1.73*(1+\x*\x)/(1-\x*\x)},{1.73*2*\x/(1-\x*\x)});
\draw [samples=50,domain=-0.5:0.5,rotate around={180:(1,0)},xshift=4cm,yshift=0cm] plot ({1.73*(-1-\x*\x)/(1-\x*\x)},{1.73*(-2)*\x/(1-\x*\x)});
\draw [samples=50,domain=-0.5:0.5,rotate around={90:(0,-1)},xshift=0cm,yshift=-4cm] plot ({1.73*(1+\x*\x)/(1-\x*\x)},{1.73*2*\x/(1-\x*\x)});
\draw [samples=50,domain=-0.5:0.5,rotate around={90:(0,-1)},xshift=0cm,yshift=-4cm] plot ({1.73*(-1-\x*\x)/(1-\x*\x)},{1.73*(-2)*\x/(1-\x*\x)});
\draw [samples=50,domain=-0.5:0.5,rotate around={-90:(0,1)},xshift=0cm,yshift=4cm] plot ({1.73*(1+\x*\x)/(1-\x*\x)},{1.73*2*\x/(1-\x*\x)});
\draw [samples=50,domain=-0.5:0.5,rotate around={-90:(0,1)},xshift=0cm,yshift=4cm] plot ({1.73*(-1-\x*\x)/(1-\x*\x)},{1.73*(-2)*\x/(1-\x*\x)});
\draw [samples=50,domain=-0.5:0.5,rotate around={0:(2,1)},xshift=8cm,yshift=4cm] plot ({3*(1+\x*\x)/(1-\x*\x)},{3*2*\x/(1-\x*\x)});
\draw [samples=50,domain=-0.5:0.5,rotate around={0:(2,1)},xshift=8cm,yshift=4cm] plot ({3*(-1-\x*\x)/(1-\x*\x)},{3*(-2)*\x/(1-\x*\x)});
\draw [samples=50,domain=-0.5:0.5,rotate around={0:(-2,-1)},xshift=-8cm,yshift=-4cm] plot ({3*(1+\x*\x)/(1-\x*\x)},{3*2*\x/(1-\x*\x)});
\draw [samples=50,domain=-0.5:0.5,rotate around={0:(-2,-1)},xshift=-8cm,yshift=-4cm] plot ({3*(-1-\x*\x)/(1-\x*\x)},{3*(-2)*\x/(1-\x*\x)});
\draw [samples=50,domain=-0.5:0.5,rotate around={0:(2,-1)},xshift=8cm,yshift=-4cm] plot ({3*(1+\x*\x)/(1-\x*\x)},{3*2*\x/(1-\x*\x)});
\draw [samples=50,domain=-0.5:0.5,rotate around={0:(2,-1)},xshift=8cm,yshift=-4cm] plot ({3*(-1-\x*\x)/(1-\x*\x)},{3*(-2)*\x/(1-\x*\x)});
\draw [samples=50,domain=-0.5:0.5,rotate around={0:(-2,1)},xshift=-8cm,yshift=4cm] plot ({3*(1+\x*\x)/(1-\x*\x)},{3*2*\x/(1-\x*\x)});
\draw [samples=50,domain=-0.5:0.5,rotate around={0:(-2,1)},xshift=-8cm,yshift=4cm] plot ({3*(-1-\x*\x)/(1-\x*\x)},{3*(-2)*\x/(1-\x*\x)});
\draw [samples=50,domain=-0.5:0.5,rotate around={90:(1,-2)},xshift=4cm,yshift=-8cm] plot ({3*(1+\x*\x)/(1-\x*\x)},{3*2*\x/(1-\x*\x)});
\draw [samples=50,domain=-0.5:0.5,rotate around={90:(1,-2)},xshift=4cm,yshift=-8cm] plot ({3*(-1-\x*\x)/(1-\x*\x)},{3*(-2)*\x/(1-\x*\x)});
\draw [samples=50,domain=-0.5:0.5,rotate around={90:(1,2)},xshift=4cm,yshift=8cm] plot ({3*(1+\x*\x)/(1-\x*\x)},{3*2*\x/(1-\x*\x)});
\draw [samples=50,domain=-0.5:0.5,rotate around={90:(1,2)},xshift=4cm,yshift=8cm] plot ({3*(-1-\x*\x)/(1-\x*\x)},{3*(-2)*\x/(1-\x*\x)});
\draw [samples=50,domain=-0.5:0.5,rotate around={90:(-1,2)},xshift=-4cm,yshift=8cm] plot ({3*(1+\x*\x)/(1-\x*\x)},{3*2*\x/(1-\x*\x)});
\draw [samples=50,domain=-0.5:0.5,rotate around={90:(-1,2)},xshift=-4cm,yshift=8cm] plot ({3*(-1-\x*\x)/(1-\x*\x)},{3*(-2)*\x/(1-\x*\x)});
\draw [samples=50,domain=-0.5:0.5,rotate around={90:(-1,-2)},xshift=-4cm,yshift=-8cm] plot ({3*(1+\x*\x)/(1-\x*\x)},{3*2*\x/(1-\x*\x)});
\draw [samples=50,domain=-0.5:0.5,rotate around={90:(-1,-2)},xshift=-4cm,yshift=-8cm] plot ({3*(-1-\x*\x)/(1-\x*\x)},{3*(-2)*\x/(1-\x*\x)});
\draw [samples=50,domain=-0.5:0.5,rotate around={0:(3,0)},xshift=12cm,yshift=0cm] plot ({3.32*(1+\x*\x)/(1-\x*\x)},{3.32*2*\x/(1-\x*\x)});
\draw [samples=50,domain=-0.5:0.5,rotate around={0:(3,0)},xshift=12cm,yshift=0cm] plot ({3.32*(-1-\x*\x)/(1-\x*\x)},{3.32*(-2)*\x/(1-\x*\x)});
\draw [samples=50,domain=-0.5:0.5,rotate around={90:(0,3)},xshift=0cm,yshift=12cm] plot ({3.32*(1+\x*\x)/(1-\x*\x)},{3.32*2*\x/(1-\x*\x)});
\draw [samples=50,domain=-0.5:0.5,rotate around={90:(0,3)},xshift=0cm,yshift=12cm] plot ({3.32*(-1-\x*\x)/(1-\x*\x)},{3.32*(-2)*\x/(1-\x*\x)});
\draw [samples=50,domain=-0.5:0.5,rotate around={180:(-3,0)},xshift=-12cm,yshift=0cm] plot ({3.32*(1+\x*\x)/(1-\x*\x)},{3.32*2*\x/(1-\x*\x)});
\draw [samples=50,domain=-0.5:0.5,rotate around={180:(-3,0)},xshift=-12cm,yshift=0cm] plot ({3.32*(-1-\x*\x)/(1-\x*\x)},{3.32*(-2)*\x/(1-\x*\x)});
\draw [samples=50,domain=-0.5:0.5,rotate around={-90:(0,-3)},xshift=0cm,yshift=-12cm] plot ({3.32*(1+\x*\x)/(1-\x*\x)},{3.32*2*\x/(1-\x*\x)});
\draw [samples=50,domain=-0.5:0.3,rotate around={-90:(0,-3)},xshift=0cm,yshift=-12cm] plot ({3.32*(-1-\x*\x)/(1-\x*\x)},{3.32*(-2)*\x/(1-\x*\x)});
\draw [color=ffqqqq] (0.5,-2) -- (0.5,2);
\draw [color=ffqqqq] (-0.5,-2) -- (-0.5,2);
\draw [color=ffqqqq,domain=-2:2] plot(\x,{(--0.5-0*\x)/1});
\draw [color=ffqqqq,domain=-2:2] plot(\x,{(-0.5-0*\x)/1});
\begin{scriptsize}
\fill [color=uququq] (0,0) circle (1.5pt);
\draw[color=uququq] (0.07,0.1) node {$O$};
\fill [color=uququq] (-0.33,0.33) circle (1.5pt);
\draw[color=uququq] (-0.27,0.44) node {$A$};
\fill [color=uququq] (0.33,0.33) circle (1.5pt);
\draw[color=uququq] (0.4,0.44) node {$B$};
\fill [color=uququq] (0.33,-0.33) circle (1.5pt);
\draw[color=uququq] (0.39,-0.23) node {$C$};
\fill [color=uququq] (-0.33,-0.33) circle (1.5pt);
\draw[color=uququq] (-0.28,-0.23) node {$D$};
\fill [color=uququq] (-0.5,0.5) circle (1.5pt);
\draw[color=uququq] (-0.43,0.61) node {$E$};
\fill [color=uququq] (0.5,0.5) circle (1.5pt);
\draw[color=uququq] (0.57,0.61) node {$F$};
\fill [color=uququq] (0.5,-0.5) circle (1.5pt);
\draw[color=uququq] (0.54,-0.39) node {$G$};
\fill [color=uququq] (-0.5,-0.5) circle (1.5pt);
\draw[color=uququq] (-0.45,-0.39) node {$H$};
\draw (1.7,-0.05) node[anchor=north west] {$\VVx$};
\draw (-0.3,2.) node[anchor=north west] {$\VVy$};
\end{scriptsize}
\end{tikzpicture}
\end{center}
\caption{$\Li$ stability areas for $\vects=(0,1,1,3/2)$. Scheme relative to $\utilde=\vectV$ ($A,B,C,D$), $\utilde=\vectz$ ($E,F,G,H$), with an intrinsic diffusion.}
\label{fig:compzoneis1}
\end{figure}

\begin{figure}
\begin{center}
\definecolor{xdxdff}{rgb}{0.49,0.49,1}
\definecolor{ffqqqq}{rgb}{1,0,0}
\definecolor{uququq}{rgb}{0.25,0.25,0.25}
\begin{tikzpicture}[scale=0.45,line cap=round,line join=round,>=triangle 45,x=4.0cm,y=4.0cm]
\draw[->,color=black] (-2,0) -- (2,0);
\draw[->,color=black] (0,-2) -- (0,2);
\clip(-2,-2) rectangle (2,2);
\draw [samples=50,domain=-0.8:0.8,rotate around={90:(0,3)},xshift=0cm,yshift=12cm] plot ({1.73*(1+\x*\x)/(1-\x*\x)},{1.73*2*\x/(1-\x*\x)});
\draw [samples=50,domain=-0.8:0.8,rotate around={90:(0,3)},xshift=0cm,yshift=12cm] plot ({1.73*(-1-\x*\x)/(1-\x*\x)},{1.73*(-2)*\x/(1-\x*\x)});
\draw [color=ffqqqq] (0.5,-2) -- (0.5,2);
\draw [color=ffqqqq,domain=-2:2] plot(\x,{(--0.5-0*\x)/1});
\draw [color=ffqqqq,domain=-2:2] plot(\x,{(-0.5-0*\x)/1});
\draw [color=ffqqqq] (-0.5,-2) -- (-0.5,2);
\draw [samples=50,domain=-0.8:0.8,rotate around={-90:(0,-3)},xshift=0cm,yshift=-12cm] plot ({1.73*(1+\x*\x)/(1-\x*\x)},{1.73*2*\x/(1-\x*\x)});
\draw [samples=50,domain=-0.8:0.8,rotate around={-90:(0,-3)},xshift=0cm,yshift=-12cm] plot ({1.73*(-1-\x*\x)/(1-\x*\x)},{1.73*(-2)*\x/(1-\x*\x)});
\draw [samples=50,domain=-0.8:0.8,rotate around={0:(3,0)},xshift=12cm,yshift=0cm] plot ({1.73*(1+\x*\x)/(1-\x*\x)},{1.73*2*\x/(1-\x*\x)});
\draw [samples=50,domain=-0.8:0.8,rotate around={0:(3,0)},xshift=12cm,yshift=0cm] plot ({1.73*(-1-\x*\x)/(1-\x*\x)},{1.73*(-2)*\x/(1-\x*\x)});
\draw [samples=50,domain=-0.8:0.8,rotate around={180:(-3,0)},xshift=-12cm,yshift=0cm] plot ({1.73*(1+\x*\x)/(1-\x*\x)},{1.73*2*\x/(1-\x*\x)});
\draw [samples=50,domain=-0.8:0.8,rotate around={180:(-3,0)},xshift=-12cm,yshift=0cm] plot ({1.73*(-1-\x*\x)/(1-\x*\x)},{1.73*(-2)*\x/(1-\x*\x)});
\draw [samples=50,domain=-0.8:0.8,rotate around={90:(1,2)},xshift=4cm,yshift=8cm] plot ({1*(1+\x*\x)/(1-\x*\x)},{1*2*\x/(1-\x*\x)});
\draw [samples=50,domain=-0.8:0.8,rotate around={90:(1,2)},xshift=4cm,yshift=8cm] plot ({1*(-1-\x*\x)/(1-\x*\x)},{1*(-2)*\x/(1-\x*\x)});
\draw [samples=50,domain=-0.8:0.8,rotate around={90:(1,-2)},xshift=4cm,yshift=-8cm] plot ({1*(1+\x*\x)/(1-\x*\x)},{1*2*\x/(1-\x*\x)});
\draw [samples=50,domain=-0.8:0.8,rotate around={90:(1,-2)},xshift=4cm,yshift=-8cm] plot ({1*(-1-\x*\x)/(1-\x*\x)},{1*(-2)*\x/(1-\x*\x)});
\draw [samples=50,domain=-0.8:0.8,rotate around={90:(-1,-2)},xshift=-4cm,yshift=-8cm] plot ({1*(1+\x*\x)/(1-\x*\x)},{1*2*\x/(1-\x*\x)});
\draw [samples=50,domain=-0.8:0.8,rotate around={90:(-1,-2)},xshift=-4cm,yshift=-8cm] plot ({1*(-1-\x*\x)/(1-\x*\x)},{1*(-2)*\x/(1-\x*\x)});
\draw [samples=50,domain=-0.8:0.8,rotate around={90:(-1,2)},xshift=-4cm,yshift=8cm] plot ({1*(1+\x*\x)/(1-\x*\x)},{1*2*\x/(1-\x*\x)});
\draw [samples=50,domain=-0.8:0.8,rotate around={90:(-1,2)},xshift=-4cm,yshift=8cm] plot ({1*(-1-\x*\x)/(1-\x*\x)},{1*(-2)*\x/(1-\x*\x)});
\draw [samples=50,domain=-0.8:0.8,rotate around={0:(2,-1)},xshift=8cm,yshift=-4cm] plot ({1*(1+\x*\x)/(1-\x*\x)},{1*2*\x/(1-\x*\x)});
\draw [samples=50,domain=-0.8:0.8,rotate around={0:(2,-1)},xshift=8cm,yshift=-4cm] plot ({1*(-1-\x*\x)/(1-\x*\x)},{1*(-2)*\x/(1-\x*\x)});
\draw [samples=50,domain=-0.8:0.8,rotate around={0:(2,1)},xshift=8cm,yshift=4cm] plot ({1*(1+\x*\x)/(1-\x*\x)},{1*2*\x/(1-\x*\x)});
\draw [samples=50,domain=-0.8:0.8,rotate around={0:(2,1)},xshift=8cm,yshift=4cm] plot ({1*(-1-\x*\x)/(1-\x*\x)},{1*(-2)*\x/(1-\x*\x)});
\draw [samples=50,domain=-0.8:0.8,rotate around={0:(-2,1)},xshift=-8cm,yshift=4cm] plot ({1*(1+\x*\x)/(1-\x*\x)},{1*2*\x/(1-\x*\x)});
\draw [samples=50,domain=-0.8:0.8,rotate around={0:(-2,1)},xshift=-8cm,yshift=4cm] plot ({1*(-1-\x*\x)/(1-\x*\x)},{1*(-2)*\x/(1-\x*\x)});
\draw [samples=50,domain=-0.8:0.8,rotate around={0:(-2,-1)},xshift=-8cm,yshift=-4cm] plot ({1*(1+\x*\x)/(1-\x*\x)},{1*2*\x/(1-\x*\x)});
\draw [samples=50,domain=-0.8:0.8,rotate around={0:(-2,-1)},xshift=-8cm,yshift=-4cm] plot ({1*(-1-\x*\x)/(1-\x*\x)},{1*(-2)*\x/(1-\x*\x)});
\draw [samples=50,domain=-0.8:0.8,rotate around={90:(-1,0)},xshift=-4cm,yshift=0cm] plot ({2.24*(1+\x*\x)/(1-\x*\x)},{2.24*2*\x/(1-\x*\x)});
\draw [samples=50,domain=-0.8:0.8,rotate around={90:(-1,0)},xshift=-4cm,yshift=0cm] plot ({2.24*(-1-\x*\x)/(1-\x*\x)},{2.24*(-2)*\x/(1-\x*\x)});
\draw [samples=50,domain=-0.8:0.8,rotate around={90:(1,0)},xshift=4cm,yshift=0cm] plot ({2.24*(1+\x*\x)/(1-\x*\x)},{2.24*2*\x/(1-\x*\x)});
\draw [samples=50,domain=-0.8:0.8,rotate around={90:(1,0)},xshift=4cm,yshift=0cm] plot ({2.24*(-1-\x*\x)/(1-\x*\x)},{2.24*(-2)*\x/(1-\x*\x)});
\draw [samples=50,domain=-0.8:0.8,rotate around={0:(0,1)},xshift=0cm,yshift=4cm] plot ({2.24*(1+\x*\x)/(1-\x*\x)},{2.24*2*\x/(1-\x*\x)});
\draw [samples=50,domain=-0.8:0.8,rotate around={0:(0,1)},xshift=0cm,yshift=4cm] plot ({2.24*(-1-\x*\x)/(1-\x*\x)},{2.24*(-2)*\x/(1-\x*\x)});
\draw [samples=50,domain=-0.8:0.8,rotate around={0:(0,-1)},xshift=0cm,yshift=-4cm] plot ({2.24*(1+\x*\x)/(1-\x*\x)},{2.24*2*\x/(1-\x*\x)});
\draw [samples=50,domain=-0.8:0.8,rotate around={0:(0,-1)},xshift=0cm,yshift=-4cm] plot ({2.24*(-1-\x*\x)/(1-\x*\x)},{2.24*(-2)*\x/(1-\x*\x)});
\begin{scriptsize}
\fill [color=uququq] (0,0) circle (1.5pt);
\draw[color=uququq] (0.03,0.06) node {$O$};
\fill [color=uququq] (-0.59,0) circle (1.5pt);
\draw[color=uququq] (-0.56,0.06) node {$A$};
\fill [color=uququq] (-0.33,0.33) circle (1.5pt);
\draw[color=uququq] (-0.25,0.25) node {$B$};
\fill [color=uququq] (0,0.59) circle (1.5pt);
\draw[color=uququq] (0.03,0.65) node {$C$};
\fill [color=uququq] (0.33,0.33) circle (1.5pt);
\draw[color=uququq] (0.25,0.25) node {$D$};
\fill [color=uququq] (0.59,0) circle (1.5pt);
\draw[color=uququq] (0.62,0.06) node {$E$};
\fill [color=uququq] (0,-0.59) circle (1.5pt);
\draw[color=uququq] (0.03,-0.53) node {$F$};
\fill [color=uququq] (-0.33,-0.33) circle (1.5pt);
\draw[color=uququq] (-0.25,-0.25) node {$G$};
\fill [color=uququq] (0.33,-0.33) circle (1.5pt);
\draw[color=uququq] (0.25,-0.25) node {$H$};
\fill [color=uququq] (-0.5,0.5) circle (1.5pt);
\draw[color=uququq] (-0.4,0.6) node {$I$};
\fill [color=uququq] (0.5,0.5) circle (1.5pt);
\draw[color=uququq] (0.58,0.6) node {$J$};
\fill [color=uququq] (0.5,-0.5) circle (1.5pt);
\draw[color=uququq] (0.6,-0.4) node {$K$};
\fill [color=uququq] (-0.5,-0.5) circle (1.5pt);
\draw[color=uququq] (-0.42,-0.4) node {$L$};
\draw (1.7,-0.05) node[anchor=north west] {$\VVx$};
\draw (-0.3,2.) node[anchor=north west] {$\VVy$};
\end{scriptsize}
\end{tikzpicture}
\end{center}
\caption{$\Li$ stability areas for $\vects=(0,3/2,3/2,3/4)$. Scheme relative to $\utilde=\vectV$  ($A,B,C,D,E,F,G,H$), $\utilde=\vectz$ ($I,J,K,L$), with an intrinsic diffusion.}
\label{fig:compzoneis2}
\end{figure}

The results are similar to the ones obtained in the non intrinsic case. We can't distinguish a scheme from an other in terms of $\Li$ stability. The MRT scheme is better on the figure \ref{fig:compzoneis1} but on the figure \ref{fig:compzoneis2}, the areas just intersect.

\section{Weighted $L^2$ stability}\label{sub:weiL2}

In this section, we present some theoretical weighted $L^2$ stability results for the relative velocity $\ddqqn$ schemes. These results, based on the notion of stability proposed by Yong and al in \cite{Yong:2006:0}, confirm the phenomena numerically observed in the section \ref{sub:stabeqeq}. In particular, taking $\utilde$ equal to the advection velocity $\vectV$ (``cascaded like'' scheme) provides good stability features. These results also extend the $\Li$ stability ones (section \ref{sub:stabLi}): the $\Li$ stability areas are included in the $L^2$ ones.

\subsection{The weighted $L^2$ stability notion}

The general framework of the relative velocity $\ddqq$ schemes, $d,q\in\N^*$, is chosen to present the stability notion introduced in \cite{Yong:2006:0} and studied in \cite{Junk:2009:0,Rhein:2010:0}.

An iteration of a lattice Boltzmann scheme splits into a transport step plus a relaxation step that is here linear since the equilibrium is linear (\ref{eq:eqd2q4},\ref{eq:eqd2q4tre},\ref{eq:eqd2q4treis}). Given $\vectf(.,t)$ the matrix composed of the distribution vectors taken on all the lattice points, we have
$$\vectf(.,t+\dt)=\MatT\MatRu\vectf(.,t),\quad t\in\R,$$
where $\MatRu$ is the relaxation matrix and $\MatT$ the linear non local transport operator. The size of the matrix $\vectf(.,t)$ is equal to the number of velocities $q$ multiplied by the number of lattice points. In the case of the relative velocity $\ddqq$ scheme, the collision matrix reads $\MatRu=\var{\MatId}{q}+\MatJu$ where $\MatJu$ is given by 
\begin{equation*}\label{eq:matju}
\MatJu=\MatMu^{-1}\MatD\MatMu(\MatB-\var{\MatI}{q}),
\end{equation*}
where $\MatD={\rm diag}(\vects)$ is the diagonal matrix of the relaxation parameters and $\MatB$ defines the linear equilibrium thanks to $\vectfeq=\MatB\vectf$.

 The calculus of the spectrum of the amplification operator $\MatT\MatRu$ being difficult because of the matrix size, the idea is to study separately the transport and the collision: we want $\MatT$ and $\MatRu$ to be bounded by one in a well-chosen norm. The amplification operator $\MatT\MatRu$ is then bounded by one and the scheme is stable for this norm. In the following, we say that an operator is stable in a given norm if the norm of this operator is bounded by one. 

The idea proposed in \cite{Yong:2006:0} consists in weighting the $L^2$ norm defined by
 \begin{equation*}
\normd{\vecty}=\big(\sum_{j=0}^{q-1}\varp{|y_j^{}|}{}{2}\big)^{\frac{1}{2}}, \quad \vecty=(y_0^{},\ldots,y_{q-1}^{})\in\R^q,
  \end{equation*}
  so that the transport is an isometry and the collision operator is stable in the weighted norm. The classical $L^2$ norm keeps the transport as isometric but the collision norm is difficult to evaluate. Introducing a weight allows to overcome this difficulty. First, we define a norm on $\R^q$ depending on an invertible matrix $\MatP\in\espM(\R)$, $$\normepv{\vecty}=\normd{\MatP\vecty},\quad \vecty\in\R^q.$$
We can then define a norm for a matrix $\vectg$ of the size $q$ multiplied by the number of lattice points with
 \begin{equation*}\label{eq:normepond}
\normepL{\vectg}=\big(\sum_{\vectx\in\mL}\varp{|\vectg(\vectx)|}{\MatP}{2}\big)^{\frac{1}{2}}, 
  \end{equation*}
  where $\vectg(\vectx)$ is the column of $\vectg$ associated with the node $\vectx$.
 The necessity to define such a non local norm is due to the non locality of the transport.
 The matrix $\MatP$ is chosen such that $^t\!\MatP\MatP$ is diagonal. Thus the transport is an isometry for the norm $\normepL{\cdot}$ for the periodic and bounceback boundary conditions \cite{Junk:2009:0}. Indeed, the hypothesis of ``quasi orthogonality'' ($^t\!\MatP\MatP$ diagonal) carrying on $\MatP$ cancels all the cross terms in the calculation of $\normepL{\vectf(.,t)}$. Then the isometry of the transport is obtained thanks to a simple change of variable: it uses the bijection between the nodes of the mesh and the nodes after the transport at a given velocity.
 
 Contrary to the transport, the collision is a local operator so that its study reduces to the vectorial norm $\normepv{\cdot}$. If the matrix $\MatRu$ is stable in a particular node $\vectx$ of the lattice $\lattice$ for $\normepv{\cdot}$, it is stable for $\normepL{\cdot}$. That's why, we use the local operator norm 
 $$\normep{\MatRu}=\underset{\normepv{\vectx}=1}{\sup}\normepv{\MatRu\vectx},$$ for the collision.

The matrix $\MatP$ is requested to diagonalize the collision: $\normep{\MatRu}$ is then easy to evaluate. These requirements define a notion of stability structure first evocated in \cite{Yong:2006:0}.

\begin{definition}[\cite{Rhein:2010:0}]\label{de:strstb}
A matrix $\MatN\in\espM(\R)$ has a pre-structure of stability if exists an invertible matrix $\MatP\in\espM(\R)$ and some vectors $\vects, \vectp\in\R^q$ so that 
\begin{gather*}
\MatP\MatN\MatP^{-1}=-{\rm diag}(\vects),\\
^t\!\MatP\MatP={\rm diag}(\vectp),
\end{gather*}
where ${\rm diag}(\vectp)$ is the diagonal matrix whose diagonal is constituted of the coefficients of $\vectp$.
The pre-structure of stability becomes a structure of stability if
\begin{equation}\label{eq:strstb}
\sk\in[0,2],\quad 0\leq k\leq q-1.
\end{equation}
\end{definition} 

Suppose that $\MatJu$ has a pre-structure of stability and consider the collision in norm $\normepv{\cdot}$. Knowing that the eigenvalues of $\MatRu$ are $1-\sk$ for $0\leq k\leq q-1$,
$$\normep{\MatRu}=\underset{\normd{\MatP\vectx}=1}{\sup}\normd{\MatP\MatRu\MatP^{-1}\MatP\vectx}=\normem{\MatP\MatRu\MatP^{-1}}=\max(|1-\vects|),$$
where $\normem{\cdot}$ is the operator norm associated with $\varp{|\cdot|}{2}{}$.
Thus the collision is stable for $\normep{\cdot}$ if the condition (\ref{eq:strstb}) is verified. Under these conditions, the scheme is stable in the norm $\normepL{\cdot}$ since the transport is isometric.

We present a theorem from \cite{Rhein:2010:0} giving a necessary and sufficient condition of existence of a pre-structure of stability. In the following, this theorem is the tool used to obtain some stability results for our relative velocity schemes.

\begin{theoreme}[\cite{Rhein:2010:0}]\label{th:cnspstr}
A matrix $\MatN\in\espM(\R)$ has a pre-structure of stability if and only if there exists a diagonal positive definite matrix $\MatLa\in\espM(\R)$ so that 
\begin{equation}\label{eq;syslinstrst}
\MatN\MatLa=\MatLa ^t\!\MatN.
\end{equation}
\end{theoreme}

This criteria gives a practical criteria of existence of a pre-structure of stability through the resolution of a linear system in the coefficients of $\MatLa$. The size of this linear system is $q(q-1)/2$ since the matrix $\MatN\MatLa-\MatLa ^t\!\MatN$ is antisymmetric.

The matrix $\MatP$ defining the weighted norm is explicitely derived in the proof of this theorem \cite{Rhein:2010:0}. We exhibit $\MatP$ through the proof of the sufficient condition of pre-structure. The identity (\ref{eq;syslinstrst}) is equivalent to the fact that $\MatLa^{-1}\MatN\MatLa$ is symmetric with respect to the scalar product
\begin{equation*}
\var{(\vectx,\vecty)}{\MatLa}=\sum_{i,j=0}^{q-1} \var{\matLa}{ij}\var{x}{i}\var{y}{j},\quad\vectx,\vecty\in\R^q.
\end{equation*}
This implies the existence of an orthonormal matrix $\MatQ$ in the sense of $\var{(\cdot,\cdot)}{\MatLa}$ diagonalizing $\MatLa^{-1}\MatN\MatLa$. Then, the matrix $\MatP=(\MatLa\MatQ)^{-1}$ diagonalizes the collision. The matrix $^t\!\MatP\MatP$ is diagonal because of the orthonormality of $\MatQ$ for the scalar product $\var{(\cdot,\cdot)}{\MatLa}$.

\subsection{Stability results for the twisted relative velocity $\ddqqn$ scheme}

In this section, we apply the stability criteria given by the theorem \ref{th:cnspstr} to obtain weighted $L^2$ results for the twisted relative velocity $\ddqqn$ scheme. Our motivation is to compare the cases $\utilde=\vectz$ and $\utilde=\vectV$. To do so, the scheme must have at least two different relaxation parameters because the BGK scheme does not depend on the relative velocity since (\ref{eq:mueq}) is verified. We begin by studying the MRT scheme ($\utilde=\vectz$) whose results are limited. We then obtain more results for the cascaded like scheme ($\utilde=\vectV$).

\subsubsection{The MRT scheme}

We focus on the MRT scheme corresponding to $\utilde=\vectz$. We limit to $\vectV=\vectz$ because when $\utilde=\vectz$ and $\vectV\neq\vectz$, there is no pre-structure of stability. The framework does not distinguish the non intrinsic and intrinsic case because for $\vectV=\vectz$, the equilibria (\ref{eq:eqd2q4tre}) and (\ref{eq:eqd2q4treis}) are identical.

\begin{proposition}[$L^2$ stability for the MRT scheme]\label{th:l2an}
Consider the twisted $\ddqqn$ scheme relative to $\utilde=\vectz$, of equilibrium $\vectmeqz=(\rho,0,0,0)$ associated with the relaxation parameters $\vects=(0,\sk[q],\sk[q],\sk[xy])\in\R^4$. The associated matrix $\MatJ(\vectz)$ has a pre-structure of stability. Moreover, if $0\leq\sk[q],\sk[xy]\leq2$, then $\MatJ(\vectz)$ has a structure of stability. The collision matrix $\MatR(\vectz)$ is then stable for $\normep{\cdot}$ and the scheme is stable in norm $\normepL{\cdot}$.
\end{proposition}

\begin{proof}
Since $\utilde=\vectV=\vectz$, the matrix $\MatJ(\vectz)$ reads 
\begin{equation*}
\MatJ(\vectz)=\MatM(\vectz)^{-1}\MatD(\MatE-\var{\MatI}{4})\MatM(\vectz),
\end{equation*}
with the diagonal matrix $\MatE={\rm diag}(1,0,0,0)$. 
The matrix $\MatD(\MatE-\var{\MatI}{4})$ is then diagonal as a product of two diagonal matrices. Since $\MatM(\vectz)^{-1}=~^t\!\MatM(\vectz)/4$, the matrix $\MatJ(\vectz)=~^t\!\MatM(\vectz)\MatD(\MatE-\var{\MatI}{4})\MatM(\vectz)/4$ is symmetric. So the identity is solution of the equation (\ref{eq;syslinstrst}) and $\MatJ(\vectz)$ has a pre-structure of stability. The spectrum of $\MatJ(\vectz)$ being composed by $0,-\sk[q],-\sk[q],-\sk[xy]$, the result on the structure of stability is obvious.
\end{proof}

\subsubsection{The scheme relative to the advection velocity (``cascaded like'' scheme)}

We now focus on the choice $\utilde=\vectV$ beginning by the non intrinsic case.

\begin{proposition}[$L^2$ stability for a non intrinsic relative velocity scheme]
Consider the twisted $\ddqqn$ scheme relative to $\utilde=\vectV=(\Vx,\Vy)\in\R^2$, of equilibrium $$\vectmeq(\vectV) = \rho(1,0,0,0),$$ associated with the relaxation parameters $\vects=(0,\sk[q],\sk[q],\sk[xy])\in\R^4$. The matrix $\MatJ(\vectV)$ has a pre-structure of stability if and only if $\normi<\lambda$. Moreover, if $0\leq\sk[q],\sk[xy]\leq2$, then $\MatJ(\vectV)$ has a structure of stability. The collision matrix $\MatR(\vectV)$ is then stable for $\normep{\cdot}$ and the scheme is stable in norm $\normepL{\cdot}$.
\end{proposition}

\begin{proof}
According to the theorem \ref{th:cnspstr}, the existence of a pre-structure of stability for $\MatN=\MatJ(\vectV)$ is equivalent to the existence of a diagonal positive definite matrix $\MatLa$ such that
\begin{equation*}\label{eq:Jdefpos}
\MatJ(\vectV)\MatLa=\MatLa~^t\!\MatJ(\vectV).
\end{equation*}
A solution of this linear system is given by 
{\small\begin{equation*}\label{eq:matD}
\MatLa=\begin{pmatrix}
(\lambda+\Vx)(\lambda+\Vy)&0&0&0\\
0&(\lambda-\Vx)(\lambda+\Vy)&0&0\\
0&0&(\lambda-\Vx)(\lambda-\Vy)&0\\
0&0&0&(\lambda+\Vx)(\lambda-\Vy)
\end{pmatrix}.
\end{equation*}}
This matrix is positive definite if and only if $\normi<\lambda$, that closes the first part of the proof. The eigenvalues of $\MatJ(\vectV)$ are $0,-\sk[q],-\sk[q],-\sk[xy]$ because the matrix $\MatM(\vectz)\MatJ(\vectV)\MatM(\vectz)^{-1}$ is given by
$$\begin{pmatrix}
0&0&0&0\\
\sk[q]\Vx&-\sk[q]&0&0\\
\sk[q]\Vy&0&-\sk[q]&0\\
(2\sk[q]-\sk[xy])\Vx\Vy&\Vy(\sk[xy]-\sk[q])&\Vx(\sk[xy]-\sk[q])&-\sk[xy]
\end{pmatrix}.$$
We deduce that $$\MatP\MatJ(\vectV)\MatP^{-1}={\rm diag}(0,-\sk[q],-\sk[q],-\sk[xy]),$$ and that $\MatJ(\vectV)$ has a structure of stability if $0\leq\sk[q],\sk[xy]\leq2$. 
The norm $\normep{\cdot}$ of the collision operator $\MatR(\vectV)$ is equal to $1$ and the scheme is stable in norm $\normepL{\cdot}$.
\end{proof}

These results extend the $\Li$ stability ones obtained in the proposition \ref{th:stlitrtv}. The proposition \ref{th:stlitrtv} gives the $\Li$ stability area $\normi\leq\lambda$ for $\sk[q]\leq\sk[xy]\leq \min(1,2\sk[q])$. The weighted $L^2$ notion generalizes it to $0\leq\sk[q],\sk[xy]\leq2$. 
This result was expected because the numerical experiments of the section \ref{sub:stabeqeq} put in evidence the independence of the area of $L^2$ stability towards the relaxation parameters. The fact that we can not obtain the same type of result for the MRT scheme as for $\utilde=\vectV$ is an other evidence of the good behaviour of this latter.

Let's also note that these stability conditions are defined by opened sets. The numerical experiments seem to show that the scheme is still $L^2$ stable on the closure of these sets. However, it is not possible to proove it with this notion because the matrix $\MatLa$ is null on this closure.
This proposition leads to a natural corollary for a single relaxation parameter (BGK) that has already been proved in \cite{Rhein:2010:0}. In the BGK case, the scheme does not depend on $\utilde$. In particular, MRT and cascaded like scheme are identical. That's why the following result is valid whatever the relative velocity $\utilde$. 
 
\begin{corollary}[The BGK non intrinsic case \cite{Rhein:2010:0}]\label{th:l2pbgk}
For $\vectV=(\Vx,\Vy)\in\R^2$, consider the twisted relative velocity $\ddqqn$ BGK scheme of equilibrium
\begin{equation*}
\vectmequ = \rho\big(1,\Vx-\utx,\Vy-\uty,(\Vx-\utx)(\Vy-\uty)\big),
\end{equation*} associated with the relaxation parameter $s\in\R$. The matrix $\MatJu$ has a pre-structure of stability if and only if $\normi<\lambda$. Moreover, if $0\leq s\leq2$, then $\MatJu$ has a structure of stability. The collision matrix $\MatR(\utilde)$ is then stable for $\normep{\cdot}$ and the scheme is stable in norm $\normepL{\cdot}$.
\end{corollary}

\begin{proposition}[$L^2$ stability for an intrinsic relative velocity scheme]\label{th:stth2relis}
Consider the twisted $\ddqqn$ scheme relative to $\utilde=\vectV=(\Vx,\Vy)\in\R^2$, of equilibrium $$\vectmeq(\vectV) = \rho(1,0,0,-\Vx\Vy),$$ associated with the relaxation parameters $\vects=(0,\sk[q],\sk[q],\sk[xy])\in\R^4$.  Suppose that $\Vx=0$ (resp. $\Vy=0$) then the associated matrix $\MatJ(\vectV)$ has a pre-structure of stability if and only if $|\Vy|<\lambda$ (resp. $|\Vx|<\lambda$). Moreover, if $0\leq\sk[q],\sk[xy]\leq2$, then $\MatJ(\vectV)$ has a structure of stability. The collision matrix $\MatR(\vectV)$ is then stable for $\normep{\cdot}$ and the scheme is stable in norm $\normepL{\cdot}$.
\end{proposition}

\begin{proof}
A non null solution of the equation (\ref{eq;syslinstrst}) exists if and only if one of the two components of $\vectV$ is null. Suppose that it is $\Vy$, then a solution of (\ref{eq;syslinstrst}) is given by
\begin{equation*}
\MatLa=\begin{pmatrix}
\lambda+\Vx&0&0&0\\
0&\lambda-\Vx&0&0\\
0&0&\lambda-\Vx&0\\
0&0&0&\lambda+\Vx
\end{pmatrix}.
\end{equation*}
This matrix is definite nonnegative if and only if $|\Vx|<\lambda$. The reasoning to get a structure of stability is identical to the one of the proposition \ref{th:l2an}.
\end{proof}

This proposition extends the $\Li$ stability results of the proposition \ref{th:stlitrtisV} for the directions $\Vx=0$ and $\Vy=0$. The $\Li$ notion provides areas decreasing as much as $\sk[q]$ and $\sk[xy]$ go far from each other. The $\normepL{\cdot}$ notion gives a constant optimal area in $\vectV$ for these directions while these parameters are bounded by $0$ and $2$. This phenomenon was observed numerically on the figures \ref{fig:vpddqq_6} to \ref{fig:vpddqq_10}.

We close the section by a proposition for the BGK case that is not a corollary of the proposition for the intrinsic cascaded like scheme.

\begin{proposition}[The BGK intrinsic case]
For $\vectV=(\Vx,\Vy)\in\R^2$, consider the twisted $\ddqqn$ scheme, of equilibrium
\begin{equation*}\label{eq:eqd2q4treis_bis}
\vectmequ = \rho\big(1,\Vx-\utx,\Vy-\uty,\utx\uty-\utx\Vy-\uty\Vx),
\end{equation*} BGK of relaxation parameter $s\in\R$. The associated matrix $\MatJu$ has a pre-structure of stability if and only if $\normu<\lambda$. Moreover, if $0\leq s\leq2$, then $\MatJu$ has a structure of stability. The collision matrix $\MatRu$ is then stable for $\normep{\cdot}$ and the scheme is stable in norm $\normepL{\cdot}$.
\end{proposition}

We note that this proposition generalizes the $\Li$ stability results in the BGK case (proposition \ref{th:stlitrtisbgk}). This weighted $L^2$ notion extends the set of $\vects$ corresponding to the stability area $\normu<\lambda$ to $0\leq s\leq2$.

\subsubsection{Interpretation of the results}

When we compare these propositions to the numerical results obtained in the section \ref{sub:stabeqeq}, we notice that this stability notion is only usable when the parameters $\vects$ and $\vectV$ are decoupled. For the scheme relative to $\utilde=\vectV$ with an intrinsic diffusion, we have only theoretical results for $\Vx=0$ ou $\Vy=0$, corresponding to a case where the areas in $\vectV$ are independent of $\vects$. Instead, when the area in $\vectV$ is a function of $\vects$, it is not possible to build a pre-structure of stability any more. For example, there is no pre-structure of stability when $\Vx$ and $\Vy$ are different from $0$ for the $\ddqqn$ scheme relative to $\utilde=\vectV$ with an intrinsic diffusion: for this scheme, the numerical results of the section \ref{sub:stabeqeq} exhibit a link between $\vectV$ and $\vects$. Similarly, the MRT scheme present stability areas in $\vectV$ depending on $\vects$ and it is impossible to build a pre-structure of stability for $\vectV\neq\vectz$.

The origin of this limitation seems to be the hypothesis of diagonalization of the collision. Indeed, the spectrum of this operator being equal to $0,-\sk[q],-\sk[q],-\sk[xy]$, the existence of a pre-structure of stability with such a diagonalization implies automatically the stability of the scheme for $0\leq\sk[q],\sk[xy]\leq2$. Then $\vectV$ and $\vects$ can't be linked because the former hypothesis is too requiring. The linear system of six equations and four unknowns for the $\ddqqn$ scheme is an evidence of the constraints imposed by the notion.
The existence of a pre-structure of stability requires its rank to be three that imposes $\vectV$ and $\vects$ to be decoupled here. For example, when $\utilde=\vectz$ and $\vectV\neq\vectz$, there is no solution positive definite to the equation (\ref{eq;syslinstrst}), because this rank is greater than four. That's why it is not possible to show similar results than in the case $\utilde=\vectV$ which requires the resolution of a rank three linear system. 

However, this notion gives some promising results as those exposed here and in \cite{Junk:2009:0,Rhein:2010:0}. It seems to describe well the limit of blow-up of a scheme. The main purpose now is to obtain theoretical stability results in more complex cases. Particularly, we want to complete this study for the $\ddqqn$ especially when $\utilde=\vectz$. To do so, we must be able to relax the constraint of diagonalization of the collision without penalizing the isometry of the transport.
Maybe, it would be possible to find some matrix $\MatP$ so that the norm $\normepL{\cdot}$ of the collision operator is computable without diagonalizing it.

\section*{Conclusion}

We have studied the stability of a four velocities relative velocity scheme with two different equilibria for a linear advection equation. The discussion is based on two notions of $\Li$ and weighted $L^2$ stability. It carries on the choice of the relative velocity equal to $\vectz$ (MRT scheme) or to the advection velocity (``cascaded like'' scheme). The main conclusions are the following: comparing MRT and relative velocity schemes, no scheme is ``better than the other'' in terms of $\Li$ stability; the stability areas generally just intersect. Instead, in terms of $L^2$ norm, the scheme relative to the advection velocity is more stable than the MRT scheme. This improvement is correlated to the cancellation of some dispersive terms of the equivalent equations for the scheme relative to the advection velocity. These terms stay for $\utilde=\vectz$ and are linked to instabilities at low diffusion. This result has been justified theoretically for the scheme relative to the advection velocity thanks to the weighted $L^2$ notion of stability. Further work needs to be done to obtain more precise theoretical results for the MRT scheme and for systems of conservation laws. The study also needs to be generalized to an arbitrary relative velocity $\utilde$, the present work being adapted to the two choices $\utilde=\vectz$ and $\utilde=\vectV$.

\begin{appendix}

\section{Link between the twisted $\ddqqn$ and the $\ddqqn$ schemes}\label{se:app3}

We proove that the $\Li$ and $L^2$ stability areas of the twisted $\ddqqn$ and $\ddqqn$ relative velocity schemes with a non intrinsic diffusion correspond through the composition of a rotation and a homothety. The same result is still true in the intrinsic case.

If we note $\vectv$ the velocity set of the $\ddqqn$ scheme, the twisted set of velocities is given by $\mR\vectv$ where 
$$\mR=\begin{pmatrix} 1&-1\\
1&1
\end{pmatrix},$$ 
is the composition of a rotation and a homothety of scale factor $\sqrt{2}$. This transformation allows to link the relaxation of both schemes.

\begin{lemme}[Relation between the relaxation operators]\label{th:relcor}
Let $\vectV\in\R^2$, $\utilde\in\R^2$, $(\sk[q],\sk[xy])\in\R^2$, the relative velocity $\ddqqn$ scheme associated with the equilibrium 
\begin{equation}\label{eq:eqd2q4ntre_4}
\vectmequ = \rho\big(1,\Vx-\utx,\Vy-\uty,(\Vx-\utx)^2-(\Vy-\uty)^2\big),
\end{equation} 
 for the relaxation parameters $(0,\sk[q],\sk[q],\sk[xy])$ and the relative velocity twisted $\ddqqn$ scheme associated with the equilibrium 
 \begin{multline*}\label{eq:eqrot}
\vectmequ{=}\rho\big(1,(\mR \vectV)^x{-}(\mR\utilde)^x,(\mR \vectV)^y{-}(\mR\utilde)^y,\\
((\mR \vectV)^x{-}(\mR\utilde)^x)((\mR \vectV)^y{-}(\mR\utilde)^y)\big),
\end{multline*} 
 and with the relaxation parameters  $(0,\sk[q],\sk[q],\sk[xy])$ have the same relaxation phase.
 \end{lemme}

\begin{proof}
We have defined a transformation $\mR$ sending the velocities of the $\ddqqn$ scheme on the velocities of the twisted $\ddqqn$ scheme. This naturally leads to the following application also denoted by $\mR$.

$$\begin{array}{ccccc}
\mR & : & \R[X,Y]& \to & \R[X,Y]\\
 & & P & \mapsto & \mR(P)\\
\end{array},$$
where $$\begin{array}{ccccc}
 \mR(P)& : & \espM[24](\R)& \to & \R\\
 & & \mV=\{\vj,0\leq j\leq 3\} & \mapsto & \dsp{\sum_{j=0}^{3}P(\mR\vj)\fj}\\
\end{array}.$$
Indeed the images of the moments $1,X,Y,X^2-Y^2$ by $\mR$ are
\begin{equation}\label{eq:trrel1}
\begin{split}
\mR(1)(\mV)&=1(\mV),\\ 
\mR(X)(\mV)&=\sum_{j=0}^{3}X(\mR\vj)\fj=\sum_{j=0}^{3}(\vjc{x}-\vjc{y})\fj=(X-Y)(\mV),\\
\mR(Y)(\mV)&=\sum_{j=0}^{3}Y(\mR\vj)\fj=\sum_{j=0}^{3}(\vjc{x}+\vjc{y})\fj=(X+Y)(\mV),
\end{split}
\end{equation}
and
\begin{equation}\label{eq:trrel2}
\mR(XY)(\mV)=\sum_{j=0}^{3}(\mR\vj)^x(\mR\vj)^y\fj=\sum_{j=0}^{3}((\vjc{x})^2-(\vjc{y})^2)\fj=(X^2-Y^2)(\mV).
\end{equation}
Providing these images, we can write the transformation of the relaxation of the $\ddqqn$ relative velocity scheme by $\mR$. We here choose to keep the polynomial notations to express the moments. The relaxation of the $\ddqqn$ reads

\begin{equation*}
\begin{split}
X^*(\mV)&=X(\mV)+\sk[q]((\Vx-\utx)\rho-X(\mV)),\\
Y^*(\mV)&=Y(\mV)+\sk[q]((\Vy-\uty)\rho-Y(\mV)),\\
(X^2{-}Y^2)^*(\mV)&=(X^2{-}Y^2)(\mV)+\sk[xy]\big(((\Vx{-}\utx)^2-(\Vy{-}\uty)^2)\rho-(X^2{-}Y^2)(\mV)\big).
\end{split}
\end{equation*}

Introducing the relations (\ref{eq:trrel1},\ref{eq:trrel2}), we obtain

\begin{equation*}
\begin{split}
\mR(X)^*(\mV)&=\mR(X)(\mV)+\sk[q](((\mR \vectV)^x-(\mR\utilde)^x)\rho-\mR(X)(\mV)),\\
\mR(Y)^*(\mV)&=\mR(Y)(\mV)+\sk[q](((\mR \vectV)^y-(\mR\utilde)^y)\rho-\mR(Y)(\mV)),\\
\mR(XY)^*(\mV)&=\mR(XY)(\mV)+\sk[xy]\big(((\mR \vectV)^x{-}(\mR\utilde)^x)((\mR V)^y{-}(\mR\utilde)^y)\rho{-}\mR(XY)(\mV)\big),
\end{split}
\end{equation*}

that is the relaxation of the twisted $\ddqqn$ relative velocity scheme. Note that this last step uses the fact that $X$ and $Y$ have the same relaxation parameter $\sk[q]$.
\end{proof}

\begin{proposition}[Relation between the stability areas]\label{th:trstab}
Let $(\sk[q],\sk[xy])\in\R^2$, note $\vars{\mS}{nt}\subset\R^2$ (resp. $\vars{\mS}{t}\subset\R^2$), the set of the velocities $\vectV\in\R^2$ such that the relative velocity $\ddqqn$ scheme (resp. twisted relative velocity $\ddqqn$ scheme) of equilibrium given by (\ref{eq:eqd2q4ntre_4}) (resp. (\ref{eq:eqd2q4},\ref{eq:eqd2q4tre})) and of relaxation parameters $\vects=(0,\sk[q],\sk[q],\sk[xy])$ is $L^{\infty}$ or $L^2$ stable. We have
\begin{equation}\label{eq:stsnt}
\vars{\mS}{t}=\mR\vars{\mS}{nt}.
\end{equation}
\end{proposition}

\begin{proof}
Note $\vars{\MatR}{t}(\vectV,\utilde,\vects)\in\espM[4](\R)$, the relaxation operator of the twisted relative velocity $\ddqqn$ scheme of equilibrium (\ref{eq:eqd2q4},\ref{eq:eqd2q4tre}): it verifies
$$\vectfe=\vars{\MatR}{t}(\vectV,\utilde,\vects)\vectf.$$
Note $\vars{\MatR}{nt}(\vectV,\utilde,\vects)\in\espM[4](\R)$ its analogue for the $\ddqqn$ scheme associated with the equilibrium (\ref{eq:eqd2q4ntre_4}).
The lemma \ref{th:relcor} means that 
\begin{equation}\label{eq:rntrt}
\vars{\MatR}{nt}(\vectV,\utilde,\vects)=\vars{\MatR}{t}(\mR\vectV,\mR\utilde,\vects).
\end{equation}

The transport does not influence the $\Li$ stability because it only exchanges the particle distributions. Thus $\vars{\mS}{nt}$ is the set of the velocities $\vectV\in\R^2$ so that the matrix $\vars{\MatR}{nt}(\vectV,\utilde,\vects)$ is nonnegative. According to the relation (\ref{eq:rntrt}), it corresponds to the velocities $\vectV\in\R^2$ so that $\mR\vectV$ belongs to the $\Li$ stability area of the twisted scheme: that is equivalent to the relation (\ref{eq:stsnt}) and the lemma is proven in the $\Li$ case.
 
Instead, the transport plays a role for the $L^2$ stability. It becomes local in the Fourier space: the transport matrices associated with the $\ddqqn$ scheme ($\vars{\MatA}{nt}$) and the twisted one ($\vars{\MatA}{t}$) are given by
 \begin{equation*}
 \begin{split}
 &\vars{\MatA}{nt}(\vectk)={\rm diag}\big(\exp(i\vj\cdot \vectk), 0\leq j\leq 3\big),\quad \vectk\in\R^2,\\
 &\vars{\MatA}{t}(\vectk)={\rm diag}\big(\exp(i\mR\vj\cdot \vectk), 0\leq j\leq 3\big),\quad \vectk\in\R^2.
\end{split}
\end{equation*}
The transport matrix $\vars{\MatA}{nt}(\vectk)$ is equal to $\vars{\MatA}{t}(\mR\vectk/2)$ because $\mR\vj\cdot \mR \vectk/2=\vj\cdot\vectk$ since $\mR$ is the composition of a rotation and a homothety of scale factor $\sqrt{2}$. The amplification matrices $\vars{\MatL}{nt}$ and $\vars{\MatL}{t}$ of both schemes are then defined by
 \begin{equation*}
 \begin{split}
 &\vars{\MatL}{t}(\vectk,\vectV,\utilde,\vects)=\vars{\MatA}{t}(\vectk)~\vars{\MatR}{t}(\vectV,\utilde,\vects),\\
 &\vars{\MatL}{nt}(\vectk,\vectV,\utilde,\vects)=\vars{\MatA}{t}(\mR \vectk/2)~\vars{\MatR}{t}(\mR\vectV,\mR\utilde,\vects).
  \end{split}
\end{equation*}
We then have the following identity 
\begin{equation}\label{eq:corampl}
\vars{\MatL}{nt}(\vectk,\vectV,\utilde,\vects)=\vars{\MatL}{t}(\mR \vectk/2,\mR\vectV,\mR\utilde,\vects).
\end{equation}

The $L^2$ stability needs to evaluate the maximum of the spectral radius $r$ of this matrix for all the wavenumbers $\vectk\in\R^2$. According to (\ref{eq:corampl}),
\begin{align*}
\underset{\vectk\in\R^2}{\max}~r(\vars{\MatL}{nt}(\vectk,\vectV,\utilde,\vects))&=\underset{\vectk\in\R^2}{\max}~r(\vars{\MatL}{t}(\mR \vectk/2,\mR\vectV,\mR\utilde,\vects))\\
&=\underset{\vectk\in\R^2}{\max}~r(\vars{\MatL}{t}(\vectk,\mR\vectV,\mR\utilde,\vects)),
\end{align*}
the last equality coming from a variable change in $\R^2$. As for the case of the $\Li$ stability, only the velocities $\vectV$ finally matter: the relation between the $L^2$ stability areas is then obtained in the same way.
\end{proof}

\section{Theoretical $\Li$ stability results for the $\ddqqn$ scheme}\label{se:app4}

\subsection{For a non intrinsic diffusion}

\begin{proposition}[$\Li$ stability areas for the MRT scheme]
Let $\vectV\in\R^2$, $(\sk[q],\sk[xy])\in\R^2$, consider the $\ddqqn$ of MRT scheme with the relaxation parameters $(0,\sk[q],\sk[q],\sk[xy])$, associated with the equilibrium $$\vectmeqz = \rho(1,\Vx,\Vy,(\Vx)^2-(\Vy)^2).$$ Note $\gamma=\sk[q]/\sk[xy]$,
\begin{itemize}
\item if $0<\sk[xy]\leq\min(\sk[q],2-\sk[q])$, the scheme is $\Li$ stable for all $\vectV$ so that 
\begin{gather*}
(\Vx\pm\lambda\gamma)^2-(\Vy)^2\geqslant \lambda^2(\gamma^2-1),\\
(\Vy\pm\lambda\gamma)^2-(\Vx)^2\geqslant \lambda^2(\gamma^2-1).
\end{gather*}
\item if $\sk[q]\leq\sk[xy]\leq2\sk[q]$ and $\sk[q]\leq1$, the scheme is $\Li$ stable for all $\vectV$ so that 
\begin{gather*}
(\Vx\pm\lambda\gamma)^2-(\Vy)^2\geqslant \lambda^2(\gamma-1)^2,\\
(\Vy\pm\lambda\gamma)^2-(\Vx)^2\geqslant \lambda^2(\gamma-1)^2.
\end{gather*}
\item if $2-\sk[q]\leq\sk[xy]\leq2(2-\sk[q])$ and $\sk[q]\geqslant1$, the scheme is $\Li$ stable for all $\vectV$ so that  
\begin{gather*}
(\Vx\pm\lambda\gamma)^2-(\Vy)^2\geqslant  \lambda^2\Big((\gamma+1)^2-\frac{4}{\sk[xy]}\Big),\\
(\Vy\pm\lambda\gamma)^2-(\Vx)^2\geqslant  \lambda^2\Big((\gamma+1)^2-\frac{4}{\sk[xy]}\Big).
\end{gather*}
\item if $\sk[xy]=0$ and $0<\sk[q]\leq2$, the scheme is $\Li$ stable for $\vectV=\vectz$.
\item if $\sk[xy]=\sk[q]=0$, the scheme is unconditionally $\Li$ stable.
\item For all other $\vects$, there is no $\vectV$ corresponding to a $\Li$ stable scheme.
\end{itemize}
\end{proposition}

\begin{proposition}
Let $\vectV\in\R^2$, $(\sk[q],\sk[xy])\in\R^2$, consider the $\ddqqn$ scheme relative to $\utilde=\vectV$ associated with the relaxation parameters $(0,\sk[q],\sk[q],\sk[xy])$, and the equilibrium $$\vectmeq(\vectV) = \rho(1,0,0,0).$$ Let us note $\gamma=\sk[xy]/(2\sk[q]-\sk[xy])$, 
\begin{itemize}
\item if $\sk[q]\leq\sk[xy]\leq \min(1,2\sk[q])$, the scheme is $\Li$ stable on
\begin{equation}\label{eq:carreid}
\normu\leq\lambda.
\end{equation}
\item if $\sk[xy]<2\sk[q]$, $\max(\sk[q],1)\leq\sk[xy]\leq2(2-\sk[q])$, the scheme is $\Li$ stable on the intersection of (\ref{eq:carreid}) with
\begin{subequations}
\begin{gather}
(\Vx\pm\lambda\gamma)^2-(\Vy)^2\geqslant\frac{4\lambda^2}{(2\sk[q]-\sk[xy])^2}(\sk[xy]-\sk[q](2-\sk[q])),\label{eq:ega3x}\\
(\Vy\pm\lambda\gamma)^2-(\Vx)^2\geqslant\frac{4\lambda^2}{(2\sk[q]-\sk[xy])^2}(\sk[xy]-\sk[q](2-\sk[q])).\label{eq:ega4x}
\end{gather} 
\end{subequations}
\item if $0\leq\sk[xy]\leq\min(\sk[q],\sk[q](2-\sk[q]))$, the scheme is $\Li$ stable on
\begin{equation*}
\normu\leq\lambda\gamma.
\end{equation*}
\item if $\sk[q](2-\sk[q])\leq\sk[xy]\leq\min(\sk[q],2(2-\sk[q]))$, the scheme is $\Li$ stable on (\ref{eq:ega3x},\ref{eq:ega4x}).
\item if $\sk[xy]=2\sk[q]$ and $1<\sk[xy]\leq2$, the scheme is $\Li$ stable on the intersection of (\ref{eq:carreid}) and 
\begin{equation*}
\normi\leq\lambda\Big(\frac{2-\sk[xy]}{\sk[xy]}\Big).
\end{equation*}
\item if $\sk[xy]=\sk[q]=0$, the scheme is unconditionally $\Li$ stable.
\item For all other $\vects$, no $\vectV$ corresponds to a $\Li$ stable scheme.
\end{itemize}
\end{proposition}

\subsection{For an intrinsic diffusion}

\begin{proposition}
Let $\vectV\in\R^2$, $(\sk[q],\sk[xy])\in\R^2$, consider the $\ddqqn$ MRT scheme associated with the relaxation parameters $(0,\sk[q],\sk[q],\sk[xy])$, and the equilibrium $$\vectmeqz = \rho(1,\Vx,\Vy,0).$$ Note $\gamma=\sk[xy]/\sk[q]$, 
\begin{itemize}

\item if $0\leq\sk[xy]\leq\min(\sk[q],2-\sk[q])$, the scheme is $\Li$ stable for all $\vectV$ such that 
\begin{equation*}
\normi\leq\frac{\lambda\gamma}{2}.
\end{equation*}
\item if $\sk[q]\leq\min(1,\sk[xy])$ and $\sk[xy]\leq2\sk[q]$ the scheme is $\Li$ stable for all $\vectV$ such that 
\begin{equation*}
\normi\leq\frac{\lambda(2-\gamma)}{2}.
\end{equation*}
\item if $\sk[q]\geqslant\max(1,2-\sk[xy])$ and $\sk[xy]\leq2(2-\sk[q])$ the scheme is $\Li$ stable for all $\vectV$ such that 
\begin{equation*}
\normi\leq\frac{\lambda\big(\frac{4}{\sk[q]}-2-\gamma\big)}{2}.
\end{equation*}
\item if $\sk[xy]=\sk[q]=0$, the scheme is unconditionally $\Li$ stable.
\item For all other $\vects$, no $\vectV$ corresponds to a $\Li$ stable scheme.
\end{itemize}
\end{proposition}

\begin{proposition}
Let $\vectV\in\R^2$, $(\sk[q],\sk[xy])\in\R^2$, consider the $\ddqqn$ scheme relative to $\utilde=\vectV$ associated with the relaxation parameters $(0,\sk[q],\sk[q],\sk[xy])$, and the equilibrium $$\vectmeq(\vectV) = \rho(1,0,0,(\Vy)^2-(\Vx)^2).$$ Note $\gu[1]=(2\sk[q]-\sk[xy])/(\sk[q]-\sk[xy])$ and $\gu[2]=\sk[xy]/(\sk[q]-\sk[xy])$,
\begin{itemize}
\item if $\sk[q]<\sk[xy]\leq\min(2\sk[q],2(2-\sk[q]))$, the scheme is stable for all $\vectV$ such that 
\begin{gather*}
(\Vx\pm\frac{\lambda\gu[1]}{2})^2-(\Vy)^2\leq \frac{\lambda^2\gu[1]\gu[2]}{4},\\
(\Vx\pm\lambda\frac{\gu[1]+\gu[2]}{4})^2-(\Vy\mp\frac{\lambda}{2})^2\leq\frac{\lambda^2\varp{\gamma}{2}{2}}{4},\\
(\Vx\pm\lambda\frac{\gu[1]+\gu[2]}{4})^2-(\Vy\pm\frac{\lambda}{2})^2\leq\frac{\lambda^2\varp{\gamma}{2}{2}}{4},\\
(\Vx\pm\frac{\lambda\gu[2]}{2})^2-(\Vy)^2\leq\frac{\lambda^2}{4(\sk[q]-\sk[xy])^2}(4\varp{s}{q}{2}-2\sk[q]\sk[xy]-\varp{s}{xy}{2}+8(\sk[xy]-\sk[q])), 
\end{gather*} 
plus the analogous inequalities exchanging $\Vx$ and $\Vy$.
\item if $\sk[xy]\leq\min(\sk[q],2(2-\sk[q]))$, the scheme is stable for all $\vectV$ such that
\begin{gather*}
(\Vx\pm\frac{\lambda\gu[1]}{2})^2-(\Vy)^2\geqslant \frac{\lambda^2\gu[1]\gu[2]}{4},\\
(\Vx\pm\lambda\frac{\gu[1]+\gu[2]}{4})^2-(\Vy\mp\frac{\lambda}{2})^2\geqslant\frac{\lambda^2\varp{\gamma}{2}{2}}{4},\\
(\Vx\pm\lambda\frac{\gu[1]+\gu[2]}{4})^2-(\Vy\pm\frac{\lambda}{2})^2\geqslant\frac{\lambda^2\varp{\gamma}{2}{2}}{4},\\
(\Vx\pm\frac{\lambda\gu[2]}{2})^2-(\Vy)^2\geqslant\frac{\lambda^2}{4(\sk[q]-\sk[xy])^2}(4\varp{s}{q}{2}-2\sk[q]\sk[xy]-\varp{s}{xy}{2}+8(\sk[xy]-\sk[q])),
\end{gather*} 
plus the analogous inequalities exchanging $\Vx$ and $\Vy$.
\item if $\sk[xy]=\sk[q]\leq 1$, the scheme is $\Li$ stable for all $\vectV$ such that 
\begin{equation*}
\normi\leq\frac{\lambda}{2}.
\end{equation*}
\item if $1\leq \sk[xy]=\sk[q]\leq 4/3$, the scheme is $\Li$ stable for all $\vectV$ such that 
\begin{equation*}
\normi\leq\frac{\lambda}{2}\big(\frac{4}{s}-3\big).
\end{equation*}

\item For all other $\vects$, no $\vectV$ corresponds to a $\Li$ stable scheme.
\end{itemize}
\end{proposition}

\section{Proof of the proposition \ref{th:stlitrtisV}}

The positivity of the matrix shifting from $\vectf$ to $\vectfe$ is equivalent to
\begin{subequations}
\begin{gather}
\lambda(2\sk[q]-\sk[xy])(\lambda\pm(\Vx+(-1)^i\Vy))+2(-1)^i(\sk[q]-\sk[xy])\Vx\Vy\geqslant0,\label{eq:twis1}\\
\lambda^2\sk[xy]\pm\lambda(\sk[xy]\Vx+(-1)^i(2\sk[q]-\sk[xy])\Vy)+2(-1)^i(\sk[q]-\sk[xy])\Vx\Vy\geqslant0,\label{eq:twis3}\\
\lambda^2\sk[xy]\pm\lambda(\sk[xy]\Vy+(-1)^i(2\sk[q]-\sk[xy])\Vx)+2(-1)^i(\sk[q]-\sk[xy])\Vx\Vy\geqslant0,\label{eq:twis4}\\
4\lambda^2-((2\sk[q]+\sk[xy])\lambda^2\pm\sk[xy]\lambda(\Vx+(-1)^i\Vy)-2(-1)^i(\sk[q]-\sk[xy])\Vx\Vy)\geqslant0.\label{eq:twis7}
\end{gather} 
\end{subequations}
for $i=0,1$.
We begin by the case $\sk[q]=\sk[xy]$ of a BGK scheme corresponding to the proposition \ref{th:stlitrtisbgk}. The four inequations (\ref{eq:twis1}) to (\ref{eq:twis7}) become
\begin{gather*}
\lambda\sk[q](\lambda\pm(\Vx+(-1)^i\Vy))\geqslant0,\\
\sk[q](\lambda\pm(\Vx+(-1)^i\Vy))\geqslant0,\\
\sk[q](\lambda\pm(\Vy+(-1)^i\Vx))\geqslant0,\\
4\lambda^2-\sk[q](3\lambda^2\pm\lambda(\Vx+(-1)^i\Vy)\geqslant0.
\end{gather*}
The case $\sk[q]\leq0$ is impossible because no velocity $\vectV$ would verify these four inequalities requiring $\Vx\pm\Vy$ to be simultaneously greater than $\lambda$ and inferior to $-\lambda$. Then $\sk[q]\geqslant0$ that leads to $$\normu\leq\min(\lambda,\lambda(\frac{4}{\sk[q]}-3)),$$
that gives exactly the proposition \ref{th:stlitrtisbgk} according to the choice of $\sk[q]$.
 If $\sk[xy]\neq\sk[q]$, we can write the inequalities (\ref{eq:twis1}) to (\ref{eq:twis7}) thanks to $\gu[1]$ and $\gu[2]$. If $\sk[xy]<\sk[q]$, we obtain 
\begin{subequations}
\begin{gather}
\lambda\frac{\gu}{2}(\lambda\pm(\Vx+(-1)^i\Vy))+(-1)^i\Vx\Vy\geqslant0,\quad i=0,1,\label{eq:twis21}\\
\lambda^2\frac{\gu[2]}{2}\pm\frac{\lambda}{2}(\gu[2]\Vx+(-1)^i\gu\Vy)+(-1)^i\Vx\Vy\geqslant0,\quad i=0,1,\label{eq:twis23}\\
\lambda^2\frac{\gu[2]}{2}\pm\frac{\lambda}{2}(\gu[2]\Vy+(-1)^i\gu\Vx)+(-1)^i\Vx\Vy\geqslant0,\quad i=0,1,\label{eq:twis24}\\
\frac{4-(2\sk[q]+\sk[xy])}{2(\sk[q]-\sk[xy])}\lambda^2\pm\frac{\gu[2]}{2}\lambda(\Vx+(-1)^i\Vy)+(-1)^i\Vx\Vy\geqslant0,\quad i=0,1.\label{eq:twis27}
\end{gather} 
\end{subequations}

We do the following change of variable $\VVx=\Vx+\Vy$, $\VVy=\Vx-\Vy$. Noticing that $\gu[1]-\gu[2]=2$, the inequalities (\ref{eq:twis21}) to (\ref{eq:twis27}) are equivalent to 
\begin{subequations}
\begin{gather*}
2\lambda\gu(\lambda\pm\VVx)+(\VVx)^2-(\VVy)^2\geqslant0,\label{eq:twis31}\\
2\lambda^2\gu[2]\pm2\lambda(\frac{\gu[1]+\gu[2]}{2}\VVx-\VVy)+(\VVx)^2-(\VVy)^2\geqslant0,\label{eq:twis32}\\
2\lambda^2\gu[2]\pm2\lambda(\frac{\gu[1]+\gu[2]}{2}\VVx+\VVy)+(\VVx)^2-(\VVy)^2\geqslant0,\label{eq:twis33}\\
\frac{8-(4\sk[q]+2\sk[xy])}{\sk[q]-\sk[xy]}\lambda^2\pm2\gu[2]\lambda\VVx+(\VVx)^2-(\VVy)^2\geqslant0,\label{eq:twis37}
\end{gather*} 
\end{subequations}
for $i=0$ in addition to the ones obtained exchanging $\VVx$ and $\VVy$ corresponding to $i=1$. We center these inequalities to get the identities from (\ref{eq:hypisgeq1}) to (\ref{eq:hypisgeq4}). The case $\sk[q]\leq\sk[xy]$ is obtained by changing the sense of the inequalities. It is characterized by the inequalities from (\ref{eq:hypisleq1}) to (\ref{eq:hypisleq4}).

 We begin by treating the case $\sk[q]<\sk[xy]$ corresponding to the inequalities from (\ref{eq:hypisleq1}) to (\ref{eq:hypisleq4}). The case $\sk[q]>\sk[xy]$ will be considered further. The reasoning first eliminates the couples $(\sk[q],\sk[xy])$ leading to no stable velocity and then concentrates on stable areas. First, we assume that $\sk[xy]>\max(0,2\sk[q])$ and show that there is no stable velocity $\vectV$. Since $\sk[xy]>\sk[q]$, we have $\gu[1]\gu[2]\leq0$ and the equation (\ref{eq:hypisleq1}) reads 
$$(\VVy)^2-(\VVx\pm\lambda\gu[1])^2\geqslant -\lambda^2\gu[1]\gu[2],$$
plus the two ones obtained exchanging $\VVx$ and $\VVy$. The intersection of these four hyperbolic areas is empty.

Secondly we eliminate the case $\sk[q]<\sk[xy]<0$. In this situation, $\gu[1]\gu[2]$ is nonnegative ($\gu[1]$ and $\gu[2]$ are both nonnegative) and (\ref{eq:hypisleq1}) has a solution if and only if the abscissa of the right pole of the hyperbole centered in $-\lambda\gu[1]$ is nonnegative. This reads
$$-\lambda\gu[1]+\lambda(\gu[1]\gu[2])^{\frac{1}{2}}\geqslant0\iff\gu[1]\leq\gu[2]\iff\sk[xy]\leq\sk[q].$$
The last inequality contradicts our hypothesis and there is no stable velocity when $\sk[q]<\sk[xy]<0$.

It remains to treat the case $\sk[q]<\sk[xy]\leq2\sk[q]$ for which $\gu[1]\gu[2]\geqslant0$.  The equation (\ref{eq:hypisleq1}) has some solutions $\vectV$ if and only if the abscissa of the right pole $P$ of the hyperbole centered in $\lambda\gu[1]$ is nonnegative
\begin{equation}\label{eq:absP}
\lambda\gu[1]+\lambda(\gu[1]\gu[2])^{\frac{1}{2}}\geqslant0.
\end{equation}

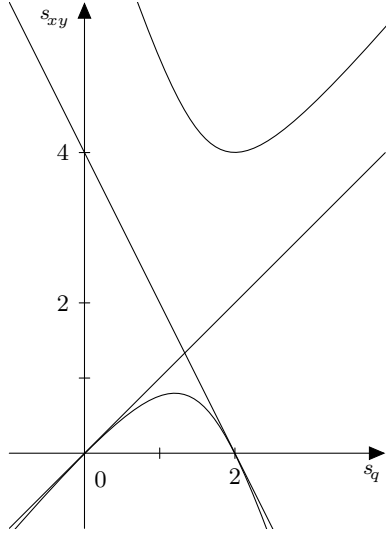
\begin{figure}
\begin{center}
\definecolor{ffqqqq}{rgb}{1,0,0}
\definecolor{qqqqff}{rgb}{0,0,1}
\begin{tikzpicture}[line cap=round,line join=round,>=triangle 45,x=1.0cm,y=1.0cm]
\draw[->,color=black] (-1,0) -- (4,0);
\foreach \x in {,2}
\draw[shift={(\x,0)},color=black] (0pt,2pt) -- (0pt,-2pt) node[below] {\footnotesize $\x$};
\draw[->,color=black] (0,-1) -- (0,6);
\foreach \y in {,2,4}
\draw[shift={(0,\y)},color=black] (2pt,0pt) -- (-2pt,0pt) node[left] {\footnotesize $\y$};
\draw[color=black] (0pt,-10pt) node[right] {\footnotesize $0$};
\clip(-1,-1) rectangle (4,6);
\draw[smooth,samples=100,domain=-1.0:4.0] plot(\x,{(\x)});
\draw [samples=50,domain=-0.99:0.99,rotate around={79.1:(1.6,2.4)},xshift=1.6cm,yshift=2.4cm,color=black] plot ({1.64*(1+\x*\x)/(1-\x*\x)},{0.87*2*\x/(1-\x*\x)});
\draw [samples=50,domain=-0.99:0.99,rotate around={79.1:(1.6,2.4)},xshift=1.6cm,yshift=2.4cm,color=black] plot ({1.64*(-1-\x*\x)/(1-\x*\x)},{0.87*(-2)*\x/(1-\x*\x)});
\draw[color=black, smooth,samples=100,domain=-1.0:4.0] plot(\x,{4-2*(\x)});
\begin{scriptsize}
\draw (3.6,-0.05) node[anchor=north west] {$\sk[q]$};
\draw (-0.7,6.) node[anchor=north west] {$\sk[xy]$};
\end{scriptsize}
\end{tikzpicture}
\end{center}
\caption[]{Hyperbole $4\varp{s}{q}{2}-2\sk[q]\sk[xy]-\varp{s}{xy}{2}+8(\sk[xy]-\sk[q])=0,$ and straight lines $\sk[xy]=2-\sk[q]$ and $\sk[xy]=\sk[q]$.} 
\label{fig:hyps_dhtwis}
\end{figure}

\begin{figure}
\begin{center}
\definecolor{qqffff}{rgb}{0,1,1}
\definecolor{qqttff}{rgb}{0,0.2,1}
\definecolor{qqccqq}{rgb}{0,0.8,0}
\definecolor{uququq}{rgb}{0.25,0.25,0.25}
\definecolor{ffqqtt}{rgb}{1,0,0.2}
\definecolor{qqqqff}{rgb}{0,0,1}
\definecolor{xdxdff}{rgb}{0.49,0.49,1}
\begin{tikzpicture}[scale=1.,line cap=round,line join=round,>=triangle 45,x=1.0cm,y=1.0cm]
\draw[->,color=black] (-3,0) -- (3,0);
\draw[shift={(2,0)},color=black] (0pt,2pt) -- (0pt,-2pt) node[below] {\footnotesize $-\lambda\gu[1]$};
\draw[shift={(-2,0)},color=black] (0pt,2pt) -- (0pt,-2pt) node[below] {\footnotesize $\lambda\gu[1]$};
\draw[->,color=black] (0,-3) -- (0,3);
\draw[shift={(0,2)},color=black] (2pt,0pt) -- (-2pt,0pt) node[left] {\footnotesize $-\lambda\gu[1]$};
\draw[shift={(0,-2)},color=black] (2pt,0pt) -- (-2pt,0pt) node[left] {\footnotesize $\lambda\gu[1]$};
\draw[color=black] (0pt,-10pt) node[right] {\footnotesize $0$};
\clip(-3,-3) rectangle (3,3);
\draw [domain=-6:6] plot(\x,{(--4--2*\x)/2});
\draw [domain=-6:6] plot(\x,{(-4--2*\x)/-2});
\draw [domain=-6:6] plot(\x,{(-4--2*\x)/2});
\draw [domain=-6:6] plot(\x,{(-4-2*\x)/2});
\draw [samples=50,domain=-0.99:0.99,rotate around={0:(-2,0)},xshift=-2cm,yshift=0cm,color=black] plot ({2.64*(1+\x*\x)/(1-\x*\x)},{2.6*2*\x/(1-\x*\x)});
\draw [samples=50,domain=-0.99:0.99,rotate around={0:(-2,0)},xshift=-2cm,yshift=0cm,color=black] plot ({2.64*(-1-\x*\x)/(1-\x*\x)},{2.6*(-2)*\x/(1-\x*\x)});
\draw [samples=50,domain=-0.99:0.99,rotate around={180:(2,0)},xshift=2cm,yshift=0cm,color=black] plot ({2.64*(1+\x*\x)/(1-\x*\x)},{2.6*2*\x/(1-\x*\x)});
\draw [samples=50,domain=-0.99:0.99,rotate around={180:(2,0)},xshift=2cm,yshift=0cm,color=black] plot ({2.64*(-1-\x*\x)/(1-\x*\x)},{2.6*(-2)*\x/(1-\x*\x)});
\draw [samples=50,domain=-0.99:0.99,rotate around={-90:(0,2)},xshift=0cm,yshift=2cm,color=black] plot ({2.64*(1+\x*\x)/(1-\x*\x)},{2.6*2*\x/(1-\x*\x)});
\draw [samples=50,domain=-0.99:0.99,rotate around={-90:(0,2)},xshift=0cm,yshift=2cm,color=black] plot ({2.64*(-1-\x*\x)/(1-\x*\x)},{2.6*(-2)*\x/(1-\x*\x)});
\draw [samples=50,domain=-0.99:0.99,rotate around={90:(0,-2)},xshift=0cm,yshift=-2cm,color=black] plot ({2.64*(1+\x*\x)/(1-\x*\x)},{2.6*2*\x/(1-\x*\x)});
\draw [samples=50,domain=-0.99:0.99,rotate around={90:(0,-2)},xshift=0cm,yshift=-2cm,color=black] plot ({2.64*(-1-\x*\x)/(1-\x*\x)},{2.6*(-2)*\x/(1-\x*\x)});
\begin{scriptsize}
\fill [color=uququq] (0.63,0) circle (1.5pt);
\draw[color=uququq] (0.84,0.38) node {$P$};
\fill [color=uququq] (0.74,0.74) circle (1.5pt);
\draw[color=uququq] (0.95,1.12) node {$A$};
\fill [color=uququq] (0.74,-0.74) circle (1.5pt);
\draw[color=uququq] (0.92,-0.35) node {$B$};
\fill [color=uququq] (-0.74,-0.74) circle (1.5pt);
\draw[color=uququq] (-0.5,-0.35) node {$C$};
\fill [color=uququq] (-0.74,0.74) circle (1.5pt);
\draw[color=uququq] (-0.5,1.12) node {$D$};
\draw (2.5,-0.05) node[anchor=north west] {$\VVx$};
\draw (-0.65,3.) node[anchor=north west] {$\VVy$};
\end{scriptsize}
\end{tikzpicture}
\end{center}
\caption{Layout of the bounds of the $\Li$ stability areas corresponding to (\ref{eq:hypisleq1}) and (\ref{eq:hypisleq4}) for the twisted $\ddqqn$ scheme relative to $\utilde=\vectV$ with an intrinsic diffusion.} 
\label{fig:zonehyp_rectwis1}
\end{figure}

The area satisfying these inequations corresponds to the points $A,B,C,D$ on the figure \ref{fig:zonehyp_rectwis1}. It increases as the abscissa of $P$ increases.
The inequality (\ref{eq:absP}) is equivalent to $\gu[1]\leq0$, that is verified since $\sk[q]<\sk[xy]\leq2\sk[q]$.
 
For the equation (\ref{eq:hypisleq4}), there are two cases. Hence $$4\varp{s}{q}{2}-2\sk[q]\sk[xy]-\varp{s}{xy}{2}+8(\sk[xy]-\sk[q])\leq0,$$ then the reasoning is the same as the one made in the case $\sk[xy]>\max(0,\sk[q])$ to show that there is no stable velocity. Hence $$4\varp{s}{q}{2}-2\sk[q]\sk[xy]-\varp{s}{xy}{2}+8(\sk[xy]-\sk[q])\geqslant0,$$ then the abscissa of the right pole of the hyperbole centered in $\lambda\gu[2]$ must be nonnegative:
$$\lambda\gu[2]+\frac{\lambda}{|\sk[q]-\sk[xy]|}(4\varp{s}{q}{2}-2\sk[q]\sk[xy]-\varp{s}{xy}{2}+8(\sk[xy]-\sk[q]))^{\frac{1}{2}}\geqslant0.$$
The parameter $\gu[2]$ being negative, this is equivalent to $$\frac{1}{(\sk[q]-\sk[xy])^2}(4\varp{s}{q}{2}-2\sk[q]\sk[xy]-\varp{s}{xy}{2}+8(\sk[xy]-\sk[q]))\geqslant\varp{\gamma}{2}{2},$$
and so to $$(\sk[q]-\sk[xy])(4\sk[q]+2\sk[xy]-8)\geqslant0\iff\sk[xy]\leq2(2-\sk[q]),$$
since $\sk[q]\leq\sk[xy]$.

We now summarize the results we have obtained for $\sk[q]<\sk[xy]$. We have proven that if  $\sk[q]<\sk[xy]$, the inequalities (\ref{eq:hypisleq1}) and (\ref{eq:hypisleq4}) have solutions for if $\sk[xy]\leqslant\min(2\sk[q],2(2-\sk[q]))$. It remains to prove that (\ref{eq:hypisleq2}) and (\ref{eq:hypisleq3}) have also solutions in this case.

The conditions (\ref{eq:hypisleq2}) and (\ref{eq:hypisleq3}) lead to a non empty area in $\vectV$ since the intersection $I$ (figure \ref{fig:zonehyp_rectwis2}) between the hyperbole given by (\ref{eq:hypisleq2}) centered in $(-\lambda(\gu[1]+\gu[2])/2,\lambda)$ and the straight line $\VVy=0$ has a negative abscissa. The same condition on the intersection of the hyperbole centered in $(-\lambda(\gu[1]+\gu[2])/2,-\lambda)$ and $\VVy=0$ is requested for the equation (\ref{eq:hypisleq3}). We focus on the hyperbole centered in $(-\lambda(\gu[1]+\gu[2])/2,\lambda)$ associated with (\ref{eq:hypisleq2}), because the one associated with (\ref{eq:hypisleq3}) leads to the same result. The intersections of the hyperbole of equation
$$(\VVx+\lambda\frac{\gu[1]+\gu[2]}{2})^2-(\VVy-\lambda)^2=\lambda^2\varp{\gamma}{2}{2},$$
with $\VVy=0$ verify $$(\VVx+\lambda\frac{\gu[1]+\gu[2]}{2})^2-\lambda^2=\lambda^2\varp{\gamma}{2}{2}.$$
Their abscissas are then given by
$$-\lambda\frac{\gu[1]+\gu[2]}{2}\pm\lambda(\varp{\gamma}{2}{2}+1)^{\frac{1}{2}}.$$

\begin{figure}
\begin{center}
\definecolor{qqffff}{rgb}{0,1,1}
\definecolor{qqttff}{rgb}{0,0.2,1}
\definecolor{ttzzqq}{rgb}{0.2,0.6,0}
\definecolor{uququq}{rgb}{0.25,0.25,0.25}
\definecolor{ffqqtt}{rgb}{1,0,0.2}
\definecolor{qqqqff}{rgb}{0,0,1}
\definecolor{xdxdff}{rgb}{0.49,0.49,1}
\definecolor{qqffff}{rgb}{0,1,1}
\definecolor{qqttff}{rgb}{0,0.2,1}
\definecolor{ffqqtt}{rgb}{1,0,0.2}
\definecolor{ttccqq}{rgb}{0.2,0.8,0}
\definecolor{qqqqff}{rgb}{0,0,1}
\definecolor{uququq}{rgb}{0.25,0.25,0.25}
\definecolor{xdxdff}{rgb}{0.49,0.49,1}
\begin{tikzpicture}[scale=0.5,line cap=round,line join=round,>=triangle 45,x=1.0cm,y=1.0cm]
\draw[->,color=black] (-8,0) -- (8,0);
\draw[->,color=black] (0,-8) -- (0,8);
\draw[shift={(2,0)},color=black] (0pt,2pt) -- (0pt,-2pt) node[below] {\footnotesize $-\lambda\frac{\gu[1]+\gu[2]}{2}$};
\draw[->,color=black] (0,-6) -- (0,6);
\draw[shift={(0,2)},color=black] (2pt,0pt) -- (-2pt,0pt) node[left] {\footnotesize $\lambda$};
\draw[shift={(0,-2)},color=black] (2pt,0pt) -- (-2pt,0pt) node[left] {\footnotesize $-\lambda$};
\draw[color=black] (0pt,-10pt) node[right] {\footnotesize $0$};
\clip(-8,-8) rectangle (8,8);
\draw [domain=-8:8] plot(\x,{(--16-4*\x)/4});
\draw [domain=-8:8] plot(\x,{(-0--2*\x)/2});
\draw [domain=-8:8] plot(\x,{(-16-4*\x)/4});
\draw [samples=50,domain=-0.99:0.99,rotate around={0:(2,2)},xshift=2cm,yshift=2cm,color=black] plot ({2.83*(1+\x*\x)/(1-\x*\x)},{2.83*2*\x/(1-\x*\x)});
\draw [samples=50,domain=-0.99:0.99,rotate around={0:(2,2)},xshift=2cm,yshift=2cm,color=black] plot ({2.83*(-1-\x*\x)/(1-\x*\x)},{2.83*(-2)*\x/(1-\x*\x)});
\draw [samples=50,domain=-0.99:0.99,rotate around={-180:(-2,-2)},xshift=-2cm,yshift=-2cm,color=black] plot ({2.83*(1+\x*\x)/(1-\x*\x)},{2.83*2*\x/(1-\x*\x)});
\draw [samples=50,domain=-0.99:0.99,rotate around={-180:(-2,-2)},xshift=-2cm,yshift=-2cm,color=black] plot ({2.83*(-1-\x*\x)/(1-\x*\x)},{2.83*(-2)*\x/(1-\x*\x)});
\draw [samples=50,domain=-0.99:0.99,rotate around={90:(-2,2)},xshift=-2cm,yshift=2cm,color=black] plot ({2.83*(1+\x*\x)/(1-\x*\x)},{2.83*2*\x/(1-\x*\x)});
\draw [samples=50,domain=-0.99:0.99,rotate around={90:(-2,2)},xshift=-2cm,yshift=2cm,color=black] plot ({2.83*(-1-\x*\x)/(1-\x*\x)},{2.83*(-2)*\x/(1-\x*\x)});
\draw [samples=50,domain=-0.99:0.99,rotate around={-90:(2,-2)},xshift=2cm,yshift=-2cm,color=black] plot ({2.83*(1+\x*\x)/(1-\x*\x)},{2.83*2*\x/(1-\x*\x)});
\draw [samples=50,domain=-0.99:0.99,rotate around={-90:(2,-2)},xshift=2cm,yshift=-2cm,color=black] plot ({2.83*(-1-\x*\x)/(1-\x*\x)},{2.83*(-2)*\x/(1-\x*\x)});
\draw [domain=-8:8] plot(\x,{(-16-4*\x)/-4});
\draw [domain=-8:8] plot(\x,{(-0--2*\x)/-2});
\draw [domain=-8:8] plot(\x,{(--16-4*\x)/-4});
\begin{scriptsize}
\fill [color=black] (-1.5,0.) circle (1.5pt);
\draw[color=black] (-1.5,0.3) node {$I$};
\fill [color=black] (-0.83,2) circle (1.5pt);
\draw[color=black] (-0.61,2.47) node {$A$};
\fill [color=black] (2,0.83) circle (1.5pt);
\draw[color=black] (2.,1.31) node {$B$};
\fill [color=black] (0.83,-2) circle (1.5pt);
\draw[color=black] (1.07,-1.67) node {$C$};
\fill [color=black] (-2,-0.83) circle (1.5pt);
\draw[color=black] (-2.15,-0.5) node {$D$};
\draw (7.,-0.05) node[anchor=north west] {$\VVx$};
\draw (-1.2,8.) node[anchor=north west] {$\VVy$};
\end{scriptsize}
\end{tikzpicture}
\end{center}
\caption{Layout of the bounds of the $\Li$ stability areas corresponding to (\ref{eq:hypisleq2}) and (\ref{eq:hypisleq3}) for the twisted $\ddqqn$ scheme relative to $\utilde=\vectV$ with an intrinsic diffusion.} 
\label{fig:zonehyp_rectwis2}
\end{figure}

We check when the abscissa of the point $I$ of the figure \ref{fig:zonehyp_rectwis2} is negative
  $$-\lambda\frac{\gu[1]+\gu[2]}{2}-\lambda(\varp{\gamma}{2}{2}+1)^{\frac{1}{2}}\leq0\iff\Big(\frac{\gu[1]+\gu[2]}{2}\Big)^2\leq \varp{\gamma}{2}{2}+1\iff \varp{s}{q}{2}\leq \varp{s}{xy}{2}+(\sk[q]-\sk[xy])^2.$$
This is equivalent to $\sk[xy](\sk[xy]-\sk[q])\geqslant0$, that is valid because $\sk[xy]\leq\sk[q]$.

We now focus on the case $\sk[xy]<\sk[q]$ corresponding to the inequalities from (\ref{eq:hypisgeq1}) to (\ref{eq:hypisgeq4}). Let's begin by eliminating the case $2\sk[q]\leq\sk[xy]<\sk[q]$. In this case (\ref{eq:hypisgeq1}) has solutions if and only if both poles of the hyperbole centered in $-\lambda\gu[1]$ ($\gu[1]\leq0$) have nonnegative abscissas that reads
$$-\lambda(\gu[1]+(\gu[1]\gu[2])^{\frac{1}{2}})\geqslant0.$$
Because of the negativity of $\gu[1]$, this is equivalent to $\sk[q]\leq\sk[xy]$ that contradicts our hypothesis.

Second, we eliminate the area $2(2-\sk[q])\leq\sk[xy]<\sk[q]$. This is done showing that there is no solution to (\ref{eq:hypisgeq4}). In this area, we have $$4\varp{s}{q}{2}-2\sk[q]\sk[xy]-\varp{s}{xy}{2}+8(\sk[xy]-\sk[q])\geqslant0,$$
(figure \ref{fig:hyps_dhtwis}) so that (\ref{eq:hypisgeq4}) has a solution if and only if the abscissas of both poles of the hyperbole centered in $\lambda|\gu[2]|$ are nonnegative. This reads $$\lambda|\gu[2]|-\frac{\lambda}{|\sk[q]-\sk[xy]|}(4\varp{s}{q}{2}-2\sk[q]\sk[xy]-\varp{s}{xy}{2}+8(\sk[xy]-\sk[q]))^{\frac{1}{2}}\geqslant0,$$ 
that is equivalent to $\sk[xy]\leq2(2-\sk[q])$. This contradicts our hypothesis.

It remains to eliminate the area $\sk[xy]<\min(0,\sk[q],2(2-\sk[q]))$ corresponding to $\gu[1]\geqslant0$ and $\gu[2]\leq0$. If $\sk[q]\leq0$, then $\frac{\gu[1]+\gu[2]}{2}\leq0$ and the inequality (\ref{eq:hypisgeq2}) is non empty if the abscissas of the intersections of the hyperbole centered in $(-\lambda(\gu[1]+\gu[2])/2,\lambda)$ with $\VVy=0$ are both nonnegative. This is equivalent to $$-\lambda\frac{\gu[1]+\gu[2]}{2}-\lambda(\varp{\gamma}{2}{2}+1)^{\frac{1}{2}}\geqslant0\iff\Big(\frac{\gu[1]+\gu[2]}{2}\Big)^2\geqslant \varp{\gamma}{2}{2}+1\iff \varp{s}{q}{2}\geqslant \varp{s}{xy}{2}+(\sk[q]-\sk[xy])^2.$$
This is equivalent to $\sk[xy](\sk[xy]-\sk[q])\leq0$, that is verified for $\sk[xy]\geqslant\sk[q]$ since $\sk[xy]<0$. This enters in contradiction with our hypothesis. If $\sk[q]>0$, the reasoning is symmetric with respect to the ordinates axis ($\VVx=0$) and leads to the same contradiction.

To close the proof, it remains to guarantee the existence of at least one velocity stable in the triangle defined by $0\leq\sk[xy]\leq\min(\sk[q],2(2-\sk[q]))$.
 The equation (\ref{eq:hypisgeq1}) has a non empty stability area if and only if the two poles of the hyperbole centered in $\lambda\gu[1]$ have a nonnegative abscissa: this case, illustrated by the figure \ref{fig:zonehyp_dhtw}, is equivalent to $$\lambda(\gu[1]-(\gu[1]\gu[2])^{\frac{1}{2}})\geqslant0.$$ 
This inequality is equivalent to $\gu[1]\geqslant0$, that is true for $\sk[xy]\leq\sk[q]$.

For the equation (\ref{eq:hypisgeq4}), we distinguish two cases: the first is 
$$4\varp{s}{q}{2}-2\sk[q]\sk[xy]-\varp{s}{xy}{2}+8(\sk[xy]-\sk[q])\geqslant0.$$
There is a non empty stability area if and only if the two poles of the hyperbole centered in $\lambda\gu[2]$ have nonnegative abscissas (represented on the figure \ref{fig:zonehyp_dhtw}): this reads
$$\lambda\gu[2]-\frac{\lambda}{|\sk[q]-\sk[xy]|}(4\varp{s}{q}{2}-2\sk[q]\sk[xy]-\varp{s}{xy}{2}+8(\sk[xy]-\sk[q]))^{\frac{1}{2}}\geqslant0.$$ 
This inequality is equivalent to $\sk[xy]\leq2(2-\sk[q])$. Otherwise, $$4\varp{s}{q}{2}-2\sk[q]\sk[xy]-\varp{s}{xy}{2}+8(\sk[xy]-\sk[q])\leq0,$$
and the equation (\ref{eq:hypisgeq4}) reads 
\begin{equation*}
(\VVy)^2-(\VVx\pm\lambda\gu[2])^2\leq-\frac{\lambda^2}{(\sk[q]-\sk[xy])^2}(4\varp{s}{q}{2}-2\sk[q]\sk[xy]-\varp{s}{xy}{2}+8(\sk[xy]-\sk[q])).
\end{equation*}

This inequation has always solutions: the area of stability in $\vectV$ is analogous to the one represented on the figure \ref{fig:zonehyp_rectw}.

The inequations (\ref{eq:hypisgeq2}) and (\ref{eq:hypisgeq3}), illustrated on the figure \ref{fig:zonehyp_rectwis3}, have solutions since the two intersections of the hyperbole centered in $(\lambda(\gu[1]+\gu[2])/2,-\lambda)$ with $\VVy=0$ have nonnegative abscissas
$$\lambda\frac{\gu[1]+\gu[2]}{2}-\lambda(\varp{\gamma}{2}{2}+1)^{\frac{1}{2}}\geqslant0\iff\Big(\frac{\gu[1]+\gu[2]}{2}\Big)^2\geqslant \varp{\gamma}{2}{2}+1\iff \varp{s}{q}{2}\geqslant \varp{s}{xy}{2}+(\sk[q]-\sk[xy])^2.$$
This is equivalent to $\sk[xy](\sk[xy]-\sk[q])\leq0$, that is verified for $\sk[xy]\leq\sk[q]$.
\begin{figure}
\begin{center}
\definecolor{qqffff}{rgb}{0,1,1}
\definecolor{ffqqtt}{rgb}{1,0,0.2}
\definecolor{ttccqq}{rgb}{0.2,0.8,0}
\definecolor{uququq}{rgb}{0.25,0.25,0.25}
\definecolor{xdxdff}{rgb}{0.49,0.49,1}
\definecolor{qqqqff}{rgb}{0,0,1}
\begin{tikzpicture}[scale=1.,line cap=round,line join=round,>=triangle 45,x=1.0cm,y=1.0cm]
\draw[->,color=black] (-4,0) -- (4,0);
\draw[->,color=black] (0,-4) -- (0,4);
\draw[shift={(2,0)},color=black] (0pt,2pt) -- (0pt,-2pt) node[below] {\footnotesize $\lambda\frac{\gu[1]+\gu[2]}{2}$};
\draw[shift={(-2,0)},color=black] (0pt,2pt) -- (0pt,-2pt) node[below] {\footnotesize $-\lambda\frac{\gu[1]+\gu[2]}{2}$};
\draw[->,color=black] (0,-4) -- (0,4);
\draw[shift={(0,1)},color=black] (2pt,0pt) -- (-2pt,0pt) node[left] {\footnotesize $\lambda$};
\draw[shift={(0,-1)},color=black] (2pt,0pt) -- (-2pt,0pt) node[left] {\footnotesize $-\lambda$};
\draw[color=black] (0pt,-10pt) node[right] {\footnotesize $0$};
\clip(-4,-4) rectangle (4,4);
\draw [domain=-6:6] plot(\x,{(-6--2*\x)/-2});
\draw [domain=-6:6] plot(\x,{(--1-1*\x)/-1});
\draw [domain=-6:6] plot(\x,{(-6-2*\x)/2});
\draw [domain=-6:6] plot(\x,{(--1--1*\x)/1});
\draw [samples=50,domain=-0.99:0.99,rotate around={0:(2,-1)},xshift=2cm,yshift=-1cm,color=black] plot ({0.71*(1+\x*\x)/(1-\x*\x)},{0.71*2*\x/(1-\x*\x)});
\draw [samples=50,domain=-0.99:0.99,rotate around={0:(2,-1)},xshift=2cm,yshift=-1cm,color=black] plot ({0.71*(-1-\x*\x)/(1-\x*\x)},{0.71*(-2)*\x/(1-\x*\x)});
\draw [samples=50,domain=-0.99:0.99,rotate around={-180:(-2,1)},xshift=-2cm,yshift=1cm,color=black] plot ({0.71*(1+\x*\x)/(1-\x*\x)},{0.71*2*\x/(1-\x*\x)});
\draw [samples=50,domain=-0.99:0.99,rotate around={-180:(-2,1)},xshift=-2cm,yshift=1cm,color=black] plot ({0.71*(-1-\x*\x)/(1-\x*\x)},{0.71*(-2)*\x/(1-\x*\x)});
\draw [samples=50,domain=-0.99:0.99,rotate around={90:(1,2)},xshift=1cm,yshift=2cm,color=black] plot ({0.71*(1+\x*\x)/(1-\x*\x)},{0.71*2*\x/(1-\x*\x)});
\draw [samples=50,domain=-0.99:0.99,rotate around={90:(1,2)},xshift=1cm,yshift=2cm,color=black] plot ({0.71*(-1-\x*\x)/(1-\x*\x)},{0.71*(-2)*\x/(1-\x*\x)});
\draw [domain=-6:6] plot(\x,{(--2-2*\x)/2});
\draw [domain=-6:6] plot(\x,{(--6.37--2.1*\x)/2.13});
\draw [samples=50,domain=-0.99:0.99,rotate around={-90:(-1,-2)},xshift=-1cm,yshift=-2cm,color=black] plot ({0.71*(1+\x*\x)/(1-\x*\x)},{0.71*2*\x/(1-\x*\x)});
\draw [samples=50,domain=-0.99:0.99,rotate around={-90:(-1,-2)},xshift=-1cm,yshift=-2cm,color=black] plot ({0.71*(-1-\x*\x)/(1-\x*\x)},{0.71*(-2)*\x/(1-\x*\x)});
\draw [domain=-6:6] plot(\x,{(--3-1*\x)/-1});
\draw [domain=-6:6] plot(\x,{(--2--2*\x)/-2});
\begin{scriptsize}
\fill [color=uququq] (0.05,0.81) circle (1.5pt);
\draw[color=uququq] (0.06,1.09) node {$A$};
\fill [color=uququq] (0.81,-0.05) circle (1.5pt);
\draw[color=uququq] (0.9,0.34) node {$B$};
\fill [color=uququq] (-0.81,0.05) circle (1.5pt);
\draw[color=uququq] (-0.86,0.29) node {$C$};
\fill [color=uququq] (-0.05,-0.81) circle (1.5pt);
\draw[color=uququq] (0.25,-0.72) node {$D$};
\draw (3.5,-0.05) node[anchor=north west] {$\VVx$};
\draw (-0.6,4.) node[anchor=north west] {$\VVy$};
\end{scriptsize}
\end{tikzpicture}
\end{center}
\caption{Layout of the bounds of the $\Li$ stability areas corresponding to (\ref{eq:hypisgeq2}) and (\ref{eq:hypisgeq3}) for the twisted scheme relative to $\utilde=\vectV$ with an intrinsic diffusion.} 
\label{fig:zonehyp_rectwis3}
\end{figure}
\end{appendix}

\bibliographystyle{plain}
\bibliography{Bibliographie}

\end{document}